\documentclass[12pt]{amsart}
\usepackage{amssymb,enumerate,mathrsfs,amsmath,amsthm}
\usepackage{url,titletoc}

\usepackage[
colorlinks,linkcolor=red,citecolor=blue,urlcolor=blue,hypertexnames=true]{hyperref}

\usepackage{cmap}

\usepackage{bbm}

\usepackage{graphicx}

\usepackage{mathdots}

\usepackage{epstopdf}

\usepackage{caption}
\usepackage[labelfont=footnotesize,font=footnotesize]{subcaption}

\usepackage{etoolbox}

\makeatletter
\def\l@subsection{\@tocline{2}{0pt}{2.5pc}{5pc}{}}
\def\l@subsubsection{\@tocline{2}{0pt}{5pc}{5pc}{}}
\def\l@chapter{\@tocline{0}{8pt plus1pt}{0pt}{}{\bfseries}}
\makeatother

\def\bmat{\begin{pmatrix}}
\def\emat{\end{pmatrix}}

\def\hpen{{\hat\Lambda}}

\newcommand{\transpose}{{\operatorname{T}}}

\newcommand{\wt}[1]{\widetilde{#1}}

\def\id{\operatorname{id}}

\def\dxd{{d \times d}}

\def\MAT{M}

\newcommand{\dom}{\operatorname{dom}}

\newcommand{\sa}{{\operatorname{sa}}}

\def\tr{\operatorname{tr}}

\newcommand{\df}[1]{{\bf{#1}}{\index{#1}}}

\newcommand{\domA}[1]{{\dom_\cA({#1}) }   }

\def\pen{{\Lambda_\MAT}}

\def\Lam{{\Lambda}}

\def\ot{\otimes}

\newcommand{\ev}{\operatorname{ev}}

\newcommand{\Sys}{\operatorname{Sys}}

\def\IA{{1_\cA}}

\def\kac{ \bmat I_k \ot \IA\\ 0 \emat}    
\def\kar{\bmat I_k \ot \IA & 0 \emat }    

\def\ra{{\rho_\cA}}

\newcommand{\hM}{\hat{M}}
\newcommand{\hXi}{\hat{\Xi}}

\newcommand{\cA}{{\mathcal A}}

\def\tA{\widetilde{A}}
\def\hA{{\hat{A}}}
\def\tA{{\widetilde A}}

\def\cAg{{\cA^g}}

\def\cB{ {\mathcal B} }

\def\hB{{\hat B}}

\def\cC{ {\mathcal C} }
\def\bbC{{\mathbb C}}
\def\tC{{\widetilde C}}
\newcommand{\CC}{{\mathbb C}}
\def\hC{{\hat C}}

\def\cE{ {\mathcal E} }

\def\bH{{\mathbb H} }

\def\tJ{\widetilde{J}}
\def\hJ{{\hat J}}

\def\bbK{\mathbb{K}}

\def\tM{\widetilde{M}}
\def\bbM{{\mathbb M}}

\def\NN{{\mathbb N}}

\newcommand{\cQ}{{\mathcal Q}}

\def\bbR{{\mathbb R}}
\def\RR{{\mathbb R}}

\newcommand{\fr}{\mathfrak{r}}
\newcommand{\ur}{\underline{r}}

\def\cS{{\mathcal S} }

\def\cW{{\mathcal W}}
\newcommand\WOR{{\mathcal W}}   

\newcommand{\bX}{\mathbb{X}}

\makeatletter
\def\moverlay{\mathpalette\mov@rlay}
\def\mov@rlay#1#2{\leavevmode\vtop{%
\baselineskip\z@skip \lineskiplimit-\maxdimen
\ialign{\hfil$#1##$\hfil\cr#2\crcr}}}
\makeatother

\def\plangle{\moverlay{(\cr<}}
\def\prangle{\moverlay{)\cr>}}

\def\Klx{{\CC\plangle x_1,\dots,x_g\prangle}}

\newcommand{\rf}{\mathfrak}
\newcommand{\fgg}{\chi}
\newcommand{\dr}{\mathbbm}

\newcommand{\br}{\mathbbm{r}}

\usepackage{mathtools}

\makeatletter
\DeclareRobustCommand\widecheck[1]{{\mathpalette\@widecheck{#1}}}
\def\@widecheck#1#2{%
    \setbox\z@\hbox{\m@th$#1#2$}%
    \setbox\tw@\hbox{\m@th$#1%
       \widehat{%
          \vrule\@width\z@\@height\ht\z@
          \vrule\@height\z@\@width\wd\z@}$}%
    \dp\tw@-\ht\z@
    \@tempdima\ht\z@ \advance\@tempdima2\ht\tw@ \divide\@tempdima\thr@@
    \setbox\tw@\hbox{%
       \raise\@tempdima\hbox{\scalebox{1}[-1]{\lower\@tempdima\box
\tw@}}}%
    {\ooalign{\box\tw@ \cr \box\z@}}}
\makeatother

\renewcommand{\check}{\widecheck}
\renewcommand{\hat}{\widehat}

\newtheorem{thm}{Theorem}[section]

\newtheorem{lem}[thm]{Lemma}
\newtheorem{prop}[thm]{Proposition}
\newtheorem{prob}[thm]{Problem}
\theoremstyle{definition}
\newtheorem{definition}[thm]{Definition}

\theoremstyle{remark}
\newtheorem{rem}[thm]{Remark}
\AtEndEnvironment{rem}{~\hfill\qedsymbol}

\newtheorem{algorithm}[thm]{Algorithm}

\numberwithin{equation}{section}

\newtheorem{example}[thm]{Example}

\newcounter{Inc}

\makeindex

\title[Applications of Realizations to Free Probability]{Applications of Realizations (aka Linearizations) to Free Probability}

\author{J. William Helton}
\address{Mathematics Department \\ University of California at San Diego 92093}
\email{helton@math.ucsd.edu}

\author{Tobias Mai}
\address{Saarland University, Faculty of Mathematics, D-66123 Saarbruecken, Germany}
\email{mai@math.uni-sb.de}

\author{Roland Speicher}
\address{Saarland University, Faculty of Mathematics, D-66123 Saarbruecken, Germany}
\email{speicher@math.uni-sb.de}


\thanks{J.W.Helton's work was partly supported by
the U.S. National Science Foundation
grants DMS 1201498 and DMS-1500835.
\\
The work of T. Mai and R. Speicher was supported by funds from the DFG (SP-419-8/1) and the ERC Advanced Grant {\it Non-commutative Distributions in Free Probability}.
\\
The authors are grateful for insights of 
Greg Anderson and Zonglin Jiang 
into the theory of P.M. Cohn, 
of Adrian Ionna into von Neumann algebras.
Also discussions with Victor Vinnikov on 
noncommutative rational functions were valuable.
Thanks also to Fran\c{c}ois Lofficial for making us aware of some typos.
Finally, we gratefully acknowledge the work of two anonymous referees; their valuable comments have led to a significant improvement of the exposition.}

\begin{document}

\maketitle

\centerline{\today}



\noindent{\bf Abstract.}

We show how the combination of new ``linearization'' ideas in free probability theory with the powerful ``realization'' machinery -- developed over the last 50 years in fields including systems engineering and automata theory -- allows solving the problem of determining the eigenvalue distribution (or even the Brown measure, in the non-selfadjoint case) of noncommutative rational functions of random matrices when their size tends to infinity. Along the way we extend evaluations of noncommutative rational expressions from matrices to stably finite algebras, e.g. type II$_1$ von Neumann algebras, with a precise control of the domains of the rational expressions.

The paper provides sufficient background information, with the intention that it should be accessible both to functional analysts and to algebraists.

\newpage

\section{Introduction}

In free probability theory \cite{VDN,HP,NS, MS} the calculation of the distribution of polynomials $p(X_1,\dots,X_g)$ of $g$ free random variables $X_1,\dots,X_g$ has made essential progress recently \cite{BMS2013}. This relies crucially on the so-called linearization trick: a polynomial like 
$p(X_1,X_2)=X_1X_2+X_2X_1$
can also be written in the form
\begin{equation}\label{eq:p}
X_1X_2+X_2X_1 = -uQ^{-1}v =
 - \begin{pmatrix} 0 & 0 & 0 & 1 \end{pmatrix} \begin{pmatrix} 0 & X_1 & X_2 & -1\\ X_1 & 0 & -1 & 0\\ X_2 & -1 & 0 & 0\\ -1 & 0 & 0 & 0\end{pmatrix}^{-1} \begin{pmatrix} 0 \\ 0 \\ 0 \\ 1 \end{pmatrix},
\end{equation}
resulting in the fact that the \emph{linearization of $p$}\index{linearization}, i.e., the matrix of polynomials
\begin{equation}\label{eq:phat}
\hat p=
\begin{pmatrix}
0&u\\v&Q
\end{pmatrix}
=
\begin{pmatrix}
0 & 0 & 0 & 0 & 1\\
0 & 0 & X_1 & X_2 & -1\\
0 & X_1 & 0 & -1 & 0\\
0 & X_2 & -1 & 0 & 0\\
1 & -1 & 0 & 0 & 0
\end{pmatrix},
\end{equation}
contains all relevant information about our polynomial $p$. This $\hat p$ is now a matrix-valued polynomial in the variables, but has the decisive advantage that all entries have degree at most 1. This means that we can write $\hat p$ as the sum of a matrix in $X_1$ and a matrix in $X_2$. For dealing with sums of (operator-valued) free variables, however, one has a quite well-developed analytic machinery in free probability theory. The application of this machinery gives then results as in Fig. \ref{fig:anticommutator-ex2} in Sect. \ref{subsec:examples}.
One should note that those results about the distribution of polynomials in free random variables yield also results about the asymptotic eigenvalue distribution of such polynomials in random matrices, by the fundamental observation of Voiculescu that in many cases random matrices become asymptotically free when the size of the matrices goes to infinity.

The crucial ingredient in the above, which allowed transforming an unmanageable polynomial into something linear, and thus into something in the realm of powerful free probability techniques, was the mentioned linearization trick. This trick is purely algebraic and does not rely on the freeness of the variables involved. In the free probability community this linearization was introduced and used with powerful applications in the work of Haagerup and Thorbj\o rnsen \cite{HT,HST}, and then it was refined by Anderson \cite{And} to a version as needed for the problem presented above. However, as it turned out this trick in its algebraic form is not new at all, but has been highly developed for about 50 years in a variety of areas ranging from system engineering to automata to ring theory; indeed it is often called a ``noncommutative system realization''. 

The aim of the present article is to bring these different communities together;
in particular, to introduce on one side the control community to a powerful application of
realizations and to provide on the other side the free probability and operator algebra community with some background on the extensive family of results in this context.

Apart from trying to give a survey on the problems and results from the various communities and show how they are related 
there are at least two new contributions arising from this endeavor:
\begin{itemize}
\item
In the realization context one is actually not solely interested in polynomials;
the realization \eqref{eq:p} given in the example above, though having quite non-trivial implications in free probability, is a trivial one from the point of view of descriptor realizations. The real gist are (noncommutative) rational functions. This makes it clear that the linearization philosophy is not restricted to polynomials but works equally well for rational functions. This results in the possibility to calculate also the distribution of rational functions in free random variables. Such results are presented here for the first time.
\item
From a noncommutative analysis point of view rational functions are noncommutative functions whose domain of definition are usually matrices of all sizes. However, the natural setting for our noncommutative random variables are operators on an infinite dimensional Hilbert space, but still equipped with a trace. This means we want to plug in as arguments in our rational functions not matrices but elements from stably finite operator algebras (stably finite $C^*$-algebras or finite von Neumann algebras). This gives a quite new perspective on noncommutative functions and raises a couple of canonical new questions -- some of which we will answer here.
\end{itemize}

Whereas it is clear what a noncommutative polynomial $p(x_1,\dots,x_g)$ in $g$ non-commuting variables $x_1,\dots,x_g$ is and how one can apply such a polynomial as a function $p(X_1,\dots,X_g)$ to any tuple $(X_1,\dots,X_g)$ of elements from any algebra, the corresponding questions for noncommutative rational functions are surprisingly much more subtle. 
(In the following it shall be understood that we talk about ``noncommutative'' rational functions, and we will usually drop the adjective.) Intuitively, it is clear that a rational function $r(x_1,\dots,x_g)$ should be anything we can get from our variables $x_1,\dots,x_g$ by algebraic manipulations if we also allow taking inverses. Of course, one should not take the inverse of $0$. But it might not be obvious if a given expression is zero or not;
rational expressions can be very complicated, in particular, they can involve nested inversions.
Whereas it is obvious 
that we cannot invert $1-x_1 x_1^{-1}$ within rational functions, it is
probably not clear to the reader whether 
$$r(x_1,x_2,x_3)=x_2^{-1}+x_2^{-1}(x_3^{-1}x_1^{-1}-x_2^{-1})^{-1} x_2^{-1}-(x_2-x_3x_1)^{-1}$$
is identically zero or not, and thus whether $r(x_1,x_2,x_3)^{-1}$ is a valid rational function or not. One of the main problems in defining and dealing with rational functions is to find a way of deciding whether a rational expression represents zero or, equivalently, whether two rational expressions represent the same function. The intuitive idea is of course that they are the same if we can transform one into the other by algebraic manipulations. However, this is not very handy for concrete questions and also a bit clumsy for developing the general theory. There exist actually quite a variety of different approaches to deal with this, resulting in different definitions of rational functions. Of course, in the end all approaches are equivalent, but it is sometimes quite tedious to arrive at this conclusion. Nevertheless, in the end a rational function $\mathfrak{r}$ can be identified with an equivalence class of rational expressions $r$, where the equivalence $r_1\sim r_2$ means that the rational expression $r_1$ can
be transformed into the rational expression $r_2$ by algebraic manipulations.
(This is, in a non-obvious way, the same as requiring that $r_1$ and $r_2$ give the same result when we are plugging in matrices of arbitrary size, such that both expressions are well-defined.)

An important (and quite large) subset of rational functions is given by those which are ``regular'' at one fixed point in $\CC$ (which we will usually choose to be the point 0). These regular rational functions can be identified with a subclass of power series and their theory becomes quite straightforward to handle. In particular, the question whether two expressions are the same becomes quite easy, as it comes down to comparing the coefficients in the power series expansion.
This results in the existence of a very powerful machinery for some aspects of the theory which is not available in the general case.
In this regular setting the main ideas and results can in principle be traced back to the work of Sch\"utzenberger \cite{S61}. In the setting without this restriction the whole theory gets a more abstract algebraic flavor and relies on the basic work of Cohn on the free field \cite{Co71}.

Whichever approach one takes, at one point one arrives at the fundamental insight that it is advantageous to represent rational functions in terms of matrices of linear polynomials of our variables. In the approach we take in this paper the possibility of such a representation is a basic theorem -- this is the content of the linearization trick in the free probability context and goes under the name of ``realization'' in the control community.  
In the more general algebraic approach according to Cohn this matrix realization is more or less the definition of the skew field of rational functions.
In any case, this matrix realization of a rational function is of the form
$\br=C\Lambda^{-1}B$,
where, for some $n\in\NN$, $\Lambda$ is an $n\times n$ matrix, $C$ is an $1\times n$ matrix and $B$ is an $n\times 1$ matrix, all with polynomials as entries. Actually, in $\Lambda$ those polynomials can always
be taken of degree less or equal 1, whereas the entries of $C$ and $B$ can be chosen as constants; such a realization is called ``linear''. We will only consider linear realizations. 
Of course, also in this realization it will be crucial to decide whether we have a meaningful expression or not; i.e., we must be able to decide whether the inverse of our matrix $\Lambda$ of polynomials makes sense.  
By a non-obvious result of Cohn, $\Lambda$ is invertible over the rational functions if and only if it is full, which means that
it cannot be decomposed as a product of smaller strictly rectangular matrices. It is important that in this factorization only polynomial (and not rational) entries are required; hence for deciding whether something is a valid expression within rational functions it suffices to answer a question within the ring of matrices over polynomial functions.

The collection of all noncommutative rational functions gives a skew field, which is also called ``free field'' and usually denoted by $\Klx$. In our context it will become important to treat elements from $\Klx$ not just as abstract algebraic objects, but we will consider them as functions, which we want to evaluate on tuples of operators $X_1,\dots,X_g$ from some fixed algebra $\cA$. We  emphasize that for our purposes  we have to consider infinite-dimensional $C^*$-algebras or von Neumann algebras; hence, many of the established results around noncommutative rational functions -- which are mainly about applying those functions to matrices (of arbitrary sizes) -- do not apply directly to our setting and one of our contributions here is to extend the theory to this more general setting.

Applying a rational function to elements in some algebra raises a number of issues.
First of all, we have to check whether being zero in $\Klx$ implies also being zero when applied to our operators. That this is not an trivial issue is shown by the following example. In $\CC\plangle x_1,x_2\prangle$ we clearly have
$x_1(x_2x_1)^{-1}x_2-1=0$.
However, if we take operators $X_1$ and $X_2$ in some $\cA$ with
$X_2X_1=1$, but $X_1X_2\not=1$,
then the expression $X_1(X_2X_1)^{-1} X_2-1$  makes sense in $\cA$, but we
have there
$$X_1(X_2X_1)^{-1} X_2-1=X_1X_2-1\not= 0.$$
This shows that a rational expression which represents 0 in the free field does not necessarily evaluate to zero when we apply it to operators like above; to put it another way, there exist operators $X_1,\dots,X_g$ and rational expressions $r_1$ and $r_2$, representing the same rational function $\fr$ in the free field, such that $r_1(X_1,\dots,X_g)$ and $r_2(X_1,\dots,X_g)$ both make sense, but do not agree; hence in this case there is no meaningful definition for 
$\fr(X_1,\dots,X_g)$. 
Luckily, it turns out that the above counter example is essentially the only obstruction for this.
One of the basic insights of Cohn (see \cite{Co06}) is that if we consider an algebra which is ``stably finite'' (sometimes also called ``weakly finite'') -- this is by definition an algebra where in the matrices over the algebra any left-inverse is also a right-inverse -- then relations in the free field will also be relations in the algebra
(provided they make sense). We will elaborate on this fact, in our setting, in Sect.~\ref{sect:2-4}.
Stably finite is of course a property which resonates well with operator algebraists. Stably finite $C^*$-algebras are a prominent class of operator algebras and on the level of von Neumann algebras this corresponds to finite ones, i.e., those where we have a trace. In our free probability context we usually are working in a finite setting, thus this is tailor-made to our purposes.
Of course, from an operator algebraic point of view, type III von Neumann algebras or purely infinite $C^*$-algebras are also of much interest, but the above shows that taking rational functions of operators in such a setting might not be a good idea.

So let us now fix a stably finite algebra $\cA$. Let $\mathfrak{r}$ be a rational function, and consider two rational expressions $r_1$ and $r_2$ representing $\mathfrak{r}$. The result of Cohn which we mentioned above says then that for any tuple $(X_1,\dots,X_g)$ in $\cA$, we have $r_1(X_1,\dots,X_g)=r_2(X_1,\dots,X_g)$, provided both sides make sense; we can then safely declare this
common value as the value $\mathfrak{r}(X_1,\dots,X_g)$ of the rational function $\mathfrak{r}$ applied to this tuple $(X_1,\dots,X_g)$. However, we have now to face the question when those evaluations in $r_1$ or $r_2$ make sense -- so we have to address the issue of the domain of rational expressions and rational functions.
Clearly the domain of a rational expression consists of all tuples of operators from $\cA$
which we can insert in our rational expression such that all operators which have to be inverted are actually invertible. For example, the domain of $r(x_1)=x_1^{-1}$ are all invertible operators in $\cA$. However, for a rational function the issue of the domain is more subtle, because different $r_1$ and $r_2$ representing the same rational function $\mathfrak{r}$ might have different domains.
For example, $r_1(x_1)=x_1+x_1x_1^{-1}$ has again only invertible operators in its domain, but the better representation of $\mathfrak{r}$ as $r_2(x_1)=x_1+1$ shows that this restriction was somehow artificial, and is owed more to the bad choice of representation than to a property of our function. A canonical choice of domain for a rational function $\mathfrak{r}$ would be
the union of the domains of all rational expressions representing $\mathfrak{r}$. It is not off-hand clear, though,
whether there exists a rational expression $r$ which has this maximal domain.

If we switch to matrix realizations $\br = C \Lambda^{-1} B$ of our rational function $\fr$ then the situation becomes somehow nicer. We still have the problem that there are different choices for matrix realizations of the same rational function, each of them having possibly different domains. But at least the description of the domain becomes now smoother, as there is only one inversion involved in a matrix realization. It is important for our applications to be able to control the relation between the various domains. It is one of our main results that for each rational expression $r$ we can construct a matrix realization $\br$ such that the domain of $\br$ contains the domain of $r$. In the regular case  we can strengthen this to a statement about rational functions $\fr$. There we can rely on the powerful machinery developed in the control community to cut down matrix realizations to a minimal one, without decreasing the domain. The uniqueness of the minimal realization shows then that for a regular rational function $\fr$ there exists a matrix realization (namely, the minimal one) $\br$ with the property that its domain includes the domain of any rational expression which represents $\fr$. Hence the domain of this minimal matrix realization $\br$ is the canonical choice for the domain of the rational function $\fr$.
We want to emphasize that such results for domains in matrices have been known before, but here we are talking about the much more general situation of domains in any stably finite algebra. For our applications to free probability theory such controls of domains in stably finite $C^*$-algebras or von Neumann algebras is crucial.

Let us clarify those last remarks by an example.
Consider the rational function $\mathfrak{r}$ which is given by the following matrix realization
$${\br}(x_1,x_2,x_3,x_4)=\begin{pmatrix}
1&0
\end{pmatrix}
\begin{pmatrix}
x_1&x_2\\
x_3&x_4
\end{pmatrix}^{-1}
\begin{pmatrix}
1\\0
\end{pmatrix}.
$$
Then the canonical maximal domain of this rational function $\mathfrak{r}$ in a stably finite algebra $\cA$  are tuples $(X_1,X_2,X_3,X_4)$ in $\cA$, for which the
matrix
$\begin{pmatrix}
X_1&X_2\\
X_3&X_4
\end{pmatrix}$
is invertible in $M_2(\cA)$; in which case the value of our rational function for those operators is given by
$$\mathfrak{r}(X_1,X_2,X_3,X_4)=\begin{pmatrix}
1&0
\end{pmatrix}
\begin{pmatrix}
X_1&X_2\\
X_3&X_4
\end{pmatrix}^{-1}
\begin{pmatrix}
1\\0
\end{pmatrix}.
$$
However, there is no global (scalar as opposed to matrix of) rational expression $r$ for capturing this whole domain. The Schur complement formulas give us such expressions, but, according to the chosen pivot, we have different expressions, and those have different domains. So we have for example
$$r_1(x_1,x_2,x_3,x_4)=x_1^{-1}+x_1^{-1}x_2(x_4-x_3x_1^{-1}x_2)^{-1}x_3x_1^{-1}$$
or 
$$r_2(x_1,x_2,x_3,x_4)=(x_1-x_2x_4^{-1}x_3)^{-1}.$$
Both $r_1$ and $r_2$ represent the same rational function $\mathfrak{r}$, but they have different domains, and none of them has the maximal domain.

The paper is organized as follows.

In Sect. \ref{sec:NCrat}, we first give a brief introduction to the general theory of noncommutative rational expressions and functions, emphasizing the important case where they are regular at $0$, i.e., when their domain contains the point $0$.

Sect. \ref{sec:Realization} describes then system realizations for NC multi-variable rational functions, extending the classical work of Sch\"{u}tzenberger \cite{S61}. M. Fliess \cite{F74a} subsequently used Hankel operators effectively in such realizations. See the book \cite{BR84} for a good exposition. A basic reference in the operator theory community is J. Ball, T. Malakorn, and G. Groenewald \cite{BMG05}, and more recently D. S. Kaliuzhnyi-Verbovetskyi and V. Vinnikov \cite{KVV12b}.

These parts are mostly of expository nature, but we lay here also the groundwork for our subsequent considerations: while noncommutative rational functions were treated in the literature before only as formal objects rather than as actual functions, we are going to address the seemingly unexplored question of evaluating rational expressions, rational functions, and their descriptor realizations on elements coming from general algebras.

In Sect. \ref{sec:evaluations}, we provide a new framework for treating such questions. As we will see in Sect. \ref{sect:2-4}, especially with Theorem \ref{thm:sameOnX} and Theorem \ref{thm:sameOnX_generalized}, the crucial condition that guarantees that rational functions and their descriptor realizations behave well under evaluation is that the underlying algebra is stably finite. The notion of stably finite algebras is introduced in Sect. \ref{subsec:stably_finite}.
In Sect. \ref{subsec:minimality_implies_maximal_domain}, we contribute with Lemma \ref{lem:domMin} the very important observation that the minimal selfadjoint descriptor realization of a selfadjoint matrix of rational expressions contains -- again under the stably finite hypothesis -- the domain of all other selfadjoint descriptor realizations and thus has the largest domain among all of them.
This extends the corresponding result of Kaliuzhnyi-Verbovetskyi and Vinnikov \cite{KVV09} and of Volcic \cite{Vol16} on matrix algebras.
In Sect. \ref{subsec:minimal_realizations} we deepen this observation. In fact, we will prove here Theorem \ref{thm:rep_min}, which states that over any stably finite algebra $\cA$, the $\cA$-domain of any matrix of rational expressions is contained in the $\cA$-domain of each of its minimal realizations; an analogous result for selfadjoint matrices of rational expressions will also be given in Theorem \ref{thm:rep_min}.

These considerations rely crucially on results that will be presented before in Sect. \ref{sec:Algorithm}. Inspired both by the theory of descriptor realizations and by closely related constructions of Anderson \cite{And}, Cohn \cite{Co71,Co06}, and Malcolmson \cite{Mal78}, we introduce here the notion of formal linear representations. Like the approach presented in \cite{Vol15}, formal linear representations apply to general rational expressions (i.e., without the regularity constraint). Furthermore, they enjoy the important feature of having comparably large domains on any unital algebra; in Theorem \ref{thm:realizations_matrices_of_rational_expressions}, we will see how formal linear representations allow us to construct special descriptor realizations having the same property.

Sect. \ref{sec:FP} is then devoted to several applications of descriptor realizations in the context of free probability. After a brief introduction to scalar- and operator-valued free probability, where we recall in particular the powerful subordination results about the operator-valued free additive convolution that were obtained in \cite{BMS2013}, we finally present our main results, Theorem \ref{thm:mainrep1} and Theorem \ref{thm:mainrep2}.
Roughly speaking, Theorem \ref{thm:mainrep1} explains, in the setting of a $C^\ast$-probability space $(\cA,\phi)$, how the (matrix-valued) Cauchy transform $G_{r(X)}$ of a selfadjoint rational expression $r$ (or a matrix thereof) and any selfadjoint point $X=(X_1,\dots,X_g)$ in the $\cA$-domain of $r$ can be computed with the help of generalized descriptor realizations realizing $r$ at the point $X$ in the sense of Definition \ref{def:realized_by}. At the end, this means that $G_{r(X)}$ can be obtained from the matrix-valued Cauchy transform $G_{\hpen(X)}$ for some linear pencil
$$\hpen(x) = \hpen_0 + \hpen_1 x_1 + \dots + \hpen_g x_g$$
that consists of selfadjoint matrices $\hpen_0,\hpen_1,\dots,\hpen_g$ over $\CC$.
Theorem \ref{thm:mainrep2} then tells us how such generalized descriptor realizations realizing $r$ at a given point $X$ can be found explicitly.

Finally, in Sect. \ref{subsec:main_problems}, we explain how these Theorems \ref{thm:mainrep1} and \ref{thm:mainrep2} provide a complete solution of the two Problems \ref{prob:dist} and \ref{prob:Brown}, asking for an algorithm that allows to compute (at least numerically) distributions and even Brown measures of rational expressions, evaluated in freely independent selfadjoint elements with given distributions; this generalizes results from \cite{BMS2013} and \cite{BSS2015} from the case of noncommutative polynomials to noncommutative rational expressions. We conclude in Sect. \ref{subsec:examples} with several concrete examples.


\section{An Introduction to NC Rational Functions}
\label{sec:NCrat}
  
At first glance this notation section may look formidable to many readers. We offer the reassurance that much of it lays out formal properties of noncommutative rational functions which merely capture manipulations with functions on matrices which are very familiar to matrix theorists and operator theorists. People in these areas might well want to skip these fairly intuitive basics, on first reading, to move quickly to Section \ref{sec:Realization} and beyond.

Beware, rather unintuitive is Section \ref{sect:2-4}, which tells us that rational expressions have good properties when evaluated on a type II$_1$ factor.

\subsection{The Schur Complement Formula}
\label{subsec:Schur}

The Schur complement formula is a well-known tool to compute inverses of (block) matrices having entries in a unital but not necessarily commutative algebra $\cA$ over the field $\bbK$ of real or complex numbers. Since this statement are crucial for our purposes, we include it here for convenience of the reader.

Throughout the article, we denote by $M_{m \times n}(\cA)$ the space of all $m\times n$ matrices with entries in $\cA$; in particular, we will abbreviate $M_{n\times n}(\cA)$ by $M_n(\cA)$.

Let matrices $A\in M_k(\bbK) \otimes_\bbK \cA$, $B\in M_{k\times l}(\bbK) \otimes_\bbK \cA$, $C\in M_{l\times k}(\bbK) \otimes_\bbK \cA$ and $D\in M_l(\bbK) \otimes_\bbK \cA$ be given and assume that $D$ is invertible in $M_l(\bbK) \otimes_\bbK \cA$. Then, the matrix
$$\begin{pmatrix} A & B\\ C & D\end{pmatrix}$$
is invertible in $M_{k+l}(\bbK) \otimes_\bbK \cA$ if and only if the \textbf{Schur complement}\index{Schur complement} $A-BD^{-1}C$ is invertible in $M_k(\bbK) \otimes_\bbK \cA$, and in this case have the relation
\begin{equation}\label{eq:Schur}
\begin{aligned}
\begin{pmatrix} A & B\\ C & D \end{pmatrix}^{-1}
&= \begin{pmatrix} 1 & 0\\ -D^{-1}C & 1\end{pmatrix} \begin{pmatrix} (A-BD^{-1}C)^{-1} & 0\\ 0 & D^{-1}\end{pmatrix} \begin{pmatrix} 1 & -BD^{-1}\\ 0 & 1\end{pmatrix}\\
&= \begin{pmatrix} 0 & 0\\ 0 & D^{-1}\end{pmatrix} + \begin{pmatrix} 1\\ -D^{-1}C\end{pmatrix} (A-BD^{-1}C)^{-1} \begin{pmatrix} 1 & -BD^{-1}\end{pmatrix},
\end{aligned}
\end{equation}
which is often called the \textbf{Schur complement formula}\index{Schur complement!formula}.

\subsection{NC Polynomials and NC Rational Expressions}
\label{sec:NCRatIntro}
  
NC rational functions suited to our purposes here are described in detail in \cite{HMV06}, Section 2 and Appendix A. Our discussion here draws heavily from that.

That process has a certain unavoidable heft to it, and we hope to make this paper accessible to people in operator theory where NC rational functions are manipulated successfully without too much formalism. Thus we give here a brief version of our formalism.

\subsubsection{A Few Words about Words}

Throughout this paper $x = (x_1, \ldots , x_g)$ denotes a $g$-tuple of noncommuting variables $x_1,\dots,x_g$.

Let ${\WOR}_g$ denote the free monoid on the $g$ symbols $\{\fgg_1,\ldots,\fgg_g\}$. For a word $w = \fgg_{i_1} \ldots \fgg_{i_k} \in {\WOR}_g$ we define $x^w = x_{i_1} \ldots x_{i_k}$ and we put $x^\emptyset = 1$ for the empty word $\emptyset \in {\WOR}_g$.

Occasionally we consider variables which are formal transposes $x_j^{\transpose}$ of a
variable $x_j$, and words in all of these variables
$ x_1, \ldots , x_g, x_1^{\transpose}, \ldots , x_g^{\transpose} $,
often called the words in $x, x^{\transpose}$.
If $w$ is in $ \WOR_{g}$, then
$w^{\transpose}$ denotes the transpose of a word $w$.
For example, given the word (in the $x_j$'s)
$x^w=x_{j_1}x_{j_2}\ldots x_{j_n}$,
the involution applied to $x^w$ is
$(x^w)^{\transpose}=x_{j_n}^{\transpose} \ldots x_{j_2}^{\transpose} x_{j_1}^{\transpose}$,
which, if the variables $x_j$ are symmetric (i.e., if $x_j^{\transpose} = x_j$),
is
$x^{(w^{\transpose})}=x_{j_n}\ldots x_{j_2} x_{j_1}$. In particular, $\emptyset^{\transpose} = \emptyset$.

Note that, depending on whether we work in a real or complex setting, we will sometimes write $x_j^{\ast}$ instead of $x_j^{\transpose}$. In these cases, it is also appropriate to call the variables $x_1, \ldots, x_g$ selfadjoint if they satisfy the condition $x_j^{\ast} = x_j$.

\subsubsection{The Algebra of NC Polynomials}

Throughout the following, $\bbK$ stands for either the field $\bbR$ of real numbers or the field $\bbC$ of complex numbers.

We define $\bbK \langle x_1,\ldots,x_g \rangle$ to be the algebra of
noncommutative polynomials over $\bbK$ in the non-commuting
variables $x_1,\ldots,x_g$. Each element $p\in \bbK \langle x_1,\ldots,x_g \rangle$ has a unique representation of the form $p = \sum_{w\in \WOR_g} p_w x^w$ with some family $(p_w)_{w\in\WOR_g}$ of coefficients $p_w \in \bbK$, for which $\{w\in \WOR_g|\ p_w\neq 0\}$ is finite. We often abbreviate
$\bbK \langle x_1,\ldots,x_g \rangle$ by $\bbK\langle x \rangle$.

When the variables $x_j$ are symmetric (respectively selfadjoint) the algebra
$\bbK \langle x \rangle$
maps to itself under the involution $^{\transpose}$ (respectively $^{\ast}$).

In the real case, for non-symmetric variables the algebra of polynomials in them is denoted
$$
\bbR \langle x_1,\ldots,x_g, x_1^{\transpose}, \ldots, x_g^{\transpose} \rangle \qquad \text{or} \qquad \bbR \langle x, x^{\transpose} \rangle.
$$
Accordingly, in the complex case, we denoted by
$$
\bbC \langle x_1,\ldots,x_g, x_1^{\ast}, \ldots, x_g^{\ast} \rangle \qquad \text{or} \qquad \bbC \langle x, x^{\ast} \rangle
$$
the $\ast$-algebra of polynomials in non-selfadjoint variables $x_1,\ldots,x_g$.

\subsubsection{NC Rational Expressions}
  \label{item:intorat3}

We define a \textbf{NC rational expression in the variables $x=(x_1,\ldots,x_g)$}\index{rational expression} to be a syntactically valid combination of elements from the ring $\bbK\langle x_1,\dots,x_g\rangle$ of NC polynomials (over the field $\bbK$ of real or complex numbers), the arithmetic operations $+$, $\cdot$, and $\cdot^{-1}$, and parentheses $($, $)$.

This definition requires some explanation.
First of all, the reader might worry for instance about
\begin{equation}\label{eq:ex_degenerate}
0^{-1} \qquad \text{and} \qquad ((-1) + x_1 \cdot x_1^{-1})^{-1}, 
\end{equation}
both of which are valid NC rational expressions according to our definition. This, however, only highlights the important fact that NC rational expressions are purely formal objects, meaning in particular that we ignore all familiar arithmetic rules. Informally speaking, NC rational expressions should thus be seen as a chain of (nested) arithmetic operations rather than as algebraic objects of their own right; this will be important in Section \ref{subsec:rat_evaluation}, where evaluations of NC rational expressions on unital algebras will be considered.
Even without having the rigorous definition of domains and evaluations yet, the reader might expect that the examples shown above should not lead to a meaningful evaluation anywhere. This is indeed the case, so that we are able to rule out such ``non-degenerate'' NC rational expressions at least in hindsight; but still, one would prefer to have a more direct tool to exclude them in advance.
Unfortunately, there is no hope to find such a criterion that works in full generality, since domains strongly depend on the underlying algebras. There is however the important class of regular NC rational expressions, which can be easily characterized and excludes the examples given above.

We use a recursion to define the notion of a
  \textbf{NC rational expression $r$ regular (at zero)} \index{rational expression!regular (at zero)}, say in the variables $x=(x_1,\ldots,x_g)$, and its value $r(0)$ at $0$.
(For details of this definition see the excellent survey \cite{KVV12}.)
  This class includes noncommutative polynomials and $p(0)$ is the
  value of $p$ at $0$, which means that if $p$ is written as $p = \sum_{w\in \WOR_g} p_w x^w$, then $p(0) := p_\emptyset$.
  If $p(0)$ is invertible, then $p$ is invertible, this
  inverse is a NC rational expression regular at $0$,
  and $p^{-1}(0)=p(0)^{-1}$.
  Formal sum and products of NC rational expressions
  regular at $0$ and their value at $0$ are defined accordingly.
  Finally, a NC rational expression $r$ regular at $0$
  can be inverted  provided $r(0)\neq 0$; this inverse
  is an  NC rational expression, and $r^{-1}(0)=r(0)^{-1}$.
Note that in general $r(0)$ can itself be zero for the rational expressions we consider. Only the parts of it which must be inverted are required to be different from $0$ at $0$.

\begin{example}   
Consider the rational expression
 \begin{equation}
 \label{eq:expr}
  r = (1-x_1)^{-1} + (1-x_1)^{-1}x_2 \left [
(1-x_1)-x_2(1-x_1)^{-1}x_2 \right ]^{-1} x_2(1-x_1)^{-1}.
\end{equation}
Note it is made from inverses of
$(1 -x_1)$ and $(1-x_1)-x_2(1-x_1)^{-1}x_2$
both of which meet our technical invertible at $0$ convention.
\end{example}

\subsection{Evaluations of NC Polynomials and of NC Rational Expressions}
\label{subsec:rat_evaluation}

Suppose that $\cA$ is any unital algebra over the field $\bbK$. The unit of $\cA$ will be denoted by $1_\cA$.

\subsubsection{Polynomial Evaluations}

If $p$ is a noncommutative polynomial in the variables $x_1,\dots,x_g$, the evaluation $p(X)$ of $p$ at any point $X=(X_1,X_2,\dots,X_g)$ in $\cA^g$ is defined by simply replacing $x_j$ by $X_j$. More formally, we declare $p(X) := \ev_X(p)$, where $\ev_X$ denotes the unital algebra homomorphism
\begin{equation}\label{eq:ev_polys}
\ev_X:\ \bbK\langle x_1,\ldots,x_g\rangle \to \cA,\quad \sum_{w\in\WOR_g} p_w x^w \mapsto \sum_{w\in\WOR_g} p_w X^w,
\end{equation}
where $X^w := X_{i_1} \ldots X_{i_k}$ for any word $w=\chi_{i_1} \ldots \chi_{i_k}\in\WOR_g$ and $X^\emptyset := 1_\cA$. In particular, we have that
$$p(0_\cA,\dots,0_\cA) = p(0_\bbK,\dots,0_\bbK) 1_\cA,$$
where we typically suppress the subscripts on $0$.

In the complex case, if $\cA$ is even a $\ast$-algebra, the evaluation homomorphism extends canonically to a $\ast$-homomorphism
$$\ev_X:\ \bbC \langle x_1,\ldots,x_g, x_1^\ast,\dots, x_g^\ast\rangle \to \cA \quad \text{resp.} \quad \ev_X:\ \bbC \langle x_1,\ldots,x_g \rangle \to \cA$$
for non-selfadjoint respectively selfadjoint elements $X_1,\dots,X_g$ in $\cA$. The same holds true, of course, in the real case.

\subsubsection{Rational Expression Evaluations}
\label{sec:evalrat}

Defining evaluations for NC rational expressions at points in $\cAg$ for a unital $\bbK$-algebra is slightly more delicate than for NC polynomials, because it requires worrying about the domain of an expression.
These considerations will lead us to a (very useful) equivalence relation on rational expressions; this will be discussed in detail in Section \ref{sec:evalequiv}.

\begin{definition}\label{def:ev_dom}
For any NC rational expression $r$ in $x=(x_1,\dots,x_g)$, we define its \textbf{$\cA$-domain $\dom_\cA(r)$}\index{rational expression!domain} together with its \textbf{evaluation $r(X) := \ev_X(r)$}\index{rational expression!evaluation} at any point $X=(X_1,\dots,X_g)\in\dom_\cA(r)$ by the following rules:
\begin{enumerate}
 \item If $p$ is any NC polynomial, then $\dom_\cA(p) = \cA^g$ with $\ev_X(p)$ defined as in \eqref{eq:ev_polys}.
 \item\label{it:ev_dom-sum} If $r_1,r_2$ are NC rational expressions in $x$, we have $$\dom_\cA(r_1 \cdot r_2) = \dom_\cA(r_1) \cap \dom_\cA(r_2) \quad\text{and}\quad \ev_X(r_1 \cdot r_2) = \ev_X(r_1) \cdot \ev_X(r_2).$$
 \item\label{it:ev_dom-prod} If $r_1,r_2$ are NC rational expressions in $x$, we have $$\dom_\cA(r_1 + r_2) = \dom_\cA(r_1) \cap \dom_\cA(r_2) \quad\text{and}\quad \ev_X(r_1 + r_2) = \ev_X(r_1) + \ev_X(r_2).$$
 \item\label{it:ev_dom-inv} If $r$ is a NC rational expression in $x$, we have $$\dom_\cA(r^{-1}) = \{X\in\dom_\cA(r): \text{$\ev_X(r)$ is invertible in $\cA$}\} \quad\text{and}\quad \ev_X(r^{-1}) = \ev_X(r)^{-1}.$$
\end{enumerate}
\end{definition}

\subsection{Evaluation equivalence of NC Rational Expressions and NC Rational Functions}

We shall ultimately define the notion of a NC rational function in terms of NC rational expressions.

A difficulty is that two different expressions, such as
\begin{equation}
\label{eq:exr1r2}
       r_1= x_1 (1 - x_2 x_1)^{-1} \ \ \ \
\text{and} \ \ \ \
       r_2 = (1 - x_1 x_2)^{-1}x_1
\end{equation}
can be converted to each other with algebraic operations. Our definition of NC rational expressions, however, fully ignores all arithmetic rules, as we have explained in Section \ref{item:intorat3}. Thus one needs to specify an equivalence relation on rational expressions that reflects this algebraic structure.

The classical notion prevailing in automata and systems theory works in the regular case and uses formal power series;
this is summarized later in Section \ref{sec:series}. Its use to a free analyst is that the classical theorems we need 
are proved and stated using this type of equivalence.

However, what a free analyst often does is substitute matrices
or operators in for the variables and so one needs a notion of equivalence based on $r_1$ and $r_2$
having the same values when evaluated on the same operators.
This type of equivalence is developed in \cite{HMV06}
for evaluation on matrices
and proved to be the same as classical power series equivalence for rational expressions regular at zero.
This alleviates technical headaches.
We now define the terms just discussed.

\subsubsection{Evaluation Equivalence of Rational Expressions and NC Rational Functions}
\label{sec:evalequiv}

   We can use evaluations
    to define an equivalence on noncommutative
   rational expressions which we call evaluation equivalence.
   Two NC rational expressions $r$ and $\widetilde r$
   are \df{$\cA$-evaluation equivalent} provided
   $$r(X)=\widetilde{r}(X) \quad  \mbox{for each } \quad 
   X \in \domA{r} \cap \domA{\widetilde{r}}.$$
  Of course the domains for some algebras $\cA$
  might be small in which case this equivalence is weak.

The usual domains considered in the free analysis context consists of matrices of all sizes. We denote by $\bbM(\RR)$ and $\bbM(\CC)$ the graded algebras of all square, real or complex, matrices, i.e.,
\begin{equation}
 \bbM(\RR):= \coprod_{n=1}^ \infty M_n(\RR)  \qquad \text{or} \qquad \bbM(\CC):= \coprod_{n=1}^\infty M_n(\CC).
\end{equation}
Accordingly, we put for any NC rational expression $r$ over $\bbK$
$$\dom_{\bbM(\bbK)}(r) := \coprod_{n=1}^\infty \dom_{M_n(\bbK)}(r) \subseteq \coprod_{n=1}^\infty M_n(\bbK)^g.$$
We may also introduce an equivalence relation with respect to those graded algebras: two NC rational expressions $r_1$ and $r_2$ are said to be \textbf{$\bbM(\bbR)$-equivalent}\index{rational expression!$\bbM(\bbK)$-evaluation equivalent} (respectively \textbf{$\bbM(\CC)$-equivalent}) or simply \df{matrix-equivalent}, if they are $M_n(\RR)$-evaluation equivalent (respectively $M_n(\CC)$-evaluation equivalent) for all $n>0$, or equivalently, if
$$r_1(X) = r_2(X) \quad\text{for each}\quad X \in \dom_{\bbM(\bbK)}(r_1) \cap \dom_{\bbM(\bbK)}(r_2).$$
We take this as our definition of a rational function. However, in order to exclude exceptional expressions such as those given in \eqref{eq:ex_degenerate}, we will restrict to \textbf{non-degenerate}\index{rational expression!non-degenerate} NC rational expressions, meaning NC rational expressions $r$ satisfying $\dom_{\bbM(\bbK)}(r) \neq \emptyset$.
A \df{NC rational function} or simply \df{rational function} is defined then to be the class of $\bbM(\bbK)$-equivalent 
non-degenerate rational expressions; see also Appendix A.6 of \cite{HMV06}. 
This corresponds to classical engineering type situations and we typically use German (Fraktur) font to
denote NC rational functions. Our definition is justified by \cite{KVV12}, where it was shown that this construction indeed yields the free field $\CC\plangle x_1,\dots,x_g\prangle$, which is the universal skew field of fractions associated to the ring $\CC\langle x_1,\dots,x_g\rangle$ of NC polynomials; see \cite{Co71}[Chapter 7].
Furthermore, it can be shown that if two rational expressions are matrix-equivalent, one expression can be changed into the other by algebraic operations; for this see \cite{CR}.
\textbf{Regular NC rational functions}\index{NC rational function!regular}\index{rational function!regular}, namely $\bbM(\bbK)$-equivalence classes of regular NC rational expressions, constitute an important subclass of all NC rational functions, for which an alternative description can be given; we will say more about this in Section \ref{sec:series}. Moreover, in Lemma 16.5 of \cite{HMV06} it is shown that $\dom_{M_n(\bbR)}(r)$, for $r$ which is regular at zero, is a non-empty Zariski open subset of $M_n(\bbR)^g$ containing $0$.

\subsection{Matrix Valued NC Rational Expressions and Functions}
\label{sec:matVrat}

  The notion of rational expression
  is broadened by using matrix constructions.
  Indeed, this more general notion is
  often used in this paper.

Throughout the following, we we will work with noncommuting variables $x=(x_1,\ldots,x_g)$ and we suppose that they commute with scalar matrices of any size. Given a matrix $\MAT$ with entries $\MAT_{ij}$ and a variable $x_l$, let
\begin{equation}\label{eq:scalar_matrices_commute_with_indeterminates}
\MAT x_l = x_l \MAT
\end{equation}
denote the matrix with entries given by
$$(\MAT x_l)_{ij}= \MAT_{ij}x_l.$$
The identification made in \eqref{eq:scalar_matrices_commute_with_indeterminates} precisely means that scalar matrices $\MAT$ are supposed to commute with the indeterminates $x_1, \ldots , x_g$.

\subsubsection{NC Linear Pencils}
\label{subsec:NC_Linear_Pencils}

Linear pencils are the most basic matrix-valued expressions and they will play a fundamental role in what follows. We introduce them together with natural operations in the following definition.
 
\begin{definition}
\label{def:linear_pencils}
Let $x=(x_1,\dots,x_g)$ be a $g$-tuple of variables.
\begin{itemize}
 \item[(i)] A \textbf{linear pencil (of size $n \times m$) in $x$}\index{NC linear pencil} is an expression of the form
$$\Lambda(x) := \Lambda^{(0)} + \Lambda^{(1)} x_1 + \dots + \Lambda^{(g)} x_g$$
with matrices $\Lambda^{(0)},\Lambda^{(1)},\dots,\Lambda^{(g)}\in M_{n\times m}(\CC)$.

Note the common term linear pencil is a misnomer
in  that {\it linear pencils are actually affine linear}, that is,
the pencil $\Lambda(x)$ is linear if and only if $\Lambda^{(0)}=0$.

 \item[(ii)] If linear pencils $\Lambda_{k,l}(x)$ of size $n_k \times m_l$ in $x$ with $\Lambda_{k,l}=(\Lambda^{(0)}_{k,l}, \Lambda^{(1)}_{k,l}, \ldots, \Lambda^{(g)}_{k,l})$ are given for $1\leq k \leq K$ and $1 \leq l \leq L$, we write
$$\Lambda(x) := \begin{pmatrix} \Lambda_{1,1}(x) & \hdots & \Lambda_{1,L}(x)\\ \vdots & \ddots & \vdots\\ \Lambda_{K,1}(x) & \hdots & \Lambda_{K,L}(x) \end{pmatrix}$$
for the linear pencil $\Lambda = \Lambda^{(0)} + \Lambda^{(1)} x_1 + \dots + \Lambda^{(g)} x_g$ of size $(n_1 + \dots + n_K) \times (m_1 + \dots + m_L)$ in $x$ with
$$\Lambda^{(j)} := \begin{pmatrix} \Lambda_{1,1}^{(j)} & \hdots & \Lambda_{1,L}^{(j)}\\ \vdots & \ddots & \vdots\\ \Lambda_{K,1}^{(j)} & \hdots & \Lambda_{K,L}^{(j)} \end{pmatrix},\qquad\text{for $j=0,1,\dots,g$}.$$

 \item[(iii)] If a linear pencil $\Lambda(x) = \Lambda^{(0)} + \Lambda^{(1)} x_1 + \dots + \Lambda^{(g)} x_g$ of size $n \times m$ in $x$ and matrices $S\in M_n(\bbC)$ and $T \in M_m(\bbC)$ are given, we define
$$(S \Lambda T)(x) := S \Lambda(x) P = (S\Lambda^{(0)}T) + (S\Lambda^{(1)}T) x_1 + \dots + (S\Lambda^{(g)}T) x_g$$
 
 \item[(iv)] If a linear pencil $\Lambda = \Lambda^{(0)} + \Lambda^{(1)} x_1 + \dots + \Lambda^{(g)} x_g$ of size $n \times m$ is given, then
$$\Lambda^\ast(x) := \Lambda(x)^\ast = (\Lambda^{(0)})^\ast + (\Lambda^{(1)})^\ast x_1 + \dots + (\Lambda^{(g)})^\ast x_g$$
defines a linear pencil of size $m \times n$ in $x$. In the real case, $\Lambda^\transpose$ is defined analogously. A square linear pencil $\Lambda$ with real (respectively complex) matrices satisfying $\Lambda^\transpose = \Lambda$ (respectively $\Lambda^\ast = \Lambda$) is called \textbf{symmetric}\index{NC linear pencil!symmetric} (respectively \textbf{selfadjoint}\index{NC linear pencil`!selfadjoint}).
\end{itemize}
\end{definition}

Conversely, if matrices $\MAT=(\MAT^{(0)}, \MAT^{(1)}, \ldots, \MAT^{(g)})$ of size $n\times m$ are given, we denote by $\Lambda_\MAT$ the linear pencil of size $n\times m$ that is defined by
\begin{equation}\label{eq:Lambda-pencil}
\Lambda_\MAT(x) := M^{(0)} - L_\MAT(x) \qquad\text{with}\qquad L_\MAT(x) := \MAT^{(1)} x_1 + \dots + \MAT^{(g)} x_g.
\end{equation}
The minus sign appearing in front of $L_\MAT$ might seem strange at first sight, but below, we will see that this is indeed a reasonable choice.
With this notation, we clearly have $S \Lambda_\MAT T = \Lambda_{S \MAT T}$ and $\Lambda_\MAT^\ast = \Lambda_{\MAT^\ast}$ (respectively $S L_\MAT T = L_{S \MAT T}$ and $L_\MAT^\ast = L_{\MAT^\ast}$), where we put
\begin{align*}
S \MAT T &:= (S \MAT^{(0)} T, S \MAT^{(1)} T, \dots, S\MAT^{(g)} T) \qquad\text{and}\\
\MAT^\ast &:= ((\MAT^{(0)})^\ast, (\MAT^{(1)})^\ast, \dots, (\MAT^{(g)})^\ast).
\end{align*}
In the real case, transposition is treated similarly.

As an example, for
$$
\MAT_0:=
\left(
\begin{array}{cc}
     1 & 0 \\
     0 & -1 \\
\end{array}
\right)
\quad
\MAT_1:= 
- \left(
\begin{array}{cc}
     3 & 2 \\
     2 & 1 \\
\end{array}
\right)
\quad
\MAT_2:= - \left(
\begin{array}{cc}
     5 & 4 \\
     4 & 2 \\
\end{array}
\right),
$$
the pencil is
$$
\pen(x)=\left(
\begin{array}{cc}
    1+ 3 x_1 +5x_2 & 2 x_1 +4 x_2 \\
     2 x_1 +4 x_2 & -1+ x_1 + 2x_2 \\
\end{array}
\right).
$$

\subsubsection{Evaluation of pencils}
\label{subsubsec:eval_pencils}

The variables $x_i$ will often be evaluated on $X_i$
which are square matrices 
or elements of a particular algebra $\cA$ over $\bbK$.
For $n\times m$ NC linear pencils the evaluation rule is 
\begin{equation}\label{eq:pencil_eval-1}
\Lambda(X) = \Lambda^{(0)} \otimes 1_\cA + \Lambda^{(1)} \otimes X_1 + \dots + \Lambda^{(g)} \otimes X_g,
\end{equation}
where $X=(X_1,X_2,\dots,X_g)$ is in $\cA^g$, so that $\Lambda(X) \in M_{n\times m}(\bbK) \otimes_\bbK \cA$.
Later evaluation will be discussed in more generality.

\subsubsection{Matrix-valued NC Rational Expressions}
\label{subsubsec:matVrat}

  Matrix-valued NC rational expressions are defined
  by analogy to (scalar-valued) rational expressions as presented in Section \ref{item:intorat3}:
  a \df{matrix-valued NC polynomial} is a NC polynomial
  with matrix coefficients.
  The $\bbK$-vector space of all $n\times m$-matrix-valued NC polynomials in the variables $x=(x_1,\ldots,x_g)$ will be denoted by $M_{n\times m}(\bbK) \langle x_1,\ldots,x_g\rangle$; note that the latter forms a $\bbK$-algebra in the case of square matrix-valued NC polynomials, i.e, if $n=m$.
  The reader should be aware of the fact that $M_{n\times m}(\bbK) \langle x_1,\ldots,x_g\rangle$ is also subjected to the convention \eqref{eq:scalar_matrices_commute_with_indeterminates}, so that canonically $M_{n\times m}(\bbK) \langle x_1,\ldots,x_g\rangle \cong M_{n\times m}(\bbK) \otimes_\bbK \bbK\langle x_1,\dots,x_g\rangle$. Accordingly, each element $p$ in $M_{n\times m}(\bbK) \langle x_1,\ldots,x_g\rangle$ has like in the scalar-valued setting a unique representation as $p = \sum_{w\in\WOR_g} p_w x^w$ with coefficients $p_w \in M_{n\times m}(\bbK)$, where only finitely many of them are different from zero. 

  All matrix-valued NC polynomials are matrix-valued rational
  expressions and a general \textbf{matrix-valued NC rational expression}\index{rational expression!matrix-valued} is built in a syntactically valid way -- by using the operations $+$, $\cdot$, and ${\cdot}^{-1}$, and by placing parentheses -- out of matrix-valued NC polynomials $M_{n\times m}(\bbK) \langle x_1,\dots,x_g\rangle$.
  Of course, the constraint ``syntactically valid'' includes here also the requirements that matrices are added and multiplied only when their sizes allow and that the operation ${\cdot}^{-1}$ is applied only to matrix-valued rational expressions of square type (i.e., $n=m$).
  Note that we also agree on the convention \eqref{eq:scalar_matrices_commute_with_indeterminates}. 
  
  The class of \textbf{regular matrix-valued NC rational expressions $r$}\index{rational expression!matrix-valued!regular} together with their value $r(0)$ are defined analogously: if $p$ is a square matrix-valued NC polynomial
and $p(0)$ is invertible, then $p$ has an inverse $p^{-1}$ whose value at $0$ is given by $p^{-1}(0) = p(0)^{-1}$.
Regular matrix-valued NC rational expressions $r_1$ and $r_2$ can be
added and multiplied whenever their dimensions allow,
with the value at $0$ of the sum and product defined accordingly.
Finally,
a regular square matrix-valued NC rational expression $r$ has a regular inverse
as long as $r(0)$ is invertible.
(See Appendix A \cite{HMV06} for details.)

\subsubsection{Evaluation of Matrix-Valued NC Rational Expressions}

Given any unital $\bbK$-algebra $\cA$, we may define evaluation of matrix-valued NC polynomials via the canonical evaluation map
\begin{equation}\label{eq:ev_polys-matval}
\ev_X:\ M_{n\times m}(\bbK) \langle x_1,\ldots,x_g\rangle \to M_{n\times m}(\bbK) \otimes_\bbK \cA,\quad \sum_{w\in \WOR_g} p_w x^w \mapsto \sum_{w\in \WOR_g} p_w \otimes X^w
\end{equation}
in analogy to \eqref{eq:ev_polys}. Like for scalar-valued NC rational expressions in Definition \ref{def:ev_dom}, we may define now the notion of $\cA$-domain and evaluations for matrix-valued NC rational expressions.

\begin{definition}\label{def:ev_dom-matval}
For any matrix-valued NC rational expression $r$ in $x=(x_1,\dots,x_g)$, we define its \textbf{$\cA$-domain $\dom_\cA(r)$}\index{rational expression!domain} together with its \textbf{evaluation $r(X) := \ev_X(r)$}\index{rational expression!evaluation} at any point $X=(X_1,\dots,X_g)\in\dom_\cA(r)$ by the following rules:
\begin{enumerate}
 \item If $p$ is any matrix-valued NC polynomial, then $\dom_\cA(p) = \cA^g$ with $\ev_X(p)$ defined like in \eqref{eq:ev_polys-matval}.
 \item If $r_1,r_2$ are matrix-valued NC rational expressions in $x$, we have $$\dom_\cA(r_1 \cdot r_2) = \dom_\cA(r_1) \cap \dom_\cA(r_2) \quad\text{and}\quad \ev_X(r_1 \cdot r_2) = \ev_X(r_1) \cdot \ev_X(r_2).$$
 \item If $r_1,r_2$ are matrix-valued NC rational expressions in $x$, we have $$\dom_\cA(r_1 + r_2) = \dom_\cA(r_1) \cap \dom_\cA(r_2) \quad\text{and}\quad \ev_X(r_1 + r_2) = \ev_X(r_1) + \ev_X(r_2).$$
 \item If $r$ is a square matrix-valued NC rational expression in $x$, we have $$\dom_\cA(r^{-1}) = \{X\in\dom_\cA(r): \text{$\ev_X(r)$ is invertible}\} \quad\text{and}\quad \ev_X(r^{-1}) = \ev_X(r)^{-1}.$$
\end{enumerate}
Note that accordingly $r(X) \in \bigcup_{n,m\geq 1} M_{n\times m}(\bbK) \otimes_\bbK \cA$.
\end{definition}

We point out that linear pencils -- as a special instance of matrix-valued NC Polynomials -- are matrix-valued NC rational expressions. For later use, let us remark that if $\Lambda = \Lambda^{(0)} + \Lambda^{(1)} x_1 + \dots + \Lambda^{(g)} x_g$ is any linear pencil of square type, then also $\Lambda^{-1}$ is a matrix-valued NC rational expressions and we have
\begin{equation}
\dom_\cA(\Lambda^{-1}) = \{X\in\cA^g|\ \text{$\Lambda(X)$ is invertible}\}.
\end{equation}

\subsubsection{Matrices of NC Rational Expressions}
\label{subsubsec:matrices_of_rational_expressions}

Due to our convention \eqref{eq:scalar_matrices_commute_with_indeterminates}, \textbf{matrices of NC rational expressions}\index{rational expression!matrix of} constitute an important subclass of all matrix-valued NC rational expressions. Indeed, if
$$\ur=(r_{ij})_{\substack{i=1,\dots,d_1\\ j=1,\dots,d_2}}$$
is any $d_1\times d_2$-matrix of NC rational expressions $r_{ij}$ in the variables $x=(x_1,\dots,x_g)$, we may write
$$\ur = \sum_{\substack{i=1,\dots,d_1\\ j=1,\dots,d_2}} E_{ij} r_{ij},$$
where the $E_{ij}$'s denote the canonical matrix units in $M_{d_1 \times d_2}(\bbK)$. Accordingly, for such $\ur$, we have for each unital algebra $\cA$ over $\bbK$ that
$$\dom_\cA(\ur) = \bigcap_{\substack{i=1,\dots,d_1\\ j=1,\dots,d_2}} \dom_\cA(r_{ij})$$
and for each point $(X_1,\dots,X_g) \in \dom_\cA(\ur)$ that
$$\ur(X_1,\dots,X_g) = \sum_{\substack{i=1,\dots,d_1\\ j=1,\dots,d_2}} E_{ij} \otimes r_{ij}(X).$$

\subsubsection{Equivalence Classes and Matrix-valued NC Rational Functions}
\label{subsubsec:matrix_rational_functions}

Two matrix-valued NC rational expressions $r_1$ and $r_2$ are called \textbf{$\bbM(\bbK)$-evaluation equivalent}\index{rational expression!matrix-valued!$\bbM(\bbK)$-equivalent} or simply \df{matrix-equivalent} provided they 
take the same values on their common matrix domain.
A matrix-valued NC rational expression $r$ is \textbf{non-degenerate}\index{rational expression!matrix-valued!non-degenerate}, if $\dom_{\bbM(\bbK)}(r) \neq \emptyset$ holds.
Like in the scalar-valued case discussed in Section \ref{sec:evalequiv}, we may define a \textbf{matrix-valued NC rational function}\index{NC rational function!matrix-valued} to be an equivalence class of non-degenerate matrix-valued NC rational expressions with respect to $\bbM(\bbK)$-evaluation equivalence.
Accordingly, we may introduce a \textbf{regular matrix-valued NC rational function}\index{NC rational function!matrix-valued!regular} as an equivalence class of regular matrix-valued NC rational expressions with respect to $\bbM(\bbK)$-evaluation equivalence.

In particular,
the definition of a regular NC rational function is now amended to mean $1\times 1$ matrix-valued
expressions regular at $0$.
We shall use the phrase
\textbf{scalar regular NC rational expression}\index{NC rational expression!scalar}
if we want to emphasize the absence of  matrix constructions.
Often when the context makes the usage clear
we drop adjectives such as scalar, $1\times 1$,
matrix rational, matrix of rational and the like.
Indeed, it is shown in Appendix A.4 of \cite{HMV06}
that a regular $m_1\times m_2$-matrix valued NC
rational function is in fact the same as a
\textbf{$m_1\times m_2$ matrix of regular NC rational functions}\index{NC rational function!regular!matrix of}, and furthermore,
any regular matrix-valued NC rational function can be a
represented by a matrix
of regular scalar-valued NC rational expressions ``near'' any point in its domain.

\begin{example}
\label{ex:matrat} 
Consider two $2 \times 2$ matrices of NC rational expressions
$$
 m(x) :=
 \begin{pmatrix}
 1-x_1 & -x_2 \\ -x_2 & 1-x_1
 \end{pmatrix}
\qquad\text{and}\qquad 
 w(x) :=
 \begin{pmatrix}
 w_{11} & w_{12} \\
 w_{21} & w_{22}
 \end{pmatrix}, 
$$
where the latter has the following entries: 
\begin{align*}
w_{11}&={\left (1-\ {x_1}\right )}^{-1}+{\left (1-{x_1}\right )}^{-1}
{x_2}{ \left ((1-{x_1})-{x_2}{ \left 
(1-{x_1} \right )}^{-1}{x_2} \right )}^{-1}
{x_2}{ \left (1-{x_1} \right )}^{-1}
\\
w_{12}&={ \left (1-{x_1} \right )}^{-1}{x_2} { \left((1-{x_1})-{x_2} { \left (1-{x_1} \right)}^{-1} {x_2} \right )}^{-1}
\\
w_{21} &= 
{\left ((1-{x_1})-{x_2} {\left (1-{x_1}\right )}^{-1}
{x_2}\right )}^{-1} {x_2} {\left (1-{x_1} \right )}^{-1}
\\
w_{22} &= { \left ((1-{x_1})-{x_2} { \left (1-{x_1} \right 
)}^{-1} {x_2} \right )}^{-1}
\end{align*}
If we substitute for $x$ a matrix tuple 
$X \in M_n(\bbK)^g$ that belongs to the $M_n(\bbK)$-domain of all of these,
then it is an easy (and standard) computation to see that
$w(X) = \ m(X)^{-1}$.
This means that the clearly non-degenerate matrix-valued NC rational expressions $w$ and $m^{-1}$ are $\bbM(\bbK)$-evaluation equivalent. Thus, if $\mathfrak{w}$ and $\mathfrak{m}$ denote the matrix-valued NC rational functions associated to $w$ and $m$, respectively, we have that $\mathfrak{w} = \mathfrak{m}^{-1}$.
Evaluation on general algebras is more subtle; we will come back to this issue in Section \ref{sec:evaluations}.
\end{example}

\subsubsection{Symmetric Matrix-valued NC Rational Expressions}
\label{subsubsec:symmetric_mat-val_rational_expressions}

Let $r$ be any square matrix-valued NC rational expression in the variables $x=(x_1,\dots,x_g)$.
If $\cA$ is a unital $\ast$-algebra, we denote by $\dom_\cA^\sa(r)$ the subset of $\dom_\cA(r)$ that consists of all symmetric (respectively selfadjoint) points $X=(X_1,\dots,X_g) \in \dom_\cA(r)$, i.e., $X = X^\transpose$ (respectively $X=X^\ast$), where $X^\transpose:=(X_1^\transpose,\dots,X_g^\transpose)$ (respectively $X^\ast := (X_1^\ast, \dots, X_g^\ast)$).

  A square matrix-valued NC rational expression $r$
  is called \textbf{symmetric}\index{rational expression!matrix-valued!symmetric} (respectively \textbf{selfadjoint}\index{rational expression!matrix-valued!selfadjoint})
  if $r(X^\transpose) = r(X)^\transpose$ (respectively $r(X^\ast) = r(X)^\ast$) holds for each $X$ in $\dom_\cA^\sa(r)$, whenever $\cA$ is unital real (respectively complex) $\ast$-algebra. Note that this definition slightly differs from the usual terminology used for instance in \cite{HMV06} as $\cA$ runs here over all $\ast$-algebras and not only over matrix algebras.
  
Accordingly, a square matrix $\ur=(r_{ij})_{i,j=1,\dots,d}$ of NC rational expressions is called \textbf{symmetric}\index{rational expression!matrix of!symmetric} (respectively \textbf{selfadjoint}\index{rational expression!matrix of!selfadjoint}), if on any unital real (respectively complex) $\ast$-algebra $\cA$ the condition $\ur(X)^\ast = \ur(X)$ (respectively $\ur(X)^\transpose = \ur(X)$) holds at any point $X\in\dom^\sa_\cA(\ur)$, where
$$\dom_\cA^\sa(\ur) = \bigcap_{i,j=1,\dots,d} \dom_\cA^\sa(r_{ij}).$$

\begin{rem}\label{rem:rational_expression_adjoint}
If $r$ is any matrix-valued NC rational expression, there exists another matrix-valued NC rational expression $r^\transpose$, such that $\dom_{\bbM(\bbR)}(r^\transpose) = \dom_{\bbM(\bbR)}(r)$ and $r(X)^\transpose = r^\transpose(X^\transpose)$ for all $X\in\dom_{\bbM(\bbR)}(r)$.
This matrix-valued NC rational expression $r^\transpose$ is uniquely determined; in fact, in can be constructed by applying successively the following rules:
\begin{itemize}
 \item we assume that all variables $x_1,\dots,x_g$ are symmetric, i.e. $x_j^\transpose = x_j$ for $j=1,\dots,g$, and that $\transpose$ acts on scalar matrices like the usual transposition;
 \item we impose on the mapping $r\mapsto r^\transpose$ the condition that $(r_1 + r_2)^\transpose = r_1^\transpose + r_2^\transpose$, $(r_1 \cdot r_2)^\transpose = r_2^\transpose \cdot r_1^\transpose$, and $(r^{-1})^\transpose = (r^\transpose)^{-1}$ for any $r$ of square type are satisfied.
\end{itemize}
The same holds true in the complex case, with the transpose $\transpose$ replaced by the conjugate transpose $\ast$.
\end{rem}

\begin{rem}
In anticipation of the the machinery of formal linear representations that we will present in Section \ref{sec:Algorithm}, we note that if $\rho=(u,Q,v)$ is a formal linear representation of a NC rational expression $r$, then a formal linear representation of $r^\ast$ is given by $\rho^\ast = (v^\ast,Q^\ast,u^\ast)$.
In the case of rational expressions, which are regular at $0$, we can clearly use their realizations, as they will be introduced in the subsequent Section \ref{sec:Realization}, instead of formal linear representations. Indeed, if $\br(x) = D + C(J-L_A(x))^{-1}B$ is any descriptor realization of a NC rational expression $r$ which is regular at $0$, then $\br^\ast(x) = D^\ast + B^\ast (J-L_{A^\ast})^{-1} C^\ast$ yields a descriptor realization of $r^\ast$.
\end{rem}

\subsection{Series Equivalence and Rational Functions}

\label{sec:series}
We conclude the parts of this paper devoted to background
on rational expressions with a brief description of
power series equivalence.
For that purpose, we restrict attention to regular rational expressions.

An example involving matrix-valued expressions that are realizations will be given in Remark \ref{rem:realizations_series}.

We shall consider
 formal power series expansions 
$$
    \sum_{w \in {\WOR}_g} r_w x^w
$$
of NC rational expressions around $0$. As an example, consider
the operation of inverting a polynomial. If $p$ is a NC polynomial
and $p(0)\ne 0$, write  $ p= p(0) - q$ where $q(0)=0$, then the
inverse $p^{-1}$ is the series expansion $r= \frac{1}{p(0)} \sum_k
(\frac{q}{p(0)})^k$.
Clearly, taking successive products, sums and inverses allows us
to obtain a NC formal power series expansion for any regular NC rational
expression.

We say that two regular NC rational expressions
$r_1$ and $r_2$ are
\textbf{power series equivalent}\index{rational expression!power series equivalence} rational expressions if their series
expansion around $0$ are the same.
For example, series expansion for the functions $r_1$
and $r_2$
in \eqref{eq:exr1r2}  are
\begin{equation}
\label{eq:serex}
\sum_{k = 0} x_1 (x_2 x_1)^k \ \ \ \text{and}  \ \ \
\sum_{k = 0}  (x_1 x_2)^k x_1.
\end{equation}
These are the same series, so $r_1$ and $r_2$
are power series equivalent.
It is shown in Lemma 2.2 of \cite{HMV06} that regular NC rational expressions are $\bbM(\bbK)$-evaluation equivalent if and only if they are power series equivalent in the above sense. Thus, we could alternatively define a regular NC rational function $\rf{r}$ to be an equivalence class of regular NC rational expressions with respect to power series equivalence. Accordingly, the \textbf{series expansion for $\rf{r}$}\index{NC rational function!regular!series expansion} is the series expansion of any representative.

Similar considerations hold for matrices
 of NC rational expressions.
Two regular $m_1\times m_2$ matrix-valued NC rational expressions
$r_1$ and $r_2$
each have a power series expansion around $0$
whose coefficients are $m_1\times m_2$ matrices.
These coefficients being equal 
define power series equivalence of $r_1$ and $r_2$,
thereby determining equivalence classes which characterizes regular matrix-valued NC rational functions.
In Proposition A.7 of \cite{HMV06} it is proved that power series equivalence and 
$\bbM(\RR)$-equivalence are the same for matrices of regular rational expressions.


\section{Realizations of NC Rational Expressions and Functions}
\label{sec:Realization}

This section begins with a review of the classical
theory of descriptor realizations for regular NC rational
functions tailored to future needs.
See the book \cite{BR84} for a more complete
exposition and the papers
\cite{B01} \cite{BMG05} for recent developments.
 From the existence of descriptor realizations,
a natural argument shows
that symmetric NC rational functions have
symmetric descriptor realizations. The section
finishes with uniqueness
results for symmetric descriptor realizations.

\textbf{Be aware that the ``NC'' will be suppressed from now in any term like ``NC rational expressions'', since we will work here solely in the noncommutative setting.}

\subsection{Descriptor Realizations}
  \label{subsec:descriptorrealizations}
  
Define a \textbf{descriptor realization of size $d_1 \times d_2$}\index{descriptor realization} to be a regular matrix-valued rational expression of the form 
\begin{equation}
\label{eq:descrN}
   \br = \br(x) = D + C ( J -  L_A(x) )^{-1} B \quad \text{with} \quad  L_A(x) = A_1 x_1 + \dots + A_g x_g
\end{equation}
where $A_j \in \bbR^\dxd$ for $j=1, \ldots , g$, $D \in M_{d_1\times d_2}(\bbR)$, $C \in M_{d_1 \times d}(\bbR)$ and $B \in M_{d \times d_2}(\bbR)$. Here $J$ denotes a $d \times d$ signature matrix,
namely, $J=J^{\transpose}$ and $J^2=I$. Here $d$ is called the dimension of the state space of the realization. We emphasize that at this point the $A_j$ are not required to be symmetric.

The same terminology is used in the complex case.

If $\br$ is a descriptor realization like in \eqref{eq:descrN}, seen as a matrix-valued rational expression, then Definition \ref{def:ev_dom-matval} says that its $\cA$-domain for any unital algebra $\cA$ over $\bbK$ is given by
\begin{equation}
\label{eq:dom_descriptor_realization}
\begin{aligned}
\dom_\cA(\br) &= \dom_\cA((J -  L_A(x))^{-1})\\
              &= \{X\in \cA^g|\ \text{$J \otimes 1_{\cA} - L_A(X)$ is invertible in $M_d(\bbK) \otimes \cA$}\}.
\end{aligned}              
\end{equation}
The tensor product notation (already used in $L_A(X)$) provides
a convenient way of expressing the evaluation
\begin{equation}
\label{eq:eval_descriptor_realization}
\br(X) = D \otimes 1_{\cA} + (C \otimes 1_{\cA})
[ J \otimes 1_{\cA} - L_A(X) ]^{-1} (B \otimes 1_{\cA})
\end{equation}
at $X\in\dom_\cA(\br)$.

In the following, we call a descriptor realization $\br$ like in \eqref{eq:descrN}
\begin{itemize}
 \item a \textbf{descriptor realization of the regular matrix rational function $\fr$}\index{descriptor realization!of a matrix rational function}, if $\fr$ is represented in the sense of Section \ref{subsubsec:matrix_rational_functions} by the regular matrix-valued rational expression $\br$.
 \item a \textbf{descriptor realization of the matrix $\ur$ of regular rational expressions}\index{descriptor realization!of a matrix of NC rational expressions} if $\ur$ and $\br$ are \textbf{$\bbM(\bbR)$-evaluation equivalent}\index{matrix-equivalent} in the sense that we have $$\ur(X) = \br(X) \qquad\text{for each $X \in \dom_{\bbM(\bbR)}(\ur) \cap \dom_{\bbM(\bbR)}(\br)$},$$ where we put $$\dom_{\bbM(\bbR)}(\ur) := \coprod_{n=1}^\infty \dom_{M_n(\bbR)}(\ur) \qquad\text{and}\qquad \dom_{\bbM(\bbR)}(\br) := \coprod_{n=1}^\infty \dom_{M_n(\bbR)}(\br).$$ 
\end{itemize}

A \df{symmetric descriptor realization} is a descriptor realization with
$$
  D=D^{\transpose}, \ \ B=C^{\transpose}, \ \ \text{$J$ and the $A_j$ are symmetric matrices.}
$$

Clearly, the rational function $\rf{r}$ corresponding to a symmetric
descriptor realization is a symmetric rational function.

A descriptor realization is called \textbf{monic}\index{descriptor realization!monic} provided $J=I$.

Associated to \eqref{eq:descrN}, we often consider its \textbf{Sys matrix}\index{descriptor realization!Sys matrix}, which is the (affine) linear pencil given by
\begin{equation}
\label{eq:linMat}
\Sys (J; A, B, C , D)(x) \ :=   
\bmat
 J - A_1 x_1 - \dots - A_g x_g   &    B\\
 C                               &    -D
 \emat.
\end{equation}
We notice that $\Sys(J; A, B, C, D)$ can be transformed into the matrix
$$\bmat D & C\\ B & -(J - A_1 x_1 - \dots - A_g x_g) \emat = - \Pi_1 \Sys(J; A, B, C, D) \Pi_2, \quad \text{where}\quad \Pi_i := \bmat 0 & I_{d}\\ -I_{d_i} & 0\emat,$$
whose Schur complement is then $\br$, for $\br$ being the descriptor realization given in \eqref{eq:descrN}; in this form, the Sys matrix will appear also in Theorem \ref{thm:mainrep2}, more precisely in \eqref{eq:shifted_linearization}.

Of course, one can write \eqref{eq:descrN} also in \textbf{monic form}\index{descriptor realization!monic form},
namely
\begin{equation}
\label{eq:descrNJ}
   \dr{r} = D +  C ( I -  L_{JA}(x) )^{-1} JB,
\end{equation}
where we abbreviate the associated Sys matrix $\Sys(I; JA, JB, C, D)$ to $\Sys(JA, JB, C, D)$.

Let us point out that, while we considered here primarily the real case, the above definitions clearly make sense also in the complex case. Thus, we will use below the terminology of descriptor realizations for both the real and the complex situation.

\subsubsection{Examples}

\begin{example}
  \label{ex:descriptor}
   Here is an example of a $1\times 1$
   rational expression in two variables
  obtained as a descriptor realization.
\begin{equation*}
   \dr{r}= \begin{pmatrix} 1 & 0 \end{pmatrix}
     \left ( I- \begin{pmatrix} 1& 0\\ 0& 1 \end{pmatrix} x_1 -
\begin{pmatrix} 0& 1 \\ 1 & 0\end{pmatrix}x_2\right )^{-1}
    \begin{pmatrix} 1 \\ 0 \end{pmatrix}
    = \begin{pmatrix} 1 & 0 \end{pmatrix} \begin{pmatrix} 1-x_1 &
-x_2 \\ -x_2 & 1-x_1\end{pmatrix}^{-1}
      \begin{pmatrix} 1 \\ 0 \end{pmatrix}.
\end{equation*} From 
Example \ref{ex:matrat} we see that
an symmetric rational expression representing
the same rational function as $\dr{r}$
is
$$
 w_{11}= (1-x_1)^{-1} + (1-x_1)^{-1}x_2 \left (
(1-x_1)-x_2(1-x_1)^{-1}x_2\right)^{-1}x_2(1-x_1)^{-1}.
$$
\end{example}

\begin{rem}\label{rem:realizations_series}
In view of Section \ref{sec:series}, computing the formal power series expansion, and thus
the equivalence class (rational function) to which a given descriptor realization belongs, is straightforward. Indeed, if we suppose that $\br$ is of the form \eqref{eq:descrN}, we first bring it to monic form as in \eqref{eq:descrNJ} and then observe that
\begin{equation*}
\dr{r} = C (I - J L_A(x))^{-1} JB \ \sim \ \sum_{n\ge 0} C (J L_A(x))^n J B = C^JB + \sum_{j=1}^g C J A_j J B x_j + \dots .
\end{equation*}
Note that this uses $A_j x_j JB = A_j JB x_j$ and hence the convention \eqref{eq:scalar_matrices_commute_with_indeterminates} we agreed upon in Section \ref{subsec:NC_Linear_Pencils}.
\end{rem}

\begin{example}
  \label{ex:descriptor2}
We return to the descriptor realization in
Example \ref{ex:descriptor}.

Note it is straightforward
to compute the power series expansion.
Also
the domain $\dom_{M_n(\bbK)}(\br)$ of the rational expression $\br$ consists
by definition
exactly of those
$X=(X_1,X_2) \in M_n(\bbK)^2$, for
which
$$
   \begin{pmatrix} I-X_1 & -X_2 \\ -X_2 & I-X_1 \end{pmatrix}
$$
is invertible; for such $X$
$$
  \br(X) =
\begin{pmatrix} I & 0 \end{pmatrix}
      \begin{pmatrix} I-X_1 & -X_2 \\ -X_2 & I-X_1 \end{pmatrix}^{-1}
    \begin{pmatrix} I \\ 0 \end{pmatrix}.
$$
\end{example}

\subsubsection{Minimality}

A descriptor realization $\br$ of the form \eqref{eq:descrN} (with matrices over the field $\bbK$ of real or complex numbers) is
\df{controllable} if the \df{controllable space} defined by
$$
\cS :=  \mbox{span} \; \{  \mbox{Range} \ (JA)^w JB : \mbox{all words} \ w \in \ \cW_g \}
$$
is all of $\mathbb \bbK^d$.
It is \df{observable} provided the \df{unobservable space} 
$$
\cQ := \{ v \in \bbK^d :   C (JA)^w v = 0 \  \mbox{for all words} \ w   \in \ \cW_g  \}  $$
is $\{0\}$.
An important property is that
both spaces are invariant under $JA_i$ for
each $i$.
Observability can be expressed 
as controllability for the transpose system,
since 
$$
\cQ^\perp = \mbox{span} \; \{  \mbox{Range} \ ((JA)^w)^{\transpose} C^{\transpose} : \mbox{all words} \ w \in \ \cW_g \}.
$$
Likewise, controllability is the same
as observability for the transpose system.
We say that the descriptor realization is \df{minimal} 
if it is both
observable and controllable. 
We emphasize that since the system has finite dimensional
``statespace'' $\bbK^d$,
only finitely many words $w \in \cW_g$
are needed in the formulas 
to produce $\cS$ and $\cQ$.

\subsection{Properties of Descriptor Realizations}
\label{subsec:descriptorsexist}
	
That regular rational functions regular have descriptor realizations can be found in \cite{BR84}. Moreover, as we will see in Lemma \ref{lem:symR} \eqref{lem:symR0} below, any two minimal monic descriptor realizations for the same regular rational function $\fr$ that have the same feed through term $D$, say
\begin{equation*}
\label{eq:rJisI}
  \begin{split}
    \dr{r}=&D+C(I-L_A(x))^{-1}B, \\
    \dr{\widetilde{r}}=& D +\widetilde{C}(I-L_{\widetilde{A}}(x))^{-1}\widetilde{B},
  \end{split}
\end{equation*}
    are \df{similar} in the sense that there exists an invertible matrix $S$ such that
\begin{equation}
\label{eq:sim}
      SA_j=\widetilde{A}_jS, \ \ \  SB=\widetilde{B}, \ \ \ C=\widetilde{C}S.
\end{equation}
  The $S$ is known as a \df{similarity transform}.

   Lemma \ref{lem:symR} also exploits the symmetry implicit
   in a symmetric regular rational function to show,
   by appropriate choice of similarity transform,
   that any symmetric regular rational function $\rf{r}$ has a
   minimal descriptor realization that is symmetric.

\begin{lem} [Lemma 4.1 \cite{HMV06}]
  \label{lem:symR}\
	
\begin{enumerate}
  \item
   \label{lem:symR0}
    \begin{itemize}
     \item[(a)] Any descriptor realization is (more precisely, determines) a matrix-valued rational function which is regular at $0$. Conversely, each $d_1 \times d_2$ matrix-valued rational function $\rf{r}$ regular at $0$, has a minimal descriptor realization (which could be taken to be monic) with $0$ feed through term ($D=0$).
      \item[(b)] Moreover, any two minimal descriptor realizations for $\rf{r}$ with the same feed through term are similar via a unique similarity transform. 
      \item[(c)] A descriptor realization $\br$ for $\fr$ whose state space has the smallest possible dimension is minimal. Conversely, a minimal descriptor realization $\br$ of $\fr$ has smallest possible state space dimension among all descriptor realizations of $\fr$ that have the same feed through term as $\br$.
     \end{itemize}
  \item
   \label{lem:symR1} Any matrix-valued rational function regular at $0$ with a symmetric descriptor realization is a symmetric rational function.
  \item
    \label{lem:symR2} If $\rf{r}$ is a symmetric matrix-valued rational function regular at $0$, then it has a minimal descriptor realization that is symmetric.
\end{enumerate}
\end{lem}

Lemma \ref{lem:symR} \eqref{lem:symR0} (b) is often called the \df{state space similarity theorem}.

\subsubsection{Cutting down to get a minimal system}
\label{sec:cutdown}

A construction from classical one variable system theory,
which dates back at least to Kalman \cite{K63},
also works well in this much more general context,
cf. \cite{CR}, \cite{BMG05}.
It is that of cutting down the descriptor realization $\br=D+C(J-L_A(x))^{-1}B$ 
of a regular matrix-valued rational expression $r$ to controllability and observability spaces
thereby obtaining 
a minimal realization:

We first write the descriptor realization $\br=D+C(J-L_A(x))^{-1}B$ of $r$ in monic form $\br = D +  C ( I -  L_{JA}(x) )^{-1} JB$ according to \eqref{eq:descrNJ}. By the cutting down to the controllability space
$\cS$
we get 
a new realization $\hA, JB, \hC, D$
whose state space is $\cS$
with
$$ \hA_i =  {(JA_i) }_{|_\cS} \qquad\text{and}\qquad \hC:  = C_{|_\cS}.
$$
Thus the system $ JA, JB,C,D$
has the following block decomposition with respect to the 
subspace decomposition
$ \bbK^d= \cS + \cS^\perp$
$$	 JA =
\begin{pmatrix}
	 \hA & 	 A_{12}  \\	 0  &	 A_{22}
\end{pmatrix}	 
\qquad \text{and} \qquad
C = \bmat
\hC & C_2
\emat.
$$ 
While the system $\hA, JB, \hC, D$ represents the
same rational function as the original system,
it may not be observable.
However we can repeat the dual of this construction
on $\hA, JB, \hC, D$ and decompose
$\cS = \hat\cQ + \hat \cQ^\perp$. 
This results in a
minimal monic descriptor
 system $\check A, \check B, \check C, D$
which represents the same rational function
(not necessarily the same rational expression)
as $JA$ and it also yields a block decomposition of the original system. We summarize these observations in the next lemma.

\begin{lem}
\label{lem:cutdown0}
Let $\br=D+C(J-L_A(x))^{-1}B$ be any descriptor realization of a regular matrix-valued rational expression $r$.
With respect to the subspace decomposition
$ \hat\cQ^\perp +  \; \hat \cQ + \; \cS^\perp$
of $\bbK^d$, the monic system $JA, JB, C,D $ for $r$ has the block decomposition
\begin{equation}
\label{eq:basiccut}
 JA =
\begin{pmatrix}
\hat A_{11}&     \hat A_{12} &      A_{12}^1  \\
0 &     \check A &      A_{12}^2  \\
0 & 0  &     A_{22}
\end{pmatrix},
\qquad
JB=
\bmat
\hat B_1 \\
\check B
\\
0
\emat,
\quad
C = \bmat  0  & \check C  &  C_2
\emat,
\end{equation}
where the monic system $\check A, \check B, \check C, D$ provides a minimal descriptor realization for $r$.
\end{lem}

\begin{rem}
\label{rem:hollow}
In preparation for what comes later in Section \ref{sect:2-4}
we record an observation about the special case where
a monic system $ JA, JB, C, 0 $  is a realization of $0$, then
$C \cS =0 $ and in terms of cutdowns 
$
C = \bmat  0  & 0  &  C_2 
\emat
$.
Thus
$\Sys( JA, JB, C, 0)$ has the form
\begin{equation}
\label{eq:sysmat}
\Sys( JA, JB, C, 0) = \widetilde{\Pi}_1
\begin{pmatrix}
\Lam_{\hat A_{11} }   &	\Lam_{ \hat A_{12}} & B_1 
                &  \Lam_{ A_{12}^1 }   \\	 
0 &	\Lam_{ \check A} &  \check B &	\Lam_{ A_{12}^2  }     \\
0 & 0  &	0 & \Lam_{A_{22}  }          \\
0 & 0  & 0 & \hC 
\end{pmatrix}	 \Pi \widetilde{\Pi}_2
\end{equation}
where $\Pi$ permutes the last two columns and $\widetilde{\Pi}_i$ is the permutation $\begin{pmatrix} 0 & I_{d_i}\\ I_d & 0\end{pmatrix}$.
This is a block $4 \times 4$ matrix.
A block $n \times n$ matrix which contains a $\alpha \times \beta$ rectangle of zeroes is called \df{hollow}
(the terminology of P.M. Cohn), if $\alpha + \beta > n$. For matrix \eqref{eq:sysmat} this count is $5>4$, so it is hollow, a fact which will be used later in Section \ref{sect:2-4}.
\end{rem}

\subsubsection{Uniqueness of Symmetric Descriptor Realizations}
\label{sec:uniqdescr} 

There is a useful refinement of the state space similarity theorem, Lemma \ref{lem:symR} \eqref{lem:symR0} (b), for symmetric descriptor realizations.

\begin{prop} [Proposition 4.3 \cite{HMV06}]
  \label{prop:uniquedescriptor}
   If
  \begin{equation*}
   \dr{r}= D+ C (J-L_A(x))^{-1} C^{\transpose}
    \ \ \ \ \
     \mbox{and}
     \ \ \ \ \
    \widetilde{\dr{r}}= D+\widetilde{C}
       (\widetilde{J}-L_{\widetilde{A}}(x))^{-1} \widetilde{C}^{\transpose}
  \end{equation*}
    are both minimal 
    symmetric descriptor realizations for the same $d_1 \times d_2$ matrix of
    regular rational functions
    (with the same symmetric feed through term $D$),
    then there is an invertible
      similarity transform $S$ between the two
    systems; it satisfies   $S^{T}\widetilde{J}S=J$ and
 $$SJA_j=\tJ  \widetilde{A}_j S    \qquad    SJC^\transpose=\tJ \tC^\transpose \qquad
                                         C=\widetilde{C}S.$$

    Thus, if $J=I$, then $\wt{J}=I$
    too and $S$ is unitary.
    In particular any two monic ($J=I$) symmetric
    minimal descriptor realizations
    with the same feed through term  for the same
    matrix rational function are unitarily equivalent.
\end{prop}

\begin{proof}
We shall recall the proof from \cite{HMV06}, since Proposition 4.3 there was only stated for noncommutative scalar expressions. However, as we now see, the proof works for matrix rational expressions.
First put $  \dr{r}$ and $\widetilde{  \dr{r}}$
in the  form \eqref{eq:rJisI} by multiplying appropriately
by $J$ and $\tJ$ respectively.
  Since both $\dr{r}$ and $\widetilde{\dr{r}}$ represent
  the same rational function
   (and share the feed through term $D$).
  From controllability and observability (from the state space
  similarity theorem)
  we know that there is an invertible 
  similarity transform  $S$; it satisfies \eqref{eq:sim}.
  Thus
  $$SJA_j=\widetilde{J} \widetilde{A}_j S \qquad 
  SJC^\transpose=\widetilde{J}\widetilde{C}^\transpose \qquad
  C=\widetilde{C}S$$
   Hence,
  $S (JA)^\alpha JC^\transpose= (\widetilde{J}\widetilde{A})^\alpha \widetilde{J}\widetilde{C}^\transpose$  and
  $C  (JA)^\alpha = 
  \tC (\widetilde{J}\widetilde{A})^\alpha S$ 
   for all words  $\alpha$.
   
   Since the $A_j$ and $\wt{A}_j$ are symmetric, it follows that
\begin{equation}
  \label{eq:ud2}
    CJ (AJ)^{\beta^\transpose} S^\transpose \wt{J} S (JA)^\alpha JC^\transpose=
   \widetilde{C}\wt{J} (\widetilde{A}\wt{J})^{\beta^\transpose} \tJ
    (J\widetilde{A})^\alpha \wt{J}\widetilde{C}^\transpose
\end{equation}
which equals
$  \tC (\tJ\tA)^{\beta^\transpose} 
    (\tJ\widetilde{A})^\alpha \tJ\widetilde{C}^\transpose
$.
The power series equivalence (see Section \ref{sec:series}) of the rational expressions
 $\dr{r}$ and $\widetilde{\dr{r}}$ implies
\begin{equation}
  \label{eq:ud1}
   C(JA)^wJC^{\transpose}=\wt{C}(\wt{J}\wt{A})^w \wt{J}\wt{C}^{\transpose}
\end{equation}
   for all words $w$.
Which we use to obtain
$
    CJ (AJ)^{(\beta^\transpose)} S^\transpose \wt{J} S (JA)^\alpha JC^\transpose=
   C (J A)^{(\beta^\transpose)} 
    (J  A)^\alpha J C^\transpose.
$
Therefore
\begin{equation}
    CJ (AJ)^{(\beta^\transpose)} \ S^\transpose \wt{J} S \ (JA)^\alpha JC^\transpose
    = C J(AJ)^{(\beta^\transpose)} \ J \
    (J  A)^\alpha J C^\transpose, so
\end{equation}
the controllability and observability 
implies $S^\transpose\tJ S= J$.
\end{proof}

Beware, if $J \not =I$ the cutdown system
$\check A, \check B, \check C, D$ 
 in  Lemma \ref{lem:cutdown0} while monic
often will not be symmetric. 
One can, however, symmetrize it as in Lemma 4.2 in Section 4.3 of \cite{HMV06}, notably without changing the size of the matrices $\check A, \check B, \check C, D$ and even without changing the feed through term $D$; this construction underlies Lemma \ref{lem:symR} \eqref{lem:symR2}.

This combines with the above to yield that minimal
realizations have maximal $\cA$-domains, as we will state formally in Lemma \ref{lem:domMin}.


\section{Explicit Algorithm for a Realization}
\label{sec:Algorithm}

Much of what is stated in Sections \ref{sec:evaluations} and \ref{sec:FP} can be understood without mastering this section.
Hence on first reading one may want to skip to Section \ref{sec:evaluations}. Of course, to understand all of the proofs, one must read Section \ref{sec:Algorithm}.

Here we present an algorithm for constructing a realization
of a rational expression $r$ while keeping a close watch
on evaluation properties of both $r$ and its realization.
We prove existence of realizations which have excellent domain
and evaluation properties with respect to any unital algebra $\cA$
without any further constraints; see Section \ref{subsec:regular_case}.

One motivation is that in our free probability applications in Section
\ref{sec:FP}
 it is important to know that the $\cA$-domain of the considered realization is not smaller than the $\cA$-domain of the rational expression we are interested in. Hence we need some control over this. By \cite{KVV09} it follows that in the case
$\cA=\bbM(\CC)$ the minimal realization has the largest possible domain. 
In Section \ref{subsec:minimality_implies_maximal_domain}, we will show the validity of this for more general $\cA$. In Section \ref{subsec:minimal_realizations}, this will be combined with the concrete realization algorithm that we present here, yielding that minimal realizations have excellent evaluation properties.

Our algorithmic construction of realizations is not restricted to the regular case, but works for any rational expression.
This means that this algorithm will also apply to the general situation of the full free field; see Sections \ref{subsec:FLR} and \ref{subsec:FLR_matval}.

Indeed, our construction is closely related to similar considerations in the context of the universal skew field of noncommutative rational functions \cite{Co71,Mal78, Co06}. 
(In the context of regular expressions this goes in principle back to the work of Kleene and Sch\"utzenberger,
c.f. \cite{K56,S61,S65}.)   
Also an algorithm for the regular case, a bit less general than here,
  appears in \cite{Sthesis} Chapter 5, and is implemented in 
  {\it NCAlgebra} a noncommutative algebra package which runs 
  under Mathematica.

The main expedience of assuming regularity is that the ``cutting down'' arguments we saw in Sections \ref{sec:cutdown} and \ref{sec:uniqdescr} behave well. Without regularity complications arise. This case is treated in \cite{Vol15};
we leave the natural question to future research, whether these tools are also suitable for our purposes like the regular ones we use here. 

In view of this, we should note that our algorithm produces realizations that have good evaluation properties but are typically not minimal; analogous cutting down arguments that work directly in our setup are still missing. 
Nevertheless, arguments not involving cutdowns work well even without assuming regularity and they can be treated without using results of \cite{Vol15}.

Below, in Section \ref{sec:evaluations}, we will see that evaluations of general realizations only behave well under some additional condition on the algebra $\cA$. We note that in the context of this section no further assumptions on $\cA$ are necessary, since this is only relevant if one is moving algebraically between different rational expressions of the same rational function; here we keep track of domains and evaluations through the construction to show that the obtained realizations produce valid identities under evaluation on any $\cA$.

\subsection{Formal Linear Representations of NC Rational Expressions}
\label{subsec:FLR}

We point out that the terminology we are going to use here is distinct from the realization language used in other parts of this paper. This marks the transition from the regular context to more algebraic considerations without regularity assumptions and allows us to distinguish properly between these two settings.

As before, we will work here with rational expressions in the variables $x=(x_1,\dots,x_g)$ over the field $\bbK$ of real or complex numbers.

\begin{definition}\label{def:rep}
Let $r$ be a rational expression in the variables $x=(x_1,\dots,x_g)$ over $\bbK$. A \textbf{formal linear representation $\rho=(u,Q,v)$ of $r$}\index{formal linear representation} consists of
\begin{itemize}
 \item an affine linear pencil $$Q = Q^{(0)} + Q^{(1)} x_1 + \dots + Q^{(g)} x_g$$ for matrices $Q^{(0)},Q^{(1)},\dots,Q^{(g)}\in M_n(\bbK)$ of some dimension $n$,
 \item a $1\times n$-matrix $u$ over $\bbK$,
 \item and a $n\times 1$-matrix $v$ over $\bbK$,
\end{itemize}
such that the following condition is satisfied:
\begin{quote}
For any unital $\bbK$-algebra $\cA$, we have that
$$\dom_\cA(r) \subseteq \dom_\cA(Q^{-1})$$
and it holds true for any $(X_1,\dots,X_g) \in \dom_\cA(r)$ that
$$r(X_1,\dots,X_g) = - u Q(X_1,\dots,X_g)^{-1} v.$$
\end{quote}
\end{definition}

The main contribution of this section is the following algorithm by which we ensure that a formal linear representation exists for any rational expression. Note that the definition of a formal linear representation requires that the domain of definition of the representation includes the domain of definition of the rational expression.
 
\begin{thm}\label{thm:existence-flr}
For each rational expression in $x=(x_1,\dots,x_g)$ over $\bbK$ (not necessarily regular at zero) there exists a formal linear representation in the sense of Definition \ref{def:rep}.
\end{thm}

Later, we will also address a symmetric (resp. selfadjoint) version and even a generalization of these result for matrices of rational expressions; see, in particular, Theorem \ref{thm:rep_sa_exist} and Theorem \ref{thm:FLR_matrices_of_rational_expressions}.

The proof of Theorem \ref{thm:existence-flr} is provided by the following algorithm for producing such a formal linear representation. Recall from Section \ref{item:intorat3} that any rational expression is built from scalars $\lambda\in\bbK$ and the variables $x_1,\dots,x_g$ by applying iteratively the arithmetic operations $+$, $\cdot$, and $\cdot^{-1}$.

\begin{algorithm}\label{alg:explicit_rep}
Let $r$ be a rational expression in the variables $x=(x_1,\dots,x_g)$ over the field $\bbK$. A formal linear representation $\rho=(u,Q,v)$ of $r$ can be constructed by using successively the following rules:
\begin{itemize}
 \item[(i)] For any affine linear polynomial $$\lambda(x) = \lambda_0 + \lambda_1 x_1 + \ldots + \lambda_g x_g$$ with coefficients $\lambda_0,\lambda_1,\ldots,\lambda_g\in\bbK$, a formal linear representations is given by
\begin{equation}\label{eq:rep_init}
\rho_{\lambda(x)} := \bigg(\begin{pmatrix} 0 & 1\end{pmatrix}, \begin{pmatrix} \lambda(x) & -1\\ -1 & 0\end{pmatrix}, \begin{pmatrix} 0\\ 1 \end{pmatrix}\bigg). 
\end{equation}
respectively.
 \item[(ii)] If $\rho_1=(u_1,Q_1,v_1)$ and $\rho_2=(u_2,Q_2,v_2)$ are formal linear representations for the rational expressions $r_1$ and $r_2$, respectively, then
\begin{equation}\label{eq:rep_sum}
\rho_1 \oplus \rho_2 := \bigg(\begin{pmatrix} u_1 & u_2\end{pmatrix}, \begin{pmatrix} Q_1 & 0\\ 0 & Q_2\end{pmatrix}, \begin{pmatrix} v_1\\ v_2 \end{pmatrix}\bigg)
\end{equation}
gives a formal linear representation of $r_1 + r_2$.
 \item[(iii)] If $\rho_1=(u_1,Q_1,v_1)$ and $\rho_2=(u_2,Q_2,v_2)$ are formal linear representations for the rational expressions $r_1$ and $r_2$, respectively, then
\begin{equation}\label{eq:rep_prod}
\rho_1 \odot \rho_2 := \bigg(\begin{pmatrix} 0 & u_1\end{pmatrix}, \begin{pmatrix} v_1u_2 & Q_1\\ Q_2 & 0\end{pmatrix}, \begin{pmatrix} 0\\ v_2 \end{pmatrix}\bigg)
\end{equation}
gives a formal linear representation of $r_1 \cdot r_2$.
 \item[(iv)] If $\rho=(u,Q,v)$ is a formal linear representation of $r$, then
\begin{equation}\label{eq:rep_inv}
\rho^{-1} := \bigg(\begin{pmatrix} 1 & 0\end{pmatrix}, \begin{pmatrix} 0 & u\\ v & -Q\end{pmatrix}, \begin{pmatrix} 1\\ 0 \end{pmatrix}\bigg)
\end{equation}
gives a formal linear representation of $r^{-1}$.
\end{itemize}
\end{algorithm}

Note that the operations \eqref{eq:rep_init}, \eqref{eq:rep_sum}, \eqref{eq:rep_prod}, and \eqref{eq:rep_inv}, which we described in Algorithm \ref{alg:explicit_rep}, have to be understood on the level of linear pencils as described in Definition \ref{def:linear_pencils}.

The proof that the rules (i) -- (iv) given in Algorithm \ref{alg:explicit_rep} are indeed correct, will be given in Section \ref{subsubsec:algorithm_proof}.

\begin{example}
We consider the rational expressions
$$r_1=(x_1 x_2)^{-1} \qquad\text{and}\qquad r_2=x_2^{-1} x_1^{-1}$$
By applying Algorithm \ref{alg:explicit_rep}, we obtain for $r_1$ the formal linear representation
$$\rho_1 = \bigg(\begin{pmatrix} 1 & 0 & 0 & 0 & 0\end{pmatrix}, \begin{pmatrix} 0 & 0 & 0 & 0 & 1\\ 0 & 0 & 0 & -x_1 & 1\\ 0 & 0 & -1 & 1 & 0\\ 0 & -x_2 & 1 & 0 & 0\\ 1 & 1 & 0 & 0 & 0\end{pmatrix}, \begin{pmatrix} 1\\ 0\\ 0\\ 0\\ 0 \end{pmatrix}\bigg)$$
and for $r_2$ the formal linear representation
$$\rho_2 = \bigg(\begin{pmatrix} 0 & 0 & 0 & 1 & 0 & 0\end{pmatrix}, \begin{pmatrix} 1 & 0 & 0 & 0 & 0 & 1\\ 0 & 0 & 0 & 0 & -x_2 & 1\\ 0 & 0 & 0 & 1 & 1 & 0\\ 0 & 0 & 1 & 0 & 0 & 0\\ 0 & -x_1 & 1 & 0 & 0 & 0\\ 1 & 1 & 0 & 0 & 0 & 0\end{pmatrix}, \begin{pmatrix} 0\\ 0\\ 0\\ 1\\ 0\\ 0 \end{pmatrix}\bigg).$$

This highlights the computational disadvantage of Algorithm \ref{alg:explicit_rep}, that roughly speaking the dimension of the linear pencil $Q$ of a formal linear representation $\rho=(u,Q,v)$ increases rapidly with the complexity of the rational expression $r$ that it represents. Clearly, since the rational expressions $r_1$ and $r_2$ in the example above are rather simple, we would expect that there are other formal linear representations of smaller dimensions. Unfortunately, since $r_1$ and $r_2$ are both not regular, we cannot use the representation machinery to cut down our realizations to minimal ones. One expedient could be to use the analogous but more general machinery that was invented recently in \cite{Vol15}; we leave this to future research.
A far less sophisticated approach is the following ad hoc construction: because any formal linear representation $\rho=(u,Q,v)$ of a rational expression $r$ can clearly be transformed by
$$S \cdot \rho \cdot T := (uT, SQT, Sv),$$
for any choice of invertible matrices $S,T\in M_n(\bbK)$, into another formal linear representation of $r$, we can try, after having $\rho=(u,Q,v)$ arranged as
$$\begin{array}{c|c} & u\\ \hline v & Q\end{array},$$
to bring this array into the form
$$\begin{array}{c|c c} & \widetilde{u} & u'\\ \hline \widetilde{v} & \widetilde{Q} & 0\\ v' & 0 & Q'\end{array},$$
by acting by elementary row and column operations only on $Q$, while bookkeeping their effect in the first row and column, respectively. If it happens in this case that $(u',Q',v')$ is a formal linear representation of $0$, we can just remove this part, which means that $\widetilde{\rho} = (\widetilde{u},\widetilde{Q},\widetilde{v})$ gives another formal linear representation of $r$; however, we do not know if such a reduction is always possible, and even if this would be the case, one cannot be sure that one reaches eventually a formal linear representation of minimal size.

In our situation, we can show by using this method that
$$\widetilde{\rho}_1 = (\widetilde{u}_1,\widetilde{Q}_1,\widetilde{v}_1) = \bigg(\begin{pmatrix} 1 & 0\end{pmatrix}, \begin{pmatrix} 0 & x_1\\ x_2 & 1\end{pmatrix}, \begin{pmatrix} 1\\ 0 \end{pmatrix}\bigg)$$
gives another formal linear representation of $r_1$ and that
$$\widetilde{\rho}_2 = (\widetilde{u}_2,\widetilde{Q}_2,\widetilde{v}_2) = \bigg(\begin{pmatrix} 0 & 1\end{pmatrix}, \begin{pmatrix} 1 & -x_2\\ -x_1 & 0\end{pmatrix}, \begin{pmatrix} 0\\ 1 \end{pmatrix}\bigg)$$
gives another formal linear representation of $r_2$.

It is easy to see that the linear pencils $Q_1,Q_2$ satisfy the relation
$$Q_1 = U Q_2 U^{-1} \qquad\text{where}\qquad U := \begin{pmatrix} 0 & 1\\ -1 & 0\end{pmatrix}.$$
Thus, we have $\dom_\cA(Q_1^{-1}) = \dom_\cA(Q_2^{-1})$ for any unital algebra $\cA$, and by using the Schur complement formula, we see that $Q_1(X_1,X_2)$ (and hence $Q_2(X_1,X_2)$) is invertible in $M_2(\bbK) \otimes_\bbK \cA$ for some $(X_1,X_2)\in\cA^2$, if and only if $X_1 X_2$ is invertible in $\cA$. In other words, we have
$$\dom_\cA(r_2) \subsetneq \dom_\cA(r_1) = \dom_\cA(Q_1^{-1}) = \dom_\cA(Q_2^{-1}).$$
\end{example}

\subsubsection{Proof of Rules in Algorithm \ref{alg:explicit_rep}}
\label{subsubsec:algorithm_proof}

First of all, we examine the validity of rule (i). This is the content of the following lemma.

\begin{lem}\label{lem:trivial_real}
Consider a rational expression of the form
$$\lambda(x) = \lambda_0 + \lambda_1 x_1 + \dots + \lambda_g x_g$$
with $\lambda_0,\lambda_1,\dots,\lambda_g\in\CC$. Then a formal linear representation of $\lambda$ is given by
$$\rho_{\lambda(x)} := \bigg(\begin{pmatrix} 0 & 1\end{pmatrix}, \begin{pmatrix} \lambda(x) & -1\\ -1 & 0\end{pmatrix}, \begin{pmatrix} 0\\ 1 \end{pmatrix}\bigg).$$
\end{lem}

\begin{proof}
Write $\rho=(u,Q,v)$. First of all, we note that
$$Q = \begin{pmatrix} \lambda_0 & -1\\ -1 & 0\end{pmatrix} + \begin{pmatrix} \lambda_1 & 0\\ 0 & 0\end{pmatrix} x_1 + \dots + \begin{pmatrix} \lambda_g & 0\\ 0 & 0\end{pmatrix} x_g.$$
Now, consider any unital complex algebra $\cA$. We observe that the matrix $Q(X)$ is invertible for any $X=(X_1,\dots,X_g) \in \dom_\cA(r) = \cA^g$ with
$$Q(X)^{-1} = \begin{pmatrix} 0 & -1\\ -1 & -\lambda(X)\end{pmatrix}.$$
Hence $X\in \dom_\cA(Q^{-1})$ and furthermore $- u Q(X)^{-1} v = \lambda(X)$, which completes the proof that $\rho$ is a formal linear representation of $\lambda$ in the sense of Definition \ref{def:rep}.
\end{proof}

Next, we give a lemma that justifies the rules (ii) and (iii).

\begin{lem}\label{lem:sum_prod_real}
Let $\rho_1=(u_1,Q_1,v_1)$ and $\rho_2=(u_2,Q_2,v_2)$ be formal linear representations of rational expressions $r_1$ and $r_2$, respectively. Then the following statements hold true:
\begin{itemize}
 \item A formal linear representation of $r_1 + r_2$ is given by $$\rho_1 \oplus \rho_2 := \bigg(\begin{pmatrix} u_1 & u_2\end{pmatrix}, \begin{pmatrix} Q_1 & 0\\ 0 & Q_2\end{pmatrix}, \begin{pmatrix} v_1\\ v_2 \end{pmatrix}\bigg).$$
 \item A formal linear representation of $r_1 \cdot r_2$ is given by $$\rho_1 \odot \rho_2 := \bigg(\begin{pmatrix} 0 & u_1\end{pmatrix}, \begin{pmatrix} v_1u_2 & Q_1\\ Q_2 & 0\end{pmatrix}, \begin{pmatrix} 0\\ v_2 \end{pmatrix}\bigg).$$
\end{itemize}
\end{lem}

\begin{proof}
For any unital complex algebra $\cA$, consider $X=(X_1,\dots,X_g) \in \dom_\cA(r_1+r_2) = \dom_\cA(r_1) \cap \dom_\cA(r_2)$. Since $\rho_1$ and $\rho_2$ are both formal linear representations, we have
$$r_1(X) = - u_1 Q_1(X)^{-1} v_1 \qquad\text{and}\qquad r_2(X) = - u_2 Q_2(X)^{-1} v_1.$$

For $\rho_1 \oplus \rho_2 = (u,Q,v)$, this means in particular that the matrix
$$Q(X) = \begin{pmatrix} Q_1(X) & 0\\ 0 & Q_2(X)\end{pmatrix}$$
is invertible, which shows $X \in \dom_\cA(Q^{-1})$, and moreover allows us to check
\begin{align*}
- u Q(X)^{-1} v
&= - \begin{pmatrix} u_1 & u_2\end{pmatrix} \begin{pmatrix} Q_1(X)^{-1} & 0\\ 0 & Q_2(X)^{-1}\end{pmatrix} \begin{pmatrix} v_1\\ v_2\end{pmatrix}\\
&= - u_1 Q_1(X)^{-1} v_1 - u_2 Q_2(X)^{-1} v_2\\
&= r_1(X) + r_2(X).
\end{align*}
Since $X\in\dom_\cA(r_1+r_2)$ was arbitrarily chosen, we conclude that $\rho_1 \oplus \rho_2$ is a formal linear representation of $r_1+r_2$.

Similarly, if we consider $X \in \dom_\cA(r_1 \cdot r_2) = \dom_\cA(r_1) \cap \dom_\cA(r_2)$, we obtain for $\rho_1 \odot \rho_2 = (u,Q,v)$ the invertibility of the matrix
$$Q(X) = \begin{pmatrix} v_1u_2 & Q_1(X)\\ Q_2(X) & 0\end{pmatrix}.$$
In fact, one can convince oneself by a straightforward computation that more precisely
$$Q(X)^{-1} = \begin{pmatrix} 0 & Q_2(X)^{-1}\\ Q_1(X)^{-1} & - Q_1(X)^{-1} v_1u_2 Q_2(X)^{-1}\end{pmatrix}.$$
This proves $X \in \dom_\cA(Q^{-1})$ and allows us to check
\begin{align*}
-u Q(X)^{-1} v
&= - \begin{pmatrix} 0 & u_1\end{pmatrix} \begin{pmatrix} 0 & Q_2(X)^{-1}\\ Q_1(X)^{-1} & - Q_1(X)^{-1} v_1u_2 Q_2(X)^{-1}\end{pmatrix} \begin{pmatrix} 0\\ v_2 \end{pmatrix}\\
&=  u_1 Q_1(X)^{-1} v_1 u_2 Q_2(X)^{-1} v_2\\
&= r_1(X) r_2(X).
\end{align*}
Since $X\in\dom_\cA(r_1\cdot r_2)$ was again arbitrarily chosen, we may conclude now that $\rho_1 \odot \rho_2$ gives as stated a formal linear representation of $r_1 \cdot r_2$.
\end{proof}

Finally, concerning rule (iv) of Algorithm \ref{alg:explicit_rep}, we show the following lemma.

\begin{lem}\label{lem:inv_real}
Let $\rho=(u,Q,v)$ be a formal linear representation of a rational expression $r$. Then
$$\rho^{-1} := \bigg(\begin{pmatrix} 1 & 0\end{pmatrix}, \begin{pmatrix} 0 & u\\ v & -Q\end{pmatrix}, \begin{pmatrix} 1\\ 0 \end{pmatrix}\bigg)$$
gives a formal linear representation of $r^{-1}$.
\end{lem}

\begin{proof}
Take any $X\in\dom_\cA(r^{-1})$, which means by definition that $X\in\dom_\cA(r)$ and that $r(X)\in\cA$ is invertible. Since $\rho$ is assumed to be a formal linear representation of $r$, this ensures the invertibility of $Q(X)$. Hence, the Schur complement formula tells us that the matrix
$$\begin{pmatrix} 0 & u\\ v & -Q(X)\end{pmatrix}$$
must be invertible since its Schur complement is given by $uQ(X)^{-1}v = -r(X)$. Hence, we infer $X\in\dom_\cA\Big(\begin{pmatrix} 0 & u\\ v & -Q\end{pmatrix}^{-1}\Big)$. Furthermore, the Schur complement formula tells us in this case that
$$-\begin{pmatrix} 1 & 0\end{pmatrix} \begin{pmatrix} 0 & u\\ v & -Q(X)\end{pmatrix}^{-1}\begin{pmatrix} 1\\ 0 \end{pmatrix} = -(uQ(X)^{-1}v)^{-1} = r(X).$$
Since this holds for all $X\in\dom_\cA(r^{-1})$, we see that $\rho^{-1}$ is indeed a formal linear representation of $r^{-1}$.
\end{proof}

\subsubsection{Selfadjoint Formal Linear Representations}

We provide now some counterpart of Definition \ref{def:rep} designed for the case of rational expressions that are selfadjoint in the sense of Section \ref{subsubsec:symmetric_mat-val_rational_expressions}.

\begin{definition}\label{def:rep_sa}
Let $r$ be a selfadjoint rational expression over $\bbC$. A \textbf{selfadjoint formal linear representation $\rho=(Q,v)$}\index{formal linear representation!selfadjoint} consists of
\begin{itemize}
 \item an affine linear pencil $$Q = Q^{(0)} + Q^{(1)} x_1 + \dots + Q^{(g)} x_g$$ for selfadjoint matrices $Q^{(0)},Q^{(1)},\dots,Q^{(g)} \in M_n(\CC)$ for some $n$,
 \item and a $n\times 1$-matrix $v$ over $\CC$,
\end{itemize}
such that the following condition is satisfied:
\begin{quote}
For any unital complex $\ast$-algebra $\cA$, we have that
$$\dom_\cA^\sa(r) \subseteq \dom_\cA(Q^{-1})$$
and it holds true for any $(X_1,\dots,X_g) \in \dom_\cA^\sa(r)$ that
$$r(X_1,\dots,X_g) = - v^\ast Q(X_1,\dots,X_g)^{-1} v.$$
\end{quote}
\end{definition}

Moreover, we may generalize Theorem \ref{thm:existence-flr}.

\begin{thm}\label{thm:rep_sa_exist}
Each selfadjoint rational expression in $x=(x_1,\dots,x_g)$ over $\bbC$ (not necessarily regular at zero) admits a selfadjoint formal linear representation in the sense of Definition \ref{def:rep_sa}.
\end{thm}

\begin{proof}
We take any formal linear representation $\rho_0=(u_0,Q_0,v_0)$ of $r$ like in Theorem \ref{thm:existence-flr} and we put
\begin{equation}
\label{eq:rep_sa}
\rho = (Q,v) := \bigg(\begin{pmatrix} 0 & Q_0^\ast\\ Q_0 & 0\end{pmatrix}, \begin{pmatrix} \frac{1}{2} u_0^\ast \\ v_0\end{pmatrix}\bigg).
\end{equation}
Clearly, the linear pencil $Q$ consists of selfadjoint matrices and satisfies $\dom_\cA(Q_0^{-1}) = \dom_\cA(Q^{-1})$ for any unital complex ($\ast$-)algebra $\cA$, since we have for arbitrary $X \in \cA^g$ that $Q(X)$ is invertible if and only if $Q_0(X)$ is invertible.

Furthermore, if $\cA$ is any unital complex $\ast$-algebra, we have
$$\dom^\sa_\cA(r) \subseteq \dom_\cA(r) \subseteq \dom_\cA(Q_0^{-1}) = \dom_\cA(Q^{-1})$$
and we may observe that for each point $X\in \dom^\sa_\cA(r)$
\begin{align*}
-u Q(X)^{-1} v
&= \begin{pmatrix} \frac{1}{2} u_0 & v_0^\ast\end{pmatrix} \begin{pmatrix} 0 & Q_0(X)^\ast\\ (Q_0(X)^\ast)^{-1} & 0\end{pmatrix} \begin{pmatrix} \frac{1}{2} u_0^\ast \\ v_0\end{pmatrix}\\
&= - \frac{1}{2} u_0 Q_0(X)^{-1} v_0 - \frac{1}{2} v_0^\ast (Q_0(X)^\ast)^{-1} u_0^\ast\\
&= - \frac{1}{2} u_0 Q_0(X)^{-1} v_0 - \frac{1}{2} \big(u_0 Q_0(X)^{-1} v_0\big)^\ast\\
&= \frac{1}{2} r(X) + \frac{1}{2} r(X)^\ast\\
&= r(X).
\end{align*}
This completes the proof.
\end{proof}

We point out that, while we only discussed the complex case here, \textbf{symmetric formal linear representation $\rho=(Q,v)$}\index{formal linear representation!symmetric} of rational expressions over $\bbR$ that are symmetric in the sense of Section \ref{subsubsec:symmetric_mat-val_rational_expressions} can be defined and treated analogously. In particular, Theorem \ref{thm:rep_sa_exist} stays valid in the real case.

\subsection{Formal Linear Representations of Matrix-valued and Matrices of NC Rational Expressions}
\label{subsec:FLR_matval}

In this section, we want to explain how the theory presented in the previous subsection can be extended to \textbf{matrices of rational expressions}\index{rational expression!matrix of}. While this is our actual goal, it is convenient to discuss the more general case of matrix-valued rational expressions first.

\subsubsection{Formal Linear Representations of Matrix-valued NC Rational Expressions}
\label{subsubsec:FLR_matval}

Recall from Section \ref{subsubsec:matVrat} that matrix-valued rational expressions are built like rational expressions in a syntactically valid way -- by using the operations $+$, $\cdot$, and ${\cdot}^{-1}$, and by placing parentheses -- out of matrix-valued polynomials $M_{n\times m}(\bbK) \langle x_1,\dots,x_g\rangle$.
Note that we also agree on the convention \eqref{eq:scalar_matrices_commute_with_indeterminates}.

\begin{definition}\label{def:rep_matval}
Let $r$ be a matrix-valued rational expression of size $d_1 \times d_2$ in the variables $x$ over $\bbK$. A \textbf{matrix-valued formal linear representation $\rho=(u,Q,v)$ of $r$}\index{formal linear representation!matrix-valued} consists of
\begin{itemize}
 \item an affine linear pencil $$Q = Q^{(0)} + Q^{(1)} x_1 + \dots + Q^{(g)} x_g$$ for matrices $Q^{(0)},Q^{(1)},\dots,Q^{(g)}\in M_n(\bbK)$ of some dimension $n$,
 \item a $d_1\times n$-matrix $u$ over $\bbK$,
 \item and a $n\times d_2$-matrix $v$ over $\bbK$,
\end{itemize}
such that the following condition is satisfied:
\begin{quote}
For any unital $\bbK$-algebra $\cA$, we have that
$$\dom_\cA(r) \subseteq \dom_\cA(Q^{-1})$$
and it holds true for any $(X_1,\dots,X_g) \in \dom_\cA(r)$ that
$$r(X_1,\dots,X_g) = - u Q(X_1,\dots,X_g)^{-1} v.$$
\end{quote}
\end{definition}

It is easy to see that Algorithm \ref{alg:explicit_rep} extends immediately to the case of matrix-valued rational expressions.

\begin{algorithm}\label{alg:explicit_rep_matval}
Let $r$ be a matrix-valued rational expression in $x=(x_1,\dots,x_g)$. A matrix-valued formal linear representation $\rho=(u,Q,v)$ of $r$ can be constructed by using successively the following rules:
\begin{itemize}
 \item[(i)] For any affine linear polynomial over $\bbK$ (i.e., a linear pencil over $\bbK$), say $$\Lambda(x) = \Lambda^{(0)} + \Lambda^{(1)} x_1 + \ldots + \Lambda^{(g)} x_g,$$ a matrix-valued formal linear representation is given by
$$\rho_{\Lambda(x)} := \bigg(\begin{pmatrix} 0 & 1\end{pmatrix}, \begin{pmatrix} \Lambda(x) & -1\\ -1 & 0\end{pmatrix}, \begin{pmatrix} 0\\ 1 \end{pmatrix}\bigg).$$
 \item[(ii)] If $\rho_1=(u_1,Q_1,v_1)$ and $\rho_2=(u_2,Q_2,v_2)$ are matrix-valued formal linear representations for the matrix-valued rational expressions $r_1$ and $r_2$, respectively, then $\rho_1 \oplus \rho_2$ as defined in \eqref{eq:rep_sum} gives a matrix-valued formal linear representation of $r_1 + r_2$.
 \item[(iii)] If $\rho_1=(u_1,Q_1,v_1)$ and $\rho_2=(u_2,Q_2,v_2)$ are matrix-valued formal linear representations for the matrix-valued rational expressions $r_1$ and $r_2$, respectively, then $\rho_1 \odot \rho_2$ as defined in \eqref{eq:rep_prod} gives a matrix-valued formal linear representation of $r_1 \cdot r_2$.
 \item[(iv)] If $\rho=(u,Q,v)$ is a matrix-valued formal linear representation of $r$, then $\rho^{-1}$ as defined in \eqref{eq:rep_inv} gives a matrix-valued formal linear representation of $r^{-1}$.
\end{itemize}
\end{algorithm}

Within this frame, we thus obtain the following analogue of Theorem \ref{thm:existence-flr}.

\begin{thm}\label{thm:existence-flr_matval}
Each matrix-valued rational expression in the variables $x=(x_1,\dots,x_g)$ over $\bbK$ has a matrix-valued formal linear representation in the sense of Definition \ref{def:rep_matval}.
\end{thm}

Concerning matrix-valued rational expressions over $\bbC$ that are selfadjoint in the sense of Section \ref{subsubsec:symmetric_mat-val_rational_expressions}, we see that Definition \ref{def:rep_sa} extends to this generality.

\begin{definition}\label{def:matval_rep_sa}
Let $r$ be a selfadjoint matrix-valued rational expression over $\bbC$ of size $d \times d$ in the variables $x$. A \textbf{selfadjoint matrix-valued formal linear representation $\rho=(Q,v)$ of $r$}\index{formal linear representation!matrix-valued!selfadjoint} consists of
\begin{itemize}
 \item an affine linear pencil $$Q = Q^{(0)} + Q^{(1)} x_1 + \dots + Q^{(g)} x_g$$ for selfadjoint matrices $Q^{(0)},Q^{(1)},\dots,Q^{(g)}\in M_n(\CC)$ for some $n$,
 \item and a $n\times d$-matrix $v$ over $\CC$,
\end{itemize}
such that the following condition is satisfied:
\begin{quote}
For any unital complex $\ast$-algebra $\cA$, we have that
$$\dom_\cA^\sa(r) \subseteq \dom_\cA(Q^{-1})$$
and it holds true for any $(X_1,\dots,X_g) \in \dom_\cA^\sa(r)$ that
$$r(X_1,\dots,X_g) = - v^\ast Q(X_1,\dots,X_g)^{-1} v.$$
\end{quote}
\end{definition}

We also have an analogue of Theorem \ref{thm:rep_sa_exist} in this setting.

\begin{thm}\label{thm:self_matval}
Each selfadjoint matrix-valued rational expression in the variables $x=(x_1,\dots,x_g)$ over $\bbC$ admits a selfadjoint matrix-valued formal linear representation in the sense of Definition \ref{def:matval_rep_sa}.
\end{thm}

The proof uses exactly the same construction already used in the proof of Theorem \ref{thm:rep_sa_exist}. Namely, for any matrix-valued formal linear representation $\rho=(u_0,Q_0,v_0)$, whose existence is guaranteed by Theorem \ref{thm:existence-flr_matval}, we obtain by
$$\rho = (Q,v) := \bigg(\begin{pmatrix} 0 & Q_0^\ast\\ Q_0 & 0\end{pmatrix}, \begin{pmatrix} \frac{1}{2} u_0^\ast \\ v_0\end{pmatrix}\bigg)$$
a selfadjoint matrix-valued formal linear representation of the same matrix-valued rational expression.

Again, while having only discussed the complex case, the case of symmetric matrix-valued rational expressions over $\bbR$ can be treated similarly. In particular, one can prove in analogy to Theorem \ref{thm:self_matval} that they allow always a symmetric formal linear representation \textbf{symmetric matrix-valued formal linear representation}\index{formal linear representation!matrix-valued!symmetric}. 

We conclude by noting -- without going much into details -- that \textbf{operator-valued formal linear representations}\index{formal linear representation!operator-valued} of \textbf{operator-valued rational expressions}\index{rational expression!operator-valued} can be defined and treated similarly. The underlying idea is that, instead of replacing the scalar field $\bbK$ by the set of all rectangular matrices over $\bbK$, we could alternatively replace $\bbK$ by any unital $\bbK$-algebra $\cB$.
Defining evaluations is then a little more tricky. In particular, it depends on whether one keeps the convention \eqref{eq:scalar_matrices_commute_with_indeterminates} or not; in the first named case, evaluations on unital $\bbK$-algebras $\cA$ would give values in $\cB \otimes_\bbK \cA$, while in the latter case algebras $\cA$ are considered, into which $\cB$ unitally embeds, leading to $\cA$-valued evaluations on $\cA$.

\subsubsection{Formal Linear Representations of Matrices of NC Rational Expressions}

As explained in Section \ref{subsubsec:matrices_of_rational_expressions}, matrices of rational expressions form a subclass of all matrix-valued rational expressions. Accordingly, the results provided in Section \ref{subsubsec:FLR_matval} directly apply to them and yield the following result.

\begin{thm}\label{thm:FLR_matrices_of_rational_expressions}
Each $d_1 \times d_2$-matrix $\ur$ of rational expressions in $x=(x_1,\dots,x_g)$ over $\bbC$ admits matrix-valued formal linear representation in the sense of Definition \ref{def:rep_matval}.
If $\ur$ is of square type ($d_1=d_2$) and moreover selfadjoint in the sense of Section \ref{subsubsec:symmetric_mat-val_rational_expressions}, then there exists a selfadjoint matrix-valued formal linear representation in the sense of Definition \ref{def:matval_rep_sa}.
\end{thm}

Note that, since the results of Section \ref{subsubsec:FLR_matval} stay valid in the real case, there is also a real version of Theorem \ref{thm:FLR_matrices_of_rational_expressions} stating the existence of (symmetric) formal linear representations for (symmetric) matrices $\ur$ of rational expressions in $x$ over $\bbR$.

\subsection{Application to Descriptor Realizations of Matrices of Regular NC Rational Expressions}
\label{subsec:regular_case}

We specify now Theorem \ref{thm:FLR_matrices_of_rational_expressions} to the case of matrices of rational expressions that are regular at zero; this result will be used in the proof of Theorem \ref{thm:rep_min}.

\begin{thm}\label{thm:realizations_matrices_of_rational_expressions}
For each $d_1 \times d_2$-matrix $\ur$ of regular rational expressions in $x=(x_1,\dots,x_g)$ over $\bbK$, the following statements hold true:
\begin{itemize}
 \item[(i)] The matrix $\ur$ admits a monic descriptor realization of the form $$\br(x) = D + C (I - L_A(x))^{-1} B,$$ where the feed through term $D \in M_{d_1 \times d_2}(\bbK)$ can be prescribed arbitrarily, which enjoys the following property:
\begin{quote}
If $\cA$ is a unital $\bbK$-algebra, then
$$\dom_\cA(\ur) \subseteq \dom_\cA(\br)$$
and
$$\ur(X) = \br(X) \qquad\text{if $X \in \dom_\cA(\ur)$}.$$
\end{quote}
 \item[(ii)] If $\ur$ is of square type (i.e., $d:=d_1=d_2$) and additionally selfadjoint (resp. symmetric), then $\ur$ admits a selfadjoint (resp. symmetric) descriptor realization of the form $$\br(x) = \Delta + \Xi^\ast (M_0 - L_M(x))^{-1} \Xi,$$ where the feed through term $\Delta\in M_d(\bbK)$ can be prescribed arbitrarily, which enjoys the following property:
\begin{quote}
If $\cA$ is a unital complex (resp. real) $\ast$-algebra, then
$$\dom_\cA^\sa(\ur) \subseteq \dom_\cA(\br)$$
and
$$\ur(X) = \br(X) \qquad\text{if $X \in \dom_\cA^\sa(\ur)$}.$$
\end{quote}
\end{itemize}
\end{thm}

We point out that, while (i) is rather straightforward to prove, its selfadjoint counterpart (ii) is slightly more intricate. In particular, it relies crucially on \cite[Proposition A.7]{HMV06}, according to which rational expressions are $\bbM(\bbK)$-evaluation equivalent, if and only if they are $\bbM(\bbK)_\sa$-evaluation equivalent. 

\begin{proof}[Proof of Theorem \ref{thm:realizations_matrices_of_rational_expressions}]
For proving (i), we proceed as follows: by Theorem \ref{thm:FLR_matrices_of_rational_expressions} we can find a matrix-valued formal linear representation $\rho=(u,Q,v)$ of $\ur-D$. Since $0 \in \dom_\cA(\ur) = \dom_\cA(\ur-D)$ holds by the regularity assumption and since we have $\dom_\cA(\ur-D) \subseteq \dom_\cA(Q^{-1})$ due to Definition \ref{def:rep_matval}, we see that the linear pencil $Q$ entails an invertible matrix $Q^{(0)}$. Thus, we may introduce
$$\br_0(x) := - u \big(I + (Q^{(0)})^{-1}Q^{(1)} x_1 + \dots + (Q^{(0)})^{-1}Q^{(g)} x_g\big)^{-1} (Q^{(0)})^{-1}v,$$
which is of the form $C (I - L_A(x))^{-1} B$ with $C=-u$, $B=(Q^{(0)})^{-1}v$ and $A_j = -(Q^{(0)})^{-1}Q^{(j)}$ for $j=1,\dots,g$. Again by Definition \ref{def:rep_matval}, we know that $\dom_\cA(\ur-D) \subseteq \dom_\cA(\br)$ holds for any unital complex algebra $\cA$ and in addition
$$\ur(X) - D = \br_0(X) \qquad\text{for all $X\in \dom_\cA(\ur)$},$$
i.e.
$$\ur(X) = D + C (I - L_A(x))^{-1} B \qquad\text{for all $X\in \dom_\cA(\ur)$}.$$
Since this applies in particular to the case $\cA=M_N(\bbK)$, we see that $\ur$ and $\br$ are equivalent under matrix evaluation and hence power series equivalent, which means that we have found by $\br(x) = D + C (I - L_A(x))^{-1} B$ the desired monic descriptor realization of $\ur$.

For proving (ii), we need some refinement of our previous argument: since $\ur$ is assumed to be regular at zero, we know that for any formal linear representation $\rho_0=(u_0,Q_0,v_0)$ of $\ur-\Delta$, the matrix $Q^{(0)}_0$ appearing in the linear pencil
$$Q_0 = Q^{(0)}_0 + Q^{(1)}_0 x_1 + \dots + Q^{(g)}_0 x_g$$
has to be invertible. Thus, we may form with $\widetilde{Q}^{(j)}_0 := (Q^{(0)}_0)^{-1} Q^{(j)}_0$ for $j=0,\dots,g$ the linear pencil
$$\widetilde{Q}_0 = \widetilde{Q}^{(0)}_0 + \widetilde{Q}^{(1)}_0 x_1 + \dots + \widetilde{Q}^{(g)}_0 x_g \qquad\text{where}\qquad \widetilde{Q}^{(0)}_0 = I.$$
We define in addition $\widetilde{u}_0 := u_0$ and $\widetilde{v}_0 := (Q^{(0)}_0)^{-1} v_0$. Clearly, we obtain via this construction another formal linear representation $\widetilde{\rho}_0 = (\widetilde{u}_0,\widetilde{Q}_0,\widetilde{v}_0)$ of $\ur-\Delta$. If we proceed now with the construction that was presented in \eqref{eq:rep_sa}, this yields a selfadjoint formal linear representation
$$\rho = (Q,v) := \bigg(\begin{pmatrix} 0 & \widetilde{Q}_0^\ast\\ \widetilde{Q}_0 & 0\end{pmatrix}, \begin{pmatrix} \frac{1}{2} \widetilde{u}_0^\ast \\ \widetilde{v}_0\end{pmatrix}\bigg).$$
Now, we continue in analogy to the proof of Item (i): starting with the selfadjoint formal linear representation $\rho = (Q,v)$, which can be seen also as a formal linear representation $(v^\ast,Q,v)$, we introduce
$$\br_0(x) := - v^\ast \big(Q^{(0)} + Q^{(1)} x_1 + \dots + Q^{(g)} x_g\big)^{-1} v,$$
which is of the form $\br_0(x) = \Xi^\ast (M_0 - L_M(x))^{-1} \Xi$ with $\Xi=v$, $M_0 = - Q^{(0)}$, and $M_j = Q^{(j)}$ for $j=1,\dots,g$. Note that indeed $M_0^2 = I$. Finally, we put
$$\br(x) := \Delta + \Xi^\ast (J - L_M(x))^{-1} \Xi.$$

Thus, by construction, we have for any unital complex $\ast$-algebra $\cA$ that
$$\dom_\cA^\sa(\ur) \subseteq \dom_\cA(\br)$$
and
$$\ur(X) = \br(X) \qquad\text{if $X \in \dom_\cA^\sa(\ur)$}.$$

It remains to prove that $\br$ is indeed a realization of $\ur$. However, if applied in the case $\cA = M_N(\bbK)$, the statement above only yields that $\ur$ and $\br$ take the same values on all selfadjoint matrices belonging to the domain of $\ur$, which does not allow to conclude directly that $\ur$ and $\br$ are $\bbM(\bbK)$-evaluation equivalent. In order to reach the desired conclusion, we need to take \cite[Proposition A.7]{HMV06} into account: Let $\widetilde{\ur}$ be any matrix of rational expressions, which represents the rational function that is determined by $\br$. Thus, in other words, $\br$ is a descriptor realization of $\widetilde{\ur}$. Since $\widetilde{\ur}$ and $\br$ are therefore $\bbM(\bbK)$-evaluation equivalent and since $\ur$ and $\br$ are $\bbM(\bbK)_\sa$-evaluation equivalent, we may conclude by transitivity that $\ur$ and $\widetilde{\ur}$ are $\bbM(\bbK)_\sa$-evaluation equivalent. Now, \cite[Proposition A.7]{HMV06} tells us that $\ur$ and $\widetilde{\ur}$ must be even $\bbM(\bbK)$-evaluation equivalent. Hence, again by transitivity, we obtain that $\ur$ and $\br$ are in fact $\bbM(\bbK)$-evaluation equivalent, as desired.
\end{proof}


\section{Evaluations of NC Rational Expressions and Their Realizations}
\label{sec:evaluations}

In the free probability context we are not so much interested in plugging in matrices in our rational functions, but we would like to take operators on infinite dimensional Hilbert spaces as arguments.

As we already alluded to in the Introduction, the domain of our rational functions should be stably finite, otherwise there will be inconsistencies. 
We want to be here a bit more precise on this.

\subsection{Stably Finite Algebras}
\label{subsec:stably_finite}

A \df{stably finite algebra $\cA$} is one with the following property
for each $n\in\NN$: every $A\in M_n(\cA)$ 
with  
either a left inverse or a right inverse has an inverse; i.e., if we have $A,B\in M_n(\cA)$, then $AB=1$ implies $BA=1$.
Sometimes ``stably finite'' is also addressed as ``weakly finite''.
These are suitable for free probability, since there we are usually working
in a context, where we have a faithful trace.

\begin{lem}
\label{lem:trace}
A  unital $C^*$-algebra $\cA$ with a faithful trace $\tau$ is stably finite.
\end{lem}

The lemma is not surprising to those in the area, though we could not find a pinpoint reference in the literature, hence we include its proof (which is short). Dima Shlyakhtenko pointed out to us that this holds not just for a type II factor but also for its affiliated algebra of unbounded operators.

\begin{proof}[Proof of Lemma \ref{lem:trace}]
This is a standard fact in operator algebras; see for example \cite{RLL00} where it shows up as an exercise. 
For the convenience of the reader, let us give a rough outline of proof. First of all, since we can extend the trace $\tau$ on $\cA$ in the canonical way to a faithful trace $\tr_n\otimes\tau$ on $M_n(\cA)$, it suffices to prove the required property for $n=1$. (Here $\tr_n$ denotes the normalized trace on $n\times n$-matrices.) 

\begin{enumerate}
\item
If $U\in\cA$ is an isometry, i.e., $U^*U=1$, then it is unitary, i.e., also $UU^*=1$. This follows because we have $0 = \tau( 1 - U^*U) = \tau( 1 - UU^*)$ and then the faithfulness of $\tau$ implies $1- UU^* =0$.

\item
Consider $A,B\in B(H)$ with $AB=1$ and $A=A^*$. Then $B^*A=1$, hence
$B^*=B^*AB=B$, and thus $BA=1$. Thus $B$ is invertible.

\item
Consider now arbitrary $A,B\in\cA$ with $AB=1$. By polar decomposition, we can write $B=UP$ with $U$ a partial isometry and $P\geq 0$, in particular $P^*=P$. Note that a priori $U$ might not be in $\cA$, but only in $B(H)$; in contrast, $P\in\cA$ is automatically satisfied, since in fact $P=(A^\ast A)^{1/2}$ holds with the positive square root defined via the continuous functional calculus.
In any case we have then $AUP=AB=1$. So $P$ has a left-inverse and thus by the previous item also a right inverse and hence is invertible in $B(H)$. Using the continuous functional calculus again, we see that this inverse must belong to $\cA$. Then $U=BP^{-1}$ belongs also to $\cA$ and must also be an isometry. By the first item, it must then be a unitary. Hence $B$ is invertible, i.e., we also have $BA=1$.
\end{enumerate}
\end{proof}

The next lemma provides an important characterization of stably finite algebras, which we will use below.

\begin{lem}
\label{lem:diagNotInv}
Suppose $\cA$ is a unital (complex or real) algebra.
Then the statement
\begin{quote}
{ \it If a $k \times k$ block triangular matrix $G$ with entries
$G_{ij}$ in $M_{m_i \times m_j}(\cA)$ for some $m_1,\dots,m_k\in\NN$ is invertible, then all its diagonal entries
$G_{ii} \in M_{m_i}(\cA)$ must be invertible.}
\end{quote}
holds if and only if $\cA$ is stably finite.
\end{lem}

\begin{proof}
First prove that stably finite is necessary.
Suppose $R,T \in M_{m}(\cA)$
satisfy $TR =I$ and define
\begin{equation}
G= \bmat
R & I \\
0 & T
\emat
\qquad \text{and} \qquad
W=  \bmat
T       & -I \\
I-  RT & R
\emat
\end{equation}
Then $G $ and $W$ are inverses, but
if $\cA$ is not stably finite, then for some $m$ there exist
such $T$ and $R$ in $M_{m}(\cA)$ which are not invertible.

Next prove stably finite is sufficient.
To treat
\begin{equation}
\label{eq:triG}
 G=\bmat
 G_{11}& * & * & *\\
0 & G_{22} & * & * \\
0 &  0 & \ddots & * \\
0 &  0 & 0& G_{kk}
\emat
\end{equation}
we  shall successively partition $G$ into 2  blocks
which respects the given block structure
and
also partition the inverse $W$ of $G$ conformably with $G$:
$$
G= \bmat
R & S \\
0 & T
\emat
\qquad \text{and} \qquad
W=  \bmat
A     & B \\
C & D
\emat.
$$
Calculate that $GW=I$ implies $TD=I$, hence by stable finiteness $T$ and $D$
are inverses.
That $WG=I$ implies $AR=I$, hence by stable finiteness $A$ and $R$
are inverses. Thus both $R,S$ are invertible
and since $R,T$ are both block triangular we face a similar problem
but fortunately with a lower number of blocks.
Moreover, it is easy to see that
\begin{equation}
\label{eq:invG}
G^{-1}= W= \bmat
R^{-1} & -R^{-1} S T^{-1} \\
0 & T^{-1}
\emat.
\end{equation}

Given $G$ in \eqref{eq:triG} apply this decomposition
argument  to $R,T$  and then to subblocks of them, ultimately
obtaining  the $G_{ii}$ are invertible.
\end{proof}

\subsection{Evaluation Equivalence on Stably Finite Algebras}
\label{sect:2-4}

The following theorem states that relations in the free field 
are valid in any stably finite algebra in our languages of $\cA$-equivalences.
   
\begin{thm}
\label{thm:sameOnX}~

\begin{enumerate}

\item\label{it:sameOnX} If $\widetilde r$ and $r$ are $\bbM(\bbK)$-evaluation equivalent rational expressions over the field $\bbK$ of real or complex numbers, then $\widetilde r$ and $r$ are $\cA$-equivalent provided $\cA$ is a stably finite $\bbK$-algebra.

\item\label{it:stably_finite} If $\cA$ is not stably finite then there exist rational expressions $\widetilde r$ and $r$, which are $\bbM(\bbK)$-evaluation equivalent, and $X\in\cA$ with
$X\in \dom_\cA(r)\cap \dom_\cA(\widetilde r)$, but $r(X)\not=\widetilde r(X)$.

\end{enumerate} 
\end{thm}

This is a special case of Theorem 7.8.3 in the book \cite{Co06}.
A warning is that the terminology and cross-referencing there required to make this
conversion is extensive.
Hence, for the reader's convenience, we now give in the regular case a proof of 
Theorem \ref{thm:sameOnX}.

It turns out that the statement (1) of Theorem \ref{thm:sameOnX} also holds for matrix-valued rational expressions; thus, in particular, we may compare matrix-valued rational expressions and their realizations under evaluation. Since this important result will also be used later, it is conventient to place it first.

\begin{thm}
\label{thm:sameOnX_generalized}~

\begin{enumerate}

 \item\label{it:sameOnX_general} Let $\widetilde r$ and $r$ be $\bbM(\bbK)$-evaluation equivalent regular matrix-valued rational expressions in the variables $x=(x_1,\dots,x_g)$ over the field $\bbK$ of real or complex numbers, then $\widetilde r$ and $r$ are $\cA$-evaluation equivalent provided $\cA$ is a stably finite $\bbK$-algebra.
 
 \item\label{it:sameOnX_descriptor} Let $\ur$ be a $d_1 \times d_2$-matrix of rational expressions in the variables $x=(x_1,\dots,x_g)$, which are all regular at $0$. Consider any descriptor realization 
$$\br(x) = D + C (J - L_A(x))^{-1} B, \qquad \text{where}\qquad L_A(x) = A_1 x_1 + \dots + A_g x_g,$$
say of size $N\times N$, such that $\br$ and $\ur$ are $\bbM(\bbK)$-evaluation equivalent.

Provided that $\cA$ is a stably finite $\bbK$-algebra, then $\ur$ and $\br$ are also $\cA$-evaluation equivalent, i.e., we have
$$\ur(X) = \br(X) \qquad\text{for all $X\in\dom_\cA(\ur) \cap \dom_\cA(\br)$}.$$
\end{enumerate}

\end{thm}

\begin{proof}
(1) It clearly suffices to check the following for any regular matrix-valued rational expression $r$ which is
 $\bbM(\bbK)$-evaluation equivalent to zero: if $\cA$ is a stably finite $\bbK$-algebra and $X=(X_1,\dots,X_g)\in\cA^g$ is in the domain of $r$, then $r(X)=0$.
In order to prove this we proceed as follows. First of all, Theorem \ref{thm:existence-flr_matval} tells us that we may find a formal linear representation $\rho=(u,Q,v)$ of $r$, say of size $N \times N$. Like in the proof of Theorem \ref{thm:realizations_matrices_of_rational_expressions}, we may use this to build a monic descriptor realization
$$\br(x) = C (I-L_A(x))^{-1} B,$$
with the feed through term $D$ chosen to be $0$, such that (by Definition \ref{def:rep_matval}) the $\cA$-domain of $\br$ contains the domain of $r$ and such that $r$ and $\br$ take the same value when evaluated at any point in the $\cA$-domain of $r$. Thus we can assume that for our fixed $X$ the matrix $\Lambda_A(X) = I-L_A(X)$ is invertible in $M_N(\bbK) \otimes_\bbK \cA$ and that $r(X) = C \Lambda_A(X)^{-1} B$ holds.

Since the monic system $A,B,C,0$ gives a realization of $0$ in the free field, its Sys matrix
$$\Sys(A,B,C,0) =
\begin{pmatrix}
\Lambda_A(x) & B \\ C & 0
\end{pmatrix}$$
cannot be invertible. By results of Cohn \cite{Co06} 
this means that it is not full, i.e. it can be written as 
the product of strictly rectangular matrices with polynomials as entries. 
Since we are working with regular rational expressions, 
we can prove this directly from the machinery we have developed
here, thus allowing the reader not to dig into the copious
writings of Cohn.

First note that Remark \ref{rem:hollow} implies that $\Sys(A,B,C,0)$
can be written as a hollow matrix.
Now hollow matrices factor as
$$\begin{pmatrix}
f_{11} & f_{12} & f_{13}\\
0 & 0 & f_{23}\\
0 & 0 & f_{33}
\end{pmatrix}
=
\begin{pmatrix}
1 & f_{13} \\
0 & f_{23}\\
0 & f_{23}\end{pmatrix}
\begin{pmatrix}
f_{11} & f_{12} & 0\\
0 & 0 & 1
\end{pmatrix}.
$$

This means we have
$$\begin{pmatrix}
\Lambda_A(x) & B \\ C & 0
\end{pmatrix}
= \begin{pmatrix} P_1(x)&0\\ u_1(x)&0
\end{pmatrix}
\begin{pmatrix}
P_2(x)& v_2(x)\\ 0&0
\end{pmatrix},
$$
where $P_1(x)$, $P_2(x)$, $u_1(x)$, $v_2(x)$ are matrices with polynomials as entries.

Replace $x$ by $X$ and apply the block LDU decomposition to get
\begin{align*}
\begin{pmatrix}
1 & 0\\
0 &-C \Lambda_A(X)^{-1} B
\end{pmatrix}
&=
\begin{pmatrix}
\Lambda_A(X)^{-1} & 0\\
-C \Lambda_A(X)^{-1} & 1
\end{pmatrix}
\begin{pmatrix}
\Lambda_A(X) & B\\
C & 0
\end{pmatrix}
\begin{pmatrix}
1 &-\Lambda_A(X)^{-1} B\\
0 & 1
\end{pmatrix}\\
&= \begin{pmatrix}
M(X) & 0\\
m(X) & 0
\end{pmatrix} 
\begin{pmatrix}
N(X) & n(X)\\
0 & 0 
\end{pmatrix}.
\end{align*}

Hence we have
$$
\begin{pmatrix}
1&0\\
0&-r(X)
\end{pmatrix}=
\begin{pmatrix}
1&0\\0&-C \Lambda_A(X)^{-1} B
\end{pmatrix}
= \begin{pmatrix}
M(X) &0\\ m(X)&0
\end{pmatrix} 
\begin{pmatrix}
N(X)  & n(X)\\0&0 
\end{pmatrix},
$$
where $M(X)$, $m(X)$, $N(X)$, $n(X)$ are matrices over $\cA$. This yields then
$$1=M(X) N(X), \quad 0=M(X) n(X),\quad 0=m(X) N(X),\quad -r(X)=m(X) n(X).
$$
Since $\cA$ is stably finite the first equation implies that the square matrices
$M(X)$ and $N(X)$ are inverses of each other, i.e., in particular, invertible; the second equation gives then $n(X)=0$, which finally yields $r(X)=0$.

(2) Item \ref{it:sameOnX_descriptor} is an immediate consequence of the result provided in Item \ref{it:sameOnX_general}, since both matrices of rational expressions and descriptor realizations form a subclass of all matrix-valued rational expressions; see Section \ref{subsubsec:matrices_of_rational_expressions} and Section \ref{subsec:descriptorrealizations}.
\end{proof}

Now, we may return to Theorem \ref{thm:sameOnX}.

\begin{proof}[Proof of Theorem \ref{thm:sameOnX} for regular rational expressions]
The assertion in Item \ref{it:sameOnX} is a special case of Theorem \ref{thm:sameOnX_generalized} \eqref{it:sameOnX_general}.
The statement of Item \ref{it:stably_finite} can be proven as follows: if $\cA$ is not stably finite then there exists square $n\times n$ matrices $Q,P$ over $\cA$ such that $PQ=1$, but $QP\not=1$. Let $T$ and $S$ be $n\times n$ matrices over the free field with indeterminate entries. We have then $T(ST)^{-1}S-1=0$. This gives $n^2$ equations in the entries of $S$, $T$ and $(ST)^{-1}$; all of which are 0 in the free field. However, not all of them are true in our algebra $\cA$, though all expressions make sense there.
\end{proof}

\subsection{Maximality of the Domain of Minimal Descriptor Realizations}
\label{subsec:minimality_implies_maximal_domain}

Let us consider a descriptor realization of size $d_1 \times d_2$ like in \eqref{eq:descrN}, i.e.,
$$
\dr{r}(x) = D +  C ( J -  L_A(x) )^{-1} B \qquad \text{with} \qquad L_A(x) = A_1 x_1 + \dots + A_g x_g.
$$
Recall from Lemma \ref{lem:cutdown0} that the associated monic system $JA, JB, C,D $ has the block decomposition given in \eqref{eq:basiccut} as
$$
JA =
\begin{pmatrix}
\hat A_{11}&     \hat A_{12} &      A_{12}^1  \\
0 &     \check A &      A_{12}^2  \\
0 & 0  &     A_{22}
\end{pmatrix},
\qquad
JB=
\bmat
\hat B_1 \\
\check B
\\
0
\emat,
\quad
C = \bmat  0  & \check C  &  C_2
\emat
$$
with respect to the subspace decomposition $ \hat\cQ^\perp +  \; \hat \cQ + \; \cS^\perp $ of $\bbK^d$, in which the minimal monic descriptor system $\check A, \check B, \check C, D$ appears.

\begin{lem}\label{lem:cutdown1}
Suppose that $\cA$ is a stably finite algebra. Then it holds true that
\begin{equation}
\label{eq:domcutdown}
\domA{[I - L_{JA}(x)]^{-1} } \subseteq
\domA{[I -  L_{\check A}(x)]^{-1} }
\end{equation}
and the $d_1 \times d_2$ descriptor realizations
\begin{equation*}
    \dr{r}(x) = D+C(I-L_{JA}(x))^{-1} JB \qquad\text{and}\qquad
    \check{\dr{r}}(x) =  D +\check{C}(I-L_{\check{A}}(x))^{-1}\check{B},
\end{equation*}
take the same values on any $X \in \domA{\dr{r}}$.
\end{lem}

\begin{proof}
The core of the proof of \eqref{eq:basiccut} was sketched
before the lemma statement,
but for literature references
to the triangular form
you can take your choice of
contexts, eg. Theorem 1.4 \cite{CR}
or \cite{BMG05} Theorem 8.2.
In the simplest of cases the notions of  controllability,
observability and also the cutdown goes back
at least to Rudy Kalman \cite{K63}.

Now to $\cA$ domains. The domain inclusion assertion \eqref{eq:domcutdown} is equivalent to saying
\begin{equation}
\Lam_{JA} (X) =
\begin{pmatrix}
\Lam_{\hat  A_{11} } (X)  &    \Lam_{\hat A_{12}} (X)  &
\Lam_{ A_{12}^1 } (X)  \\
0 &    \Lam_{ \check A} (X)  &
    \Lam_{ A_{12}^2 } (X)  \\
0 & 0  &    \Lam_{ A_{22} } (X)
\end{pmatrix}
\end{equation}
is invertible implies
$\Lam_{ \check A} (X) $ must be invertible; note that this uses the notation introduced in \eqref{eq:Lambda-pencil}.
Lemma \ref{lem:diagNotInv} analyses this implication and immediately yields
\eqref{eq:domcutdown}.

That the values of the descriptor realizations $\dr{r}$ and $\check{\dr{r}}$
are equal
at $X \in \cA$ for which they are both defined
is true because \eqref{eq:invG} applied twice implies
\begin{align*}
\dr{r}(X) &= D+ C \Lam_{JA} (X) ^{-1} J B\\
&=
 \bmat  0  & \check C  &  C_2  \emat
\begin{pmatrix}
\Lam_{\hat  A_{11} } (X)^{-1}  &    *  &
*  \\
0 &    \Lam_{ \check A} (X)^{-1}   & *  \\
0 & 0  &    \Lam_{ A_{22} } (X)^{-1}
\end{pmatrix}
\bmat
\hat B_1 \\
\check B
\\
0
\emat
\end{align*}
which clearly equals
$
 D+ \check C \
\Lam_{ \check A} (X)^{-1}  \
\check B =
\check{\dr{r}}(X)$.
\end{proof}

Lemma \ref{lem:cutdown1} combines with the state space similarity theorem formulated in Lemma \ref{lem:symR} and the characterization of stably finite algebras provided in Lemma \ref{lem:diagNotInv} to yield that minimal descriptor realizations have maximal $\cA$-domains, as we now state formally.

\begin{lem}
\label{lem:domMin}\

\begin{enumerate}
\item
\label{it:mindom}
Any two minimal realizations of the same regular matrix-valued rational expression over $\bbK$, which both have the same feed through term,
have the same $\cA$-domain for any unital $\bbK$-algebra $\cA$.

Furthermore, they take the same values on their joint $\cA$-domain.

\item
\label{it:domcutdownMin}
Suppose $J,A,B,C,D$ and $\hJ,\hA,\hB,\hC,D$ are both (selfadjoint/symmetric) descriptor
realizations for the same (selfadjoint/symmetric) matrix-valued rational expression with $\hJ,\hA,\hB,\hC,D$
being minimal.
If $\cA$ is a unital complex/real ($\ast$-)algebra, which is stably finite, then
\begin{equation}
\label{eq:domcutdownMin}
\domA{[J - L_{A}(x)]^{-1}} \subseteq 
\domA{[\hJ - L_{\hA}(x)]^{-1}}
\end{equation}
and the matrix-valued rational expressions 
\begin{equation*}
    \dr{r} = D+C(J-L_A(x))^{-1}B \qquad \text{and} \qquad
    \dr{\hat{r}} = D + {\hC}(\hJ-L_{\hA}(x))^{-1} \hB,
\end{equation*}
take the same values on any $X \in \domA{\dr{r}}$.
\end{enumerate}
\end{lem}

\begin{proof}
The first assertion in Item \ref{it:mindom} is an immediate consequence of the state space similarity theorem formulated in Item \ref{lem:symR0} of Lemma \ref{lem:symR}. Indeed, if
\begin{equation*}
\dr{r} = D + C (J-L_A(x))^{-1} C^{\transpose} \qquad \text{and} \qquad
\widetilde{\dr{r}} = D + \widetilde{C} (\widetilde{J}-L_{\widetilde{A}}(x))^{-1} \widetilde{C}^{\transpose}
\end{equation*}
are both minimal descriptor realizations of the same matrix-valued rational expression $r$, then Item \ref{lem:symR0} of Lemma \ref{lem:symR} guarantees that there is an invertible matrix $S$, which satisfies
$$SJA_j = \widetilde{J} \widetilde{A}_j S \qquad SJC^\transpose = \widetilde{J} \widetilde{C}^\transpose \qquad C = \widetilde{C}S.$$
If we put $\widetilde{S} := \widetilde{J} S J$, which gives as well an invertible matrix, then we may check that
$$\widetilde{S} (J-L_A(x)) = (\widetilde{J}-L_{\widetilde{A}}(x)) S.$$
Indeed, since $SJ L_A(x) = \widetilde{J} L_{\widetilde{A}}(x) S$, we obtain
$$SJ (J-L_A(x)) = S - SJ L_A(x) = S - \widetilde{J} L_{\widetilde{A}}(x) S = \widetilde{J} (\widetilde{J}-L_{\widetilde{A}}(x)) S$$
and hence $\widetilde{S} (J-L_A(x)) = (\widetilde{J}-L_{\widetilde{A}}(x)) S$, as stated. Thus,
$$\dom_\cA([J-L_A(x)]^{-1}) = \dom_\cA([\widetilde{J}-L_{\widetilde{A}}(x)]^{-1}).$$

The proof of Item \ref{it:domcutdownMin} is based on explicit constructions which can be implemented numerically.
We start with $J,A,B,C,D$ and apply Lemma \ref{lem:cutdown0} to obtain a monic minimal realization $I,\check{A},\check{B},\check{C},D$ satisfying domain inclusion as in \eqref{eq:domcutdown}, i.e.
$$\domA{[I - L_{JA}(x)]^{-1}} \subseteq \domA{[I - L_{\check A}(x)]^{-1}}$$
Furthermore, by Item \ref{it:mindom}, the cut down system $I,\check{A},\check{B},\check{C},\check{D}$ must have the same $\cA$-domain as the given minimal system $\hJ,\hA,\hB,\hC,D$. Thus,
\begin{align*}
\domA{[J-L_A(x)]^{-1}} &= \domA{[I - L_{JA}(x)]^{-1}}\\
                       &\subseteq \domA{[I - L_{\check A}(x)]^{-1}} = \domA{[\hJ - L_{\hA}(x)]^{-1}}.
\end{align*}
Moreover, we know
\begin{itemize}
 \item by Lemma \ref{lem:cutdown1} that $\br$ and $\check{\br}$ take the same value on any $X\in\dom_\cA(\br)$ and
 \item by Item \ref{it:mindom} that $\check{\br}$ and $\hat{\br}$ take the same value on any $X\in\dom_\cA(\check{\br})$,
\end{itemize}
thus
\begin{equation*}
    \dr{r} = D+C(J-L_A(x))^{-1}B \qquad \text{and} \qquad
    \dr{\hat{r}} = D + {\hC}(\hJ-L_{\hA}(x))^{-1} \hB,
\end{equation*}
take the same values on any $X \in \domA{\dr{r}}$.

We conclude by noting that the selfadjoint/symmetric case of Item \ref{it:domcutdownMin} is clearly covered by the more general statement that was proven above, since any selfadjoint/symmetric minimal realization is in particular a minimal realization.
\end{proof}

\subsection{Minimal Descriptor Realizations of NC Rational Expressions}
\label{subsec:minimal_realizations}

Theorem \ref{thm:sameOnX_generalized} \eqref{it:sameOnX_descriptor} tells us that descriptor realizations of rational expressions provide valid identities when evaluated at points which belong both to the domain of the descriptor realization and to the domain of the rational expression. However, one would clearly prefer to work with such descriptor realizations, which can be evaluated on the entire domain of the rational expression without any further constraints.
While formal linear representations have this desirable property by definition, they have the disadvantage of being typically of large size. However, as we are going to prove now, these excellent evaluation properties pass to descriptor realizations in the setting of stably finite algebras and under the assumption of minimality.

\begin{thm}\label{thm:rep_min}
Let $\ur$ be a $d_1 \times d_2$-matrix of regular rational expressions in the variables $x=(x_1,\dots,x_g)$ over $\bbK$. Then the following statements hold true:
\begin{itemize}
 \item[(i)] Each minimal descriptor realization
$$\hat{\br}(x) = D + \hat{C} (\hat{J} - L_{\hat{A}}(x))^{-1} \hat{B},$$
of $\ur$ satisfies the following property:
\begin{quote}
If $\cA$ is a unital $\bbK$-algebra, which is stably finite, then
\begin{equation}
\label{eq:QrealizDom}
\dom_\cA(\ur) \subseteq \dom_\cA(\hat{\br})
\end{equation}
and 
\begin{equation}
\label{eq:QrealizEval}
\ur(X) = \hat{\br}(X) \qquad\text{if $X \in \dom_\cA(\ur)$.}
\end{equation}
\end{quote}
 \item[(ii)] If $\ur$ is of square type (i.e., $d:=d_1=d_2$) and additionally selfadjoint/symmetric, then any selfadjoint/symmetric minimal descriptor realization
$$\hat{\br}(x) = \Delta + \hat{\Xi}^\ast (\hat{M}_0 - L_{\hat{M}}(x))^{-1} \hat{\Xi},$$
of $\ur$ satisfies the following property:
\begin{quote}
If $\cA$ is a unital complex/real $\ast$-algebra, which is stably finite, then
\begin{equation}
\label{eq:QrealizDom_sa}
\dom_\cA^\sa(\ur) \subseteq \dom_\cA(\hat{\br})
\end{equation}
 and 
\begin{equation}
\label{eq:QrealizEval_sa}
\ur(X) = \hat{\br}(X) \qquad\text{if $X \in \dom_\cA^\sa(\ur)$.}
\end{equation}
\end{quote}
\end{itemize}
\end{thm}

\begin{proof}
For proving (i), we start with any descriptor realization $\br$ of $\ur$ like in Theorem \ref{thm:realizations_matrices_of_rational_expressions} (i). Then we know
\begin{itemize}
 \item by the properties of $\br$ stated in Theorem \ref{thm:realizations_matrices_of_rational_expressions} (i) that $\dom_\cA(\ur) \subseteq \dom_\cA(\br)$ holds and that we have $$\ur(X) = \br(X) \qquad\text{for all $X \in \dom_\cA(\ur)$}.$$
 \item by Item \ref{it:domcutdownMin} of Lemma \ref{lem:domMin} that $\dom_\cA(\br) \subseteq \dom_\cA(\hat{\br})$ holds and that moreover $$\br(X) = \hat{\br}(X) \qquad\text{for all $X \in \dom_\cA(\br)$}.$$
\end{itemize}
In summary, we get a chain of inclusions $\dom_\cA(\ur) \subseteq \dom_\cA(\br) \subseteq \dom_\cA(\hat{\br})$ proving \eqref{eq:QrealizDom} and we see that $\ur(X) = \br(X) = \hat{\br}(X)$ holds for all $X\in\dom_\cA(\ur)$, which finally shows the validity of \eqref{eq:QrealizEval}.

For proving (ii), we proceed as follows. Since $\ur$ is regular at zero, we may consider besides its selfadjoint minimal realization
$$\hat{\br}(x) = \Delta + \hat{\Xi}^\ast (\hat{J} - L_{\hat{M}}(x))^{-1} \hat{\Xi}$$
any other selfadjoint descriptor realization
$$\br(x) = \Delta + \Xi^\ast (J - L_M(x))^{-1}  \Xi,$$ 
with the prescribed feed through term $\Delta$, as constructed in Theorem \ref{thm:realizations_matrices_of_rational_expressions} (ii). Thus, if $\cA$ is any unital complex/real $\ast$-algebra that is stably finite, we know
\begin{itemize}
 \item by part (ii) of Theorem \ref{thm:realizations_matrices_of_rational_expressions} that $\dom_\cA^\sa(\ur) \subseteq \dom_\cA(\br)$ holds and that we have $$\ur(X) = \br(X) \qquad\text{for any $X \in \dom_\cA^\sa(\ur)$}.$$
 \item by Item \ref{it:domcutdownMin} of Lemma \ref{lem:domMin} that $\dom_\cA(\br) \subseteq \dom_\cA(\hat{\br})$ holds and that moreover $$\br(X) = \hat{\br}(X) \qquad\text{for any $X \in \dom_\cA(\br)$}.$$
\end{itemize}
Combining both observation proves the stated inclusion \eqref{eq:QrealizDom_sa} and also the representation given in \eqref{eq:QrealizEval_sa}.
\end{proof}


\section{Free Probability}
\label{sec:FP}

Free probability theory was invented around 1985 by D. Voiculescu as a tool to attack the isomorphism problem for the free group factors $L(\mathbb{F}_n)$. Although this initial question is still open, free probability gave deep insights to this problem and provides presently many powerful tools which are also applied in other fields of mathematics like random matrix theory.

Roughly speaking, free probability theory can be seen as a highly noncommutative counterpart of classical probability, where the notion of classical independence is replaced by free independence. In the initial example $L(\mathbb{F}_n)$, free independence reflects the structure that is induced on the operator algebraic side by free products on the group side.

Some surprising connection to random matrix theory was found by Voiculescu. He observed that free independence shows up for many classes of independent random matrices in the limit when their dimension tends to infinity.

For more information on free probability in general we refer the reader to the monographs \cite{VDN,HP,NS,MS}.

\subsection{A quick introduction to free probability}
\label{subsec:free_prob_scalar} 

The underlying object of free probability theory are \textbf{noncommutative probability spaces $(\cA,\phi)$}\index{noncommutative probability space!scalar-valued}. In the most general case, namely in a purely algebraic setting, it consists of a unital complex algebra $\cA$ and a linear functional $\phi: \cA\to\CC$ that satisfies $\phi(1_\cA)=1$. Elements of $\cA$ are called \df{noncommutative random variables} and we refer to $\phi$ as \df{expectation} on $\cA$.

Given noncommutative random variables $X_1,\dots,X_g$ in some noncommutative probability space $(\cA,\phi)$, we call the linear mapping
$$\mu_{X_1,\dots,X_g}:\ \CC\langle x_1,\dots,x_g\rangle \to \CC,\quad P \mapsto \phi(P(X_1,\dots,X_g))$$
the \textbf{(joint) noncommutative distribution of $X_1,\dots,X_g$}\index{noncommutative distribution} and we refer to the encoded values $\mu_{X_1,\dots,X_g} (x^w) = \phi(X^w)$ for $w\in\WOR_g$ as the \textbf{mixed moments of $X_1,\dots,X_g$}\index{mixed moments}.
If $\cA$ is even a $\ast$-algebra, the joint noncommutative distribution of $X_1,\dots,X_g,X_1^\ast,\dots,X_g^\ast$ is called the \textbf{(joint) noncommutative $\ast$-distribution of $X_1,\dots,X_g$}\index{noncommutative $\ast$-distribution}.

Free probability is formulated in this language of noncommutative probability theory, but it is characterized by its own notion of independence, which we shall introduce next.

\begin{definition}[Free independence]
Let $(\cA,\phi)$ be a noncommutative probability space and let $(\cA_i)_{i\in I}$ be a family of unital subalgebras of $\cA$ with an arbitrary index set $I\neq\emptyset$. We call $(\cA)_{i\in I}$ \textbf{freely independent}\index{free independence!for subalgebras} (or just \textbf{free}\index{freeness!for subalgebras}), if
$$\phi(X_1 \cdots X_n) = 0$$
holds whenever the following conditions are fulfilled:
\begin{itemize}
\item We have $n\geq 1$ and there are indices $i_1,\dots,i_n\in I$ satisfying $$i_1 \neq i_2, \dots, i_{n-1} \neq i_n.$$
\item For $j=1,\dots,n$, we have $X_j \in \cA_{i_j}$ and it holds true that $\phi(X_j) = 0$.
\end{itemize}
Elements $(X_i)_{i\in I}$ are called \textbf{freely independent}\index{free independence!for elements} (or just \textbf{free}\index{freeness!for elements}), if $(\cA_i)_{i\in I}$ are freely independent in the above sense, where $\cA_i$ denotes for each $i\in I$ the subalgebra of $\cA$ that is generated by $1$ and $X_i$.
\end{definition}

Here, we will mainly work in the more regular setting of $C^\ast$-probability spaces: if $\cA$ is a $C^\ast$-algebra and $\phi: \cA\rightarrow\CC$ a positive state on $\cA$, we call $(\cA,\phi)$ a \textbf{$C^\ast$-probability space}\index{noncommutative probability space!scalar-valued!$C^\ast$-probability space}.

In this case, there corresponds to each element $X=X^\ast \in \cA$ a unique probability measure $\mu_X$ on the real line $\RR$ (compactly supported on the spectrum of $X$) that is determined by the condition that $\mu_X$ has the same moments as $X$ with respect to $\phi$, i.e.
$$\phi(X^k) = \int_\RR t^k\, d\mu_X(t) \qquad\text{for all $k\in\NN$}.$$
We call $\mu_X$ the \df{(analytic) distribution} of $X$. Note that the analytic distribution $\mu_X$ encodes all moments of $X$ and hence the noncommutative distribution of $X$, which justifies to use the same symbol for both.

It was a fundamental observation of Voiculescu, that the distribution $\mu_{X+Y}$ of $X+Y$ for freely independent elements $X,Y\in\cA$ only depends on the distributions $\mu_X$ and $\mu_Y$. Thus, we may write $\mu_{X+Y} = \mu_X \boxplus \mu_Y$ and we call this operation $\boxplus$ the \textbf{free additive convolution}\index{free convolution!additive}.

Only for the sake of completeness, we want to mention that there is another operation $\boxtimes$, the so-called \textbf{multiplicative free convolution}\index{free convolution!multiplicative}, which constitutes a multiplicative counterpart of $\boxplus$. The multiplicative free convolution $\mu_X \boxtimes \mu_Y$ of distributions $\mu_X$ and $\mu_Y$ is defined whenever at least one of the involved operators $X$ and $Y$ is positive. If we suppose for instance that $X$ is positive, then $\mu_X \boxtimes \mu_Y := \mu_{X^{1/2} Y X^{1/2}}$ holds.

The main tool for calculating the free additive convolution is the \df{Cauchy transform}, which is in classical probability (up to sign) also known under the name \df{Stieltjes transform}. The Cauchy transform $G_\mu$ of any probability measure $\mu$ on $\RR$ is the regular function defined by
$$G_\mu(z) := \int_\RR \frac{1}{z-t}\, d\mu(t) \qquad\text{for all $z\in\CC^+$},$$
where $\CC^+ := \{z\in\CC| \Im(z)>0\}$ denotes to upper half-plane. It is easy to check that $G_\mu$ maps $\CC^+$ to the lower half-plane $\CC^- := \{z\in\CC| \Im(z)<0\}$.

Note, that if $\mu_X$ is the distribution of any noncommutative random variable $X=X^\ast\in\cA$ in a $C^\ast$-probability space $(\cA,\phi)$, we have
$$G_{\mu_X}(z) = \int_\RR \frac{1}{z-t}\, d\mu_X(t) = \phi\big((z-X)^{-1}\big) \qquad\text{for $z\in\CC^+$}.$$
Thus, in such cases, we often write $G_X$ instead of $G_{\mu_X}$.

\begin{rem}\label{rem:Stieltjes}
The \df{Stieltjes inversion formula} tells us, that $\mu$ can be recovered from its Cauchy transform $G_\mu: \CC^+\to\CC^-$. In fact, the absolutely continuous probability measures $\mu_\varepsilon$ given by
\begin{equation}\label{eq:Stieltjes}
d\mu_\varepsilon(t) = \frac{-1}{\pi} \Im(G_\mu(t+i\varepsilon))\, dt
\end{equation}
converge in distribution to $\mu$ as $\varepsilon\searrow0$, i.e. we have
$$\lim_{\varepsilon\searrow 0} \int_\RR f(t)\, d\mu_\varepsilon(t) = \int_\RR f(t)\, d\mu(t)$$
for all bounded continuous functions $f: \RR\to\CC$.
\end{rem}

\subsection{Polynomials in free random variables}
\label{subsec:polynomials}

Free probability is deeply connected to random matrix theory. The reason for this surprising link is that free independence turns out to describe the limit behavior of independent random matrices of many types when their size tends to infinity; this phenomenon is known under the name of asymptotic freeness.
More formally, this means that a $g$-tuple $(X_1^{(N)},\dots,X_g^{(N)})$ of such random matrices, say selfadjoint and of size $N\times N$, converge in distribution as $N\to\infty$ to a family $(X_1,\dots,X_g)$ of freely independent elements in some noncommutative probability space. In particular, if any selfadjoint polynomial $P$ in $g$ indeterminates $x_1,\dots,x_g$ is given, then the asymptotic eigenvalue distribution of the random matrix
$$Y^{(N)} := P(X_1^{(N)},\dots,X_g^{(N)})$$
is given by the distribution of the operator
$$Y := P(X_1,\dots,X_g).$$
Thus, starting with a fundamental question in random matrix theory, asymptotic freeness leads us to the some problem that can be formulated in the language of free probability theory:
\begin{quote}
\textit{Suppose that $P\in\CC\langle x_1,\dots,x_g\rangle$ is a selfadjoint noncommutative polynomial and let freely independent noncommutative random variables $X_1,\dots,X_n$ in some $C^\ast$-probability space be given. How can we calculate the analytic distribution of $P(X_1,\dots,X_g)$ out of the given individual distributions $\mu_{X_1},\dots,\mu_{X_g}$?}
\end{quote}
This is a well-posed problem, since the basic theory of free probability tells us that the distribution of $P(X_1,\dots,X_g)$ is indeed uniquely determined by $P$ and $\mu_{X_1},\dots,\mu_{X_g}$.
In the case $g=2$, for the special selfadjoint polynomial $P = x_1 + x_2$, the question is answered by the free additive convolution $\boxplus$, and for $P = x_1 x_2 x_1$, it is settled by the free multiplicative convolution $\boxtimes$.
For general $g$ and $P$, one can introduce analogously the \textbf{free polynomial convolution}\index{free convolution!polynomial}, sometimes written as
$$\mu_{P(X_1,\dots,X_g)} = P^{\square}(\mu_{X_1},\dots,\mu_{X_g}),$$
but the complexity of this operation is considerably higher than for $\boxplus$ and $\boxtimes$.
This problem was treated systematically for the first time in \cite{BMS2013}, where analytic methods from operator-valued free probability theory were combined with linearization techniques.
Subsequently, in \cite{BSS2015}, this approach was extended to the case of non-selfadjoint polynomials, where the so-called Brown measure served as a suitable replacement for the analytic distribution.
Our main goal is to generalize these two approaches from polynomials to rational expressions with the help of descriptor realizations.
For that purpose, we are going to provide in the next subsections some more details about both operator-valued free probability and the Brown measure.

\subsection{Operator-valued free probability and subordination}

Compared to the scalar-valued setting of free probability as presented in Section \ref{subsec:free_prob_scalar}, the main novelty in the operator-valued frame is that expectations get replaced by some natural noncommutative analogues of conditional expectations in classical probability.

An \textbf{operator-valued $C^\ast$-probability space $(\cA,E,\cB)$}\index{noncommutative probability space!operator-valued $C^\ast$-probability space} consists of a unital $C^\ast$-algebra $\cA$, a unital $C^\ast$-subalgebra $\cB$ of $\cA$, and a \textbf{conditional expectation $E: \cA\to\cB$}\index{expectation!conditional}, i.e., a (completely) positive and unital map $E: \cA\to\cB$ satisfying
\begin{itemize}
  \item $E[B] = B$ for all $B\in\cB$ and
  \item $E[B_1 X B_2] = B_1 E[X] B_2$ for all $X\in\cA$, $B_1,B_2\in\cB$.
\end{itemize}
The definition of free independence in the operator-valued setting reads as follows.

\begin{definition}[Free independence with amalgamation]
Let $(\cA,E,\cB)$ be an operator-valued $C^\ast$-probability space and let $(\cA_i)_{i\in I}$ be a family of subalgebras $\cB \subseteq \cA_i \subseteq \cA$ with an arbitrary index set $I\neq\emptyset$. We call $(\cA)_{i\in I}$ \textbf{freely independent with amalgamation over $\cB$}\index{free independence with amalgamation!for subalgebras} (or just \textbf{free over $\cB$}\index{freeness with amalgamation!for subalgebras}), if
$$E[X_1 \cdots X_n] = 0$$
holds whenever the following conditions are fulfilled:
\begin{itemize}
\item We have $n\geq 1$ and there are indices $i_1,\dots,i_n\in I$ satisfying $$i_1 \neq i_2, \dots, i_{n-1} \neq i_n.$$
\item For $j=1,\dots,n$, we have $X_j \in \cA_{i_j}$ and it holds true that $E[X_j] = 0$.
\end{itemize}
Elements $(X_i)_{i\in I}$ are called \textbf{freely independent with amalgamation over $\cB$}\index{free independence with amalgamation!for elements} (or just \textbf{free with amalgamation over $\cB$}\index{freeness with amalgamation!for elements}), if $(\cA_i)_{i\in I}$ are freely independent with amalgamation over $\cB$ in the above sense, where $\cA_i$ denotes for each $i\in I$ the subalgebra of $\cA$ that is generated by $\cB$ and $X_i$.
\end{definition}

\begin{rem}\label{freeness_matrices}
For our purposes, it is important to note that operator-valued $C^\ast$-probability spaces can easily be constructed by passing to matrices over scalar-valued $C^\ast$-probability spaces. Indeed, if $(\cC,\phi)$ is any $C^\ast$-probability space, then
$$\cA := M_N(\CC)\otimes \cC, \qquad \cB := M_N(\CC), \qquad\text{and}\qquad E := \id_{M_N(\CC)} \otimes \phi$$
give an operator-valued $C^\ast$-probability space $(\cA,E,\cB)$; note that the natural identification $M_N(\CC) \cong M_N(\CC) \otimes 1_\cC$ allows us to treat $\cB$ as an actual subalgebra of $\cA$. Moreover, a direct calculation shows that if $(\cC_i)_{i\in I}$ is a family of freely independent subalgebras of $\cC$, then $\cA_i := M_N(\CC) \otimes \cC_i$ for $i\in I$ defines a family $(\cA_i)_{i\in I}$ of subalgebras of $\cA$, which is freely independent with amalgamation over $\cB$.
\end{rem}

We stress that free probability theory possesses, both in the scalar- and operator-valued setting, a rich combinatorial structure given by non-crossing partitions. Since this approach will not be used in the following, we do not go into details here; the interested reader is referred to \cite{Speicher1998,MS}.

Similar to the scalar-valued case, Cauchy transforms play an important role in the regular theory of free independence with amalgamation, but they need to be generalized in the following way: Let $(\cA,E,\cB)$ be an operator-valued $C^\ast$-probability space. We call
$$\bH^+(\cB) := \{B\in\cB|\ \exists \varepsilon>0:\ \Im(B) \geq \varepsilon 1\}$$
the upper half-plane of $\cB$, where we use the notation $\Im(B) := \frac{1}{2i}(B-B^\ast)$. The \textbf{$\cB$-valued Cauchy transform $G_X$}\index{Cauchy transform!operator-valued} of any $X=X^\ast\in\cA$ is the Fr\'echet holomorphic function defined by
$$G_X(B) := E[(B-X)^{-1}] \qquad\text{for all $B\in\bH^+(\cB)$.}$$
In fact, it induces a map $G_X: \bH^+(\cB) \to \bH^-(\cB)$ from $\bH^+(\cB)$ to the lower half-plane $\bH^-(\cB)$ defined by
$$\bH^-(\cB) := \{B\in\cB|\ \exists \varepsilon>0:\ -\Im(B) \geq \varepsilon 1\}.$$

The most convenient way to deal with free additive convolution, both in the scalar- and in the operator-valued setting, is to use subordination, since it is easily accessible for numerical computations. Before we  give the precise statement, we introduce the following regular transforms, which are both related to Cauchy transforms, namely
\begin{itemize}
\item the reciprocal Cauchy transform $F_X: \bH^+(\cB) \to \bH^+(\cB)$ by $$F_X(B)=E\left[(B-X)^{-1}\right]^{-1}=G_X(B)^{-1},$$
\item and the h transform $h_X: \bH^+(\cB) \to \overline{\bH^+(\cB)}$ by $$h_X(B)=E\left[(B-X)^{-1}\right]^{-1}-B=F_X(B)-B.$$
\end{itemize}
Note, that these mappings are indeed well-defined since it has been shown in \cite{BPV2012} that $
\Im(F_X(B))\geq\Im(B)$ for all $B\in\bH^+(\cB)$, which immediately implies $\Im(h_X(B))\geq0$ for all $B\in\bH^+(\cB)$.

\begin{thm}[\cite{BMS2013}]\label{thm:subordination}
Assume that $(\cA, E,\cB)$ is a $C^*$-operator-valued noncommutative
probability space and $X,Y\in\mathcal \cA$ are two selfadjoint operator-valued random variables
free over $\cB$. Then there exists a unique pair of Fr\'echet (and thus also G\^{a}teaux) regular maps $\omega_1,\omega_2: \bH^+(\cB)\to\bH^+(\cB)$ so that
\begin{enumerate}
\item $\Im(\omega_j(B))\geq\Im(B)$ for all $B\in\bH^+(\cB)$ and $j\in\{1,2\}$,
\item $F_X(\omega_1(B))+B=F_Y(\omega_2(B))+B=\omega_1(B)+\omega_2(B)$ for all $B\in\bH^+(\cB)$,
\item $G_X(\omega_1(B))=G_Y(\omega_2(B))=G_{X+Y}(B)$ for all $B\in\bH^+(\cB).$
\end{enumerate}
Moreover, if $B\in\bH^+(\cB)$, then $\omega_1(B)$ is the unique fixed point of the map
$$
f_B:\ \bH^+(\cB)\to\bH^+(\cB),\quad f_B(W)=h_Y(h_X(W)+B)+B,
$$
and $\omega_1(B)=\lim_{n\to\infty}f_B^{\circ n}(W)$ for any $W\in\bH^+(\cB)$, where $f_B^{\circ n}$ means the $n$-fold composition of $f_B$ with itself. Same statements hold for $\omega_2$, with $f_B$ replaced by $W\mapsto h_X(h_Y(W)+B)+B$.
\end{thm}

\begin{rem}\label{rem:Cauchy_matval}
Below, we will apply Theorem \ref{thm:subordination} mostly in situations like in Remark \ref{freeness_matrices}. Thus, it becomes necessary to compute the matrix-valued Cauchy transform $G_{\hat{X}}$ of an operator $\hat{X} = \Lambda \otimes X$ for some selfadjoint matrix $\Lambda\in M_N(\CC)$ and a selfadjoint operator $X \in \cC$ with known analytic distribution $\mu_X$, where the latter is typically given in terms of its scalar-valued Cauchy transform $G_X$.
Since the operator-valued distribution of $\hat{X}$ is fully determined by $\mu_X$, this is indeed a well-posed problem, making this dependency explicit and efficiently accessible to numerical computations is however an intricate problem.
The first attempt following relation
$$G_{\Lambda \otimes X}(B) = \int_{\RR} (B - t \Lambda)^{-1}\, d\mu_X(t)$$
with the matrix-valued integral understood in the Bochner sense, from which we can deduce with the help of Stieltjes inversion formula (see Remark \ref{rem:Stieltjes}) that
\begin{equation}\label{eq:Cauchy-transform_opval_integration}
G_{\Lambda \otimes X}(B) = \lim_{\varepsilon\searrow 0} \frac{-1}{\pi} \int_{\RR} (B - t \Lambda)^{-1} \Im(G_X(t+i\varepsilon))\, dt.
\end{equation}
From a computational point of view, however, formula \eqref{eq:Cauchy-transform_opval_integration} is not really satisfying, since it requires computing an integral with unbounded support.
Using some small amount of linear algebra, we may provide a significantly simplified approach. Roughly, this is done in the following steps:
\begin{itemize}
 \item Diagonalize $\Lambda$ by some unitary matrix $U$ such that $$U^\ast \Lambda U = \begin{pmatrix} \Lambda_0 & 0\\ 0 & 0\end{pmatrix},$$ where $\Lambda_0$ is a diagonal matrix with diagonal entries $\lambda_1,\dots,\lambda_d$, being the non-zero eigenvalues of $\Lambda$ listed according to their multiplicities.
 \item Consider the block decomposition $$U^\ast B U = \begin{pmatrix} B_{1,1} & B_{1,2}\\ B_{2,1} & B_{2,2}\end{pmatrix}$$ with $B_{1,1}$ belonging to $M_d(\CC)$ and with all other blocks of appropriate size. Since $B_{2,2}$ is invertible due to $\Im(B_{2,2}) > 0$, we may introduce $B_0 := B_{1,1}-B_{1,2}B_{2,2}^{-1}B_{2,1}$, and the Schur complement formula tells us $$G_{\Lambda \otimes X}(B) = U \begin{pmatrix} I & 0\\ -B_{2,2}^{-1}B_{2,1} & I\end{pmatrix} \begin{pmatrix} G_{\Lambda_0 \otimes X}(B_0) & 0\\ 0 & B_{2,2}^{-1}\end{pmatrix} \begin{pmatrix} I & -B_{1,2}B_{2,2}^{-1}\\ 0 & I\end{pmatrix} U^\ast.$$
 \item In order to compute $G_{\Lambda_0 \otimes X}(B_0)$, we proceed as follows: notice that $G_{\Lambda_0 \otimes X}(B_0) = \widehat{G_{X}}(\Lambda_0^{-1} B_0) \Lambda_0^{-1}$ holds, where $\widehat{G_X}$ denotes the fully matricial extension of the scalar-valued Cauchy transform $G_X$ of $X$ to the noncommutative set $\Omega(X) := \coprod_{n=1}^\infty \Omega_n(X)$ with $$\Omega_n(X) := \{A\in M_n(\CC)|\ \text{$A \otimes 1_\cA - I_n \otimes X$ is invertible in $M_n(\CC) \otimes \cA$}\},$$ which is given by the holomorphic functional calculus.
\end{itemize} 
\end{rem}

Remark \ref{freeness_matrices} and Theorem \ref{thm:subordination} constitute the main ingredients from operator-valued free probability for solving the questions formulated in Section \ref{subsec:polynomials}; this will done in Section \ref{subsec:main_problems}.

\subsection{Brown measures}
\label{subsec:Brown}

If a noncommutative random variable $X$ living in some $C^\ast$-probability space $(\cA,\phi)$ fails to be selfadjoint, one needs to work with its $\ast$-distribution in order to capture its spectral properties properly. However, we typically loose then the analytic description in terms of probability measures. The only exception are normal operators, but also in such cases, the machinery of Cauchy transforms that we use for selfadjoint operators is sufficient anymore.

An appropriate substitute for the notion of analytic distributions, when going beyond the case of selfadjoint operators, is the so-called Brown measure, which was introduced in \cite{Brown1986} and revived in \cite{HL}. As we shall see below, they enjoy the important feature that they are accessible by means of operator-valued free probability theory.

In order for this theory to work we need to stay within in the setting of finite von Neumann algebras. The main point is that we need our state $\phi$ to be a trace.
Hence we will in the following discussions around the Brown measure always assume that our noncommutative probability space $(\cA,\phi)$ is actually a \df{tracial $W^*$-probability space}, which means that $\cA$ is a von Neumann algebra and $\phi$ is a faithful, normal trace.
This is however not an actual restriction of generality, because many concrete situations appearing in free probability theory are embedded in a tracial $W^*$-probability setting anyway.
Furthermore, since the presence of a faithful trace guarantees that we are in a stably finite situation, this goes very well also with our observation that we need stably finite algebras as domains of our rational functions in order to ensure good evaluation properties.

Given an arbitrary element $X$ in any tracial $W^\ast$-probability space $(\cA,\phi)$, we may define its \df{Fuglede-Kadison determinant} $\Delta(X)$ via the equation
\begin{equation}\label{eq:FKdet}
\log(\Delta(X)) = \lim_{\varepsilon \searrow 0} \frac{1}{2}\phi(\log(X X^\ast+\varepsilon^2)).
\end{equation}
It was shown in \cite{Brown1986}, that the function $z\mapsto \frac{1}{2\pi}\log(\Delta(X-z{\bf 1}))$ is subharmonic on $\CC$ and harmonic outside the spectrum of $X$. Thus, the Riesz Decomposition Theorem (cf. \cite[Theorem 3.7.9]{Ransford1995}) shows that there exists a Radon measure $\mu_X$ on $\CC$ such that
$$\int_\CC \psi(z)\, d\mu_X(z) = \frac{1}{2\pi} \int_\CC \Big(\frac{\partial^2 \psi}{\partial x^2}(z) + \frac{\partial^2\psi}{\partial y^2}(z)\Big)\log(\Delta(X-z{\bf 1}))\, d\lambda^2(z)$$
holds for all functions $\psi\in C^\infty_c(\CC)$. Here, we denote by $\lambda^2$ the Lebesgue measure on $\CC$, which is induced under the usual identification of $\CC$ with $\RR^2$. In other words, the \df{Brown measure} $\mu_X$ is the \df{generalized Laplacian} of $z\mapsto \frac{1}{2\pi}\log(\Delta(X-z{\bf 1}))$, which means that $\mu_X$ is determined (in the distributional sense) by
\begin{equation}\label{eq:BrownDef}
\mu_X = \frac{2}{\pi} \frac{\partial}{\partial z} \frac{\partial}{\partial\overline{z}} \log(\Delta(X-z)).
\end{equation}
Note that we made use of the fact that, on $C^2$-functions, the usual Laplacian $\frac{\partial^2}{\partial x^2} + \frac{\partial^2}{\partial y^2}$ can be rewritten as
$$\frac{\partial^2}{\partial x^2} + \frac{\partial^2}{\partial y^2} = 4 \frac{\partial}{\partial z} \frac{\partial}{\partial\overline{z}}$$
in terms of the Pompeiu-Wirtinger derivatives
$$\frac{\partial}{\partial z} = \frac{1}{2}\Big(\frac{\partial}{\partial x} - i\frac{\partial}{\partial y}\Big) \qquad\text{and}\qquad \frac{\partial}{\partial\overline{z}} = \frac{1}{2}\Big(\frac{\partial}{\partial x} + i\frac{\partial}{\partial y}\Big).$$

In order to compute the Brown measure $\mu_X$ of $X$, it is convenient to approximate $\mu_X$ by certain regularizations $\mu_{X,\varepsilon}$. The \textbf{regularized Brown measures $\mu_{X,\varepsilon}$}\index{Brown measure!regularized} are obtained by replacing in the defining equation \eqref{eq:BrownDef} of $\mu_X$ the Fuglede-Kadison determinant $\Delta$ by a certain regularization $\Delta_\varepsilon$. In analogy to \eqref{eq:FKdet}, the \textbf{regularized Fuglede-Kadison determinant $\Delta_\varepsilon$}\index{Fuglede-Kadison determinant!regularized} is determined by
\begin{equation}
\log(\Delta_\varepsilon(X)) = \frac{1}{2}\phi(\log(X X^\ast+\varepsilon^2)).
\end{equation}
Explicitly and again in the distributional sense, this means that
\begin{equation}\label{eq:regBrownDef}
\mu_{X,\varepsilon}(z) = \frac{2}{\pi} \frac{\partial}{\partial z} \frac{\partial}{\partial\overline{z}} \log(\Delta_\varepsilon(X-z)).
\end{equation}
If we consider the \textbf{regularized Cauchy transform}\index{Cauchy transform!regularized}
$$G_{X,\varepsilon}(z) = \phi\big((z-X)^\ast\big((z-X)(z-X)^\ast + \varepsilon^2\big)^{-1}\big),$$
which is a $C^\infty$-function on $\CC$ (but obviously not holomorphic), then we may rewrite \eqref{eq:regBrownDef} as
\begin{equation}\label{eq:BrownCauchy}
d\mu_{X,\varepsilon}(z) = \frac{1}{\pi} \frac{\partial}{\partial\overline{z}} G_{X,\varepsilon}(z)\, d\lambda^2(z).
\end{equation}
The latter can be seen as an analogue of Stieltjes inversion formula \eqref{eq:Stieltjes}.

Since free probability theory -- both in the scalar- and operator-valued setting -- provides powerful tools to deal with Cauchy transforms, the formula for $\mu_{X,\varepsilon}$ given in \eqref{eq:BrownCauchy} looks rather appealing. However, there is the disadvantage that $G_{X,\varepsilon}$ is fairly close to a usual Cauchy transform but still a totally different object.

Fortunately, we can calculate $G_{X,\varepsilon}$ by using the so-called hermitian reduction method, which goes back to \cite{JNPZ97} and was taken up in \cite{BSS2015}.
This method is based on the $M_2(\CC)$-valued $C^\ast$-probability space $(M_2(\cA),E,M_2(\CC))$, where $E$ denotes the conditional expectation that is induced by $\id_{M_2(\CC)} \otimes \phi$ under the identification $M_2(\cA) \cong M_2(\CC) \otimes \cA$. We consider the selfadjoint element
$$\bX := \begin{pmatrix} 0 & X\\ X^\ast & 0\end{pmatrix} \in M_2(\cA).$$
The regularized Cauchy transform $G_{X,\varepsilon}(z)$ can then be obtained as the $(2,1)$-entry of the $M_2(\CC)$-valued Cauchy transform of $\bX$, if it is evaluated at the point
$$\Lambda_\varepsilon(z) := \begin{pmatrix} i\varepsilon & z\\ \overline{z} & i\varepsilon\end{pmatrix} \in \bH^+(M_2(\CC)).$$
More precisely, we have for each $z\in\CC^+$ that
\begin{equation}\label{eq:hermCauchy}
G_{X,\varepsilon}(z) = [G_{\bX}(\Lambda_\varepsilon(z))]_{2,1}.
\end{equation}

Collecting our observations, we see that we can compute the regularized Brown measures $\mu_{X,\varepsilon}$ with the help of \eqref{eq:BrownCauchy} from its regularized Cauchy transforms $G_{X,\varepsilon}$, whereas the regularized Cauchy transform $G_{X,\varepsilon}$ itself can be deduced by \eqref{eq:hermCauchy} from the $M_2(\CC)$-valued Cauchy transform of the selfadjoint element $\bX$. So, in this way, also Brown measures become accessible via Cauchy transforms.

\subsection{Cauchy Transforms of Matrices of NC Rational Expressions}

In this section, we finally bring together our ``linearization'' techniques presented in Sections \ref{sec:Realization} and developed further in Sections \ref{sec:Algorithm} and \ref{sec:evaluations} with tools from (operator-valued) free probability theory in order to settle the problem discussed in Section \ref{subsec:polynomials}.
Our approach follows ideas of both \cite{BMS2013} and \cite{BSS2015}, but while they could only treat the case of polynomials, we can deal more generally with rational expressions, since formal linear representations, and the powerful machinery of descriptor realizations in particular, provide a vast generalization of the linearization techniques from \cite{And}.
We will work here even with matrices of rational expressions, and in the case where they are regular at $0$, we can take advantage of the fact that we can pass to minimal descriptor realizations.

\subsubsection{Minimal Descriptor Realizations of Matrices of Rational Expressions}

The first step is the following likely known lemma in which we combine earlier results on the existence and uniqueness of minimal descriptor realizations of matrices of rational functions. For reasons of clarity, we first restrict ourselves to the case of (symmetric) real matrices, but in the subsequent Remark \ref{rem:real_to_complex}, we will extend it to the complex case.

\begin{lem}
\label{lem:descr}
Let $r$ be a symmetric $k \times k$ matrix of rational expressions 
regular at 0 in symmetric variables $x=(x_1,\dots,x_g)$. 
Then $r$ has a minimal (i.e., controllable and observable) symmetric descriptor realization
\begin{equation}
  \label{eq:descr}
   \br(x) = \Delta +  \Xi^\transpose ( M_0 - L_M(x) )^{-1}  \Xi, \qquad L_M(x) = M_1 x_1 + \ldots + M_g x_g,
\end{equation}
with a symmetric $k \times k$ matrix $\Delta$, symmetric $N\times N$ matrices $M_j$, $j=0,\dots,g$ with $M_0^2 = I_N$, 
and a $N\times k$ matrix $\Xi$.

The representation \eqref{eq:descr} is unique in the sense that another such representation
 ${\widetilde \Delta}, {\widetilde \Xi},  {\widetilde M}_j,\  j=0,\dots,n,$ for the same $r$ and with $ \widetilde \Delta = \Delta$ satisfies
\begin{equation}
SM_0 M_j = \tM_0 \widetilde{M}_j S  \qquad      
SM_0  \Xi = \tM_0 \widetilde{\Xi}   \qquad
\Xi^\transpose = \widetilde{\Xi}^\transpose S
\end{equation}
for some invertible matrix $S$ satisfying $S^\transpose M_0 S = M_0$.
\end{lem}

\begin{proof}
For scalar rational expressions $r$, existence of this minimal realization rephrases Item \ref{lem:symR2} of Lemma \ref{lem:symR}; for general $r$, this follows from Theorem \ref{thm:realizations_matrices_of_rational_expressions} with the help of cutting down as explained in Section \ref{sec:cutdown}. Uniqueness rephrases Proposition \ref{prop:uniquedescriptor}.
\end{proof}

\begin{rem}
\label{rem:real_to_complex}
The rational expression $r$ can have complex coefficients and if it is selfadjoint, then \eqref{eq:descr} holds with matrices $\Xi$ and selfadjoint matrices $\Delta$ and $M_j$, $j=1,\dots,n$, having possibly complex entries. Here $\Xi^\transpose$, $S^\transpose$ become conjugate transpose $\Xi^\ast$, $S^\ast$. This follows from the identification of complex numbers $z= a + i b$ with matrices in $M_2(\RR)$ of the form
 $$ \eta_z := \bmat a & b \\
              -b & a
              \emat
 $$
Note $\bar z = \eta_z^\transpose$.
\end{rem}

\subsubsection{Representation of Cauchy transforms}

In this section, we want to formulate our main result on the representation of Cauchy transforms. For that purpose, it is convenient to introduce the following terminology.

\begin{definition}\label{def:realized_by}~

\begin{enumerate}

 \item\label{it:realized_by-1} A \textbf{generalized descriptor realization (of size $k_1\times k_2$)}\index{generalized descriptor realization} is a $k_1\times k_2$ matrix-valued rational expression
$$\br(x) = \Delta + \Xi_1 \Lambda(x)^{-1} \Xi_2$$
with matrices $\Delta \in M_{k_1 \times k_2}(\CC)$, $\Xi_1 \in M_{k_1 \times N}(\CC)$, and $\Xi_2 \in M_{N \times k_2}(\CC)$, and a linear pencil $\Lambda(x)$ of size $N\times N$.

In the case that $k_1 = k_2$, a generalized descriptor realization is called \textbf{selfadjoint}\index{generalized descriptor realization!selfadjoint} if $\Delta=\Delta^\ast$, $\Xi_1 = \Xi_2^\ast$, and $\Lambda^\ast = \Lambda$ holds.

 \item\label{it:realized_by-2} Let $r$ be a selfadjoint $k \times k$ matrix of rational expressions in the variables $x_1,\dots,x_g$ and let $X_1,\dots,X_g$ be selfadjoint elements in some $C^\ast$-probability space $(\cA,\phi)$, such that $X = (X_1,\dots,X_g) \in \dom_\cA(r)$ is satisfied.
Assume that
$$\br(x) = \Delta + \Xi^\ast \Lambda(x)^{-1} \Xi$$
is a selfadjoint generalized descriptor realization of size $k \times k$, such that the operator $\Lambda(X) \in M_N(\CC) \otimes \cA$ is invertible (i.e., that $X\in\dom_\cA(\br)$) and
$$r(X) = \br(X) = \Delta + \Xi^\ast \Lambda(X)^{-1} \Xi$$
holds true. Then we say that \textbf{$r$ is realized at $X$ by $\br$ (in dimension $N$)}\index{generalized descriptor realization!realized by}.

\end{enumerate}
\end{definition}

Generalized descriptor realizations can be introduced also in the real case. Definition \ref{def:realized_by} \eqref{it:realized_by-1} indeed generalizes the notion of descriptor realizations since they do not require the invertibility of the scalar matrix $\Lambda(0)$. Note that Definition \ref{def:realized_by} \eqref{it:realized_by-2} allows $\br$ to depend on $X$. 

The first part of our main statement is the following theorem, which explains how the Cauchy transform of $r(X)$ can be computed if $r$ is realized at the point $X$ by some $\br$.

Recall from Remark \ref{freeness_matrices} that any $C^\ast$-probability space $(\cA,\phi)$ gives rise to an operator-valued $C^\ast$-probability space $(M_N(\CC) \otimes \cA,E_N,M_N(\CC))$ for each $N\in\NN$, where $E_N$ denotes the conditional expectation given by $E_N := \id_{M_N(\CC)} \otimes \phi$.

\begin{thm}\label{thm:mainrep1}
Let $r$ be a selfadjoint $k \times k$ matrix of rational expressions in the variables $x_1,\dots,x_g$ and let $X_1,\dots,X_g$ be selfadjoint elements in some $C^\ast$-probability space $(\cA,\phi)$, such that $X = (X_1,\dots,X_g) \in \dom_\cA(r)$.
Assume that $r$ is realized at $X$ in dimension $N$ by some selfadjoint generalized descriptor realization
$$\br(x) = \Delta + \Xi^\ast \Lambda(x)^{-1} \Xi$$
in the sense of Definition \ref{def:realized_by}. Consider the linear pencil given by
\begin{equation}\label{eq:shifted_linearization}
\hpen(x) := \begin{pmatrix} \Delta & \Xi^\ast\\ \Xi & -\Lambda(x)\end{pmatrix}.
\end{equation}
Then, the following statements hold true:

\begin{itemize}
  
  \item[(i)] For all $B\in\bH^+(M_k(\CC))$, we have that
\begin{equation}
\label{eq:repRes}
\big(B \otimes 1_\cA - r(X)\big)^{-1}
 = \kar
\left(\begin{pmatrix} B & 0\\ 0 & 0\end{pmatrix} \ot \IA
 - \hpen(X)\right)^{-1} 
 \kac.
\end{equation}

 \item[(ii)] The $M_k(\CC)$-valued Cauchy transform $G_{r(X_1,\dots,X_g)}$ (calculated with respect to the conditional expectation $E_k$) is determined by the $M_{N+k}(\CC)$-valued Cauchy transform $G_{\hpen(X_1,\dots,X_g)}$ (calculated with respect to the conditional expectation $E_{N+k}$) by
\begin{equation}
\label{eq:repG}
\begin{aligned}
G_{r(X)} (B) = \lim_{\varepsilon\searrow 0}
  \bmat I_k & 0 \emat
G_{\hpen(X)} \left( \begin{pmatrix} B & 0\\ 0 & i \varepsilon I_N\end{pmatrix} \right)
\bmat I_k \\ 0 \emat.
\end{aligned}
\end{equation}
for all $B\in\bH^+(M_k(\CC))$.

\end{itemize}

\end{thm}

Before giving the proof, we point out that the linear pencil $\hpen$, as defined in \eqref{eq:shifted_linearization}, can be seen as a shifted version of the \textbf{linearizations}\index{linearization} considered in \cite{And,BMS2013,BSS2015}; see also the example \eqref{eq:phat} that was discussed at the beginning.

\begin{proof}[Proof of Theorem \ref{thm:mainrep1}]
The proof of (i) relies in the Schur complement formula, which was outlined in Section \ref{subsec:Schur}. Let us abbreviate
$$\cE_k = \bmat I_k \\ 0 \emat \qquad \text{and} \qquad  \rho_\cA = \bmat I_k \otimes \IA \\ 0 \emat.$$
Since $r$ is realized at $X$ by $\br$, we know that $\Lambda(X)$ is invertible in $M_N(\CC) \otimes \cA$. Thus, the Schur complement formula tells us, for any $B\in M_k(\CC)$, that
$$\begin{pmatrix} B \otimes 1_\cA -\Delta \otimes 1_\cA & -\Xi^\ast \otimes 1_\cA\\ -\Xi\otimes1_\cA & \Lambda(X)\end{pmatrix}$$
is invertible in $M_{N+k} \otimes \cA$ if any only if its Schur complement
$$B \otimes 1_\cA - \big(\Delta \otimes 1_\cA + (\Xi^\ast \otimes 1_\cA) \Lambda(X)^{-1} (\Xi \otimes 1_\cA) \big) = B \otimes 1_\cA - r(X)$$
is invertible in $M_k(\CC) \otimes \cA$. Hence, since $r(X)$ is selfadjoint and thus $B \otimes 1_\cA - r(X)$ must be invertible at least for all $B\in \bH^+(M_k(\CC))$, we know that
$$\begin{pmatrix} B \otimes 1_\cA - \Delta \otimes 1_\cA & -\Xi^\ast \otimes 1_\cA\\ -\Xi \otimes 1_\cA & \Lambda(X)\end{pmatrix}$$
is invertible for all $B\in \bH^+(M_k(\CC))$. In this case, \eqref{eq:Schur} yields that
\begin{align*}
\rho_\cA^\ast \left(\begin{pmatrix} B & 0\\ 0 & 0\end{pmatrix} \otimes 1_\cA - \hpen(X)\right)^{-1} \rho_\cA
&= \rho_\cA^\ast \begin{pmatrix} B \otimes 1_\cA - \Delta \otimes 1_\cA & -\Xi^\ast \otimes 1_\cA \\ -\Xi \otimes 1_\cA & \Lambda(X)\end{pmatrix}^{-1} \rho_\cA \\
&= \big(B\otimes 1_\cA - r(X)\big)^{-1},
\end{align*}
which is the stated formula \eqref{eq:repRes}.

For seeing \eqref{eq:repG}, we first note that by definition
$$E_k [ \ra^* W \ra ] =  \cE_k^* \;  E_{N+k} [ W ] \;  \cE_k \qquad\text{for any $W \in M_{N+k}(\CC) \otimes \cA$}.$$
Thus, we get by applying $E_k$ to both sides of \eqref{eq:repRes} that
$$ \cE_k^* \; E_{N+k} \
 \left[  \left(\begin{pmatrix} B & 0\\ 0 & 0\end{pmatrix}  \ot \IA
 - \hpen(X)\right)^{-1}  \right] \cE_k
 = E_k \big[ (B \ot \IA  - r(X))^{-1}\big]$$
and hence, by definition of $G_{r(X_1,\dots,X_g)}$, that
\begin{equation}\label{eq:repG_part1}
 \cE_k^* \  E_{N+k}\left[    
 \left(\begin{pmatrix} B & 0\\ 0 & 0\end{pmatrix} \ot \IA 
  - \hpen(X)\right)^{-1}      \right] \cE_k
= G_{r(X)}(B).
\end{equation}

Now, we must pay some attention to the fact that the expression
$$E_{N+k}\left[\left(\begin{pmatrix} B & 0\\ 0 & 0\end{pmatrix}  \ot \IA 
- \hpen(X)\right)^{-1}\right]$$
appearing on the left hand side of this equation is not precisely an evaluation of the $M_{N+k}(\CC)$-valued Cauchy transform of $\hpen(X)$ but rather a boundary value of it, since $\begin{pmatrix} B & 0\\ 0 & 0\end{pmatrix}$ does not belong to the upper half-plane $\bH^+(M_{N+k}(\CC))$. The representation given in \eqref{eq:repG} therefore involves a limit procedure which allows to move our observation by $\begin{pmatrix} B & 0\\ 0 & i\varepsilon I_N\end{pmatrix}$ to the domain of the Cauchy transforms, where all our regular tools apply.

In order to check the validity of the representation given in \eqref{eq:repG}, we just have to observe that $G_{\hpen(X)}$ can be extended in the obvious way by
$$G_{\hpen(X)}(A) = E_{N+k}\big[(A  \ot \IA  -\hpen(X))^{-1}\big]$$
to the open set $\Omega \supset \bH^+(M_{N+k}(\CC))$ of all matrices $A\in M_{N+k}(\CC)$, for which $A \ot \IA    -\hpen(X)$ is invertible in $M_{N+k}(\CC) \otimes \cA$. Since this extension is regular and in particular continuous, and since $\begin{pmatrix} B & 0\\ 0 & 0\end{pmatrix} \in \Omega$, we may deduce that
\begin{equation}\label{eq:repG_part2}
\begin{aligned}
\lim_{\varepsilon\searrow 0} G_{\hpen(X)}  \left( \begin{pmatrix} B & 0\\ 0 & i \varepsilon I_N\end{pmatrix} \right)
 = E_{N+k}\left[\left(\begin{pmatrix} B & 0\\ 0 & 0\end{pmatrix}  \ot \IA - \hpen(X)\right)^{-1}\right].
\end{aligned}
\end{equation}
By the continuity of the map compressing 
$ M_{N+k}(\CC)$ to $M_k(\CC)$, 
a combination of \eqref{eq:repG_part1} and \eqref{eq:repG_part2} 
yields the stated formula \eqref{eq:repG}.
\end{proof}

The next theorem, which is the second part of our main result, tells us how we can find a selfadjoint generalized descriptor realization $\br$, by which $r$ is realized at some given point $X=X^\ast$ in its domain.

\begin{thm}\label{thm:mainrep2}
Let $r$ be a selfadjoint $k \times k$ matrix of rational expressions in the variables $x=(x_1,\dots,x_g)$ and let $X_1,\dots,X_g$ be selfadjoint elements in some $C^\ast$-probability space $(\cA,\phi)$, such that the condition $X=(X_1,\dots,X_g)\in\dom_\cA(r)$ is satisfied. Then the following statements hold true:

\begin{itemize}
 
 \item[(1)] If $\rho=(Q,v)$ is any selfadjoint formal linear representation of $r$ (whose existence is guaranteed by Theorem \ref{thm:FLR_matrices_of_rational_expressions}), then $r$ is realized at $X$ by $$\br(x) = \br_\rho(x) := -v^\ast Q(x)^{-1} v.$$

 \item[(2)] Suppose that $r$ is regular at zero, meaning that all its entries are regular at zero, and suppose in addition that the state $\phi$ on $\cA$ is both faithful and tracial. If
$$\br(x) = \Delta + \Xi^\ast (M_0 - L_M(x))^{-1} \Xi$$
is any selfadjoint descriptor realization of $r$ and if we assume
\begin{itemize}
 \item[(i)] either that the condition $X\in\dom_\cA(\br)$ is satisfied,
 \item[(ii)] or that $\br$ is chosen minimal (such $\br$ exist according to Lemma \ref{lem:descr}),
\end{itemize}
then $r$ is realized at $X$ by $\br$.

\end{itemize}

\end{thm}

\begin{proof}
We need to check that the $\br$'s given in the statement of the theorem satisfy the conditions of Definition \ref{def:realized_by}.

\begin{itemize}

 \item[(1)] Given any selfadjoint matrix-valued formal linear representation $\rho=(Q,v)$ of $r$, we know according to Definition \ref{def:matval_rep_sa}, since $X$ belongs to $\dom^\sa_\cA(r)$ by assumption, that $X$ must lie in $\dom_\cA(Q^{-1}) = \dom_\cA(\br_\rho)$, and moreover, that the equality $r(X) = - v^\ast Q(X)^{-1} v = \br_\rho(X)$ holds. This proves that $r$ is realized at $X$ by $\br_\rho$.

 \item[(2)] Note that the additional assumption on $\phi$ being a faithful trace guarantees that $\cA$ is stably finite according to Lemma \ref{lem:trace}.
      \begin{itemize}
       \item[(i)] If we suppose that $X \in \dom_\cA(\br)$, i.e. $X\in \dom_\cA(r) \cap \dom_\cA(\br)$, then the values $r(X)$ and $\br(X)$ must coincide according to Theorem \ref{thm:sameOnX_generalized} \eqref{it:sameOnX_descriptor}. This means that $r$ is realized at $X$ by $\br$, as claimed.
       \item[(ii)] Since $X\in \dom_\cA(r)$ by assumption, Theorem \ref{thm:rep_min} (ii) tells us that $X$ belongs to the domain $\dom_\cA(\br)$ of the given minimal descriptor realization $\br$ of $r$ and that $r$ and $\br$ yield the same when evaluated at the point $X$. Thus $r$ is realized at $X$ by $\br$.
      \end{itemize}

\end{itemize}
This concludes the proof.
\end{proof}

\subsubsection{Uniqueness of the representation}

In Theorem \ref{thm:mainrep2}, we constructed for a given realization
$$\br(x) = \Delta +  \Xi^\ast (M_0  - M_1  x_1  - \dots  -  M_g  x_g )^{-1} \Xi$$
of $r(x)$ a matrix $\hpen(x)$ by \eqref{eq:shifted_linearization}, i.e.
$$\hpen(x) = \begin{pmatrix} \Delta & \Xi^\ast\\ \Xi & -(M_0 - L_M(x))\end{pmatrix}.$$

The uniqueness result formulated in Lemma \ref{lem:descr} for minimal descriptor realizations of $r(x)$ guarantees that matrices $\hpen(x)$ and $\widetilde{\hpen}(x)$ obtained from different minimal descriptor realizations of $r(x)$ with the same feed through term $\Delta$ are related by $\hpen(x) = \hat{S}^\ast \widetilde{\hpen}(x) \hat{S}$, where the matrix $\hat{S}$ is of the form
$$\hat{S} = \begin{pmatrix} I_k & 0\\ 0 & S\end{pmatrix}$$
for some invertible $N\times N$ matrix $S$.

Indeed, if we consider two descriptor realizations
$$\br(x) = \Delta +  \Xi^\ast (M_0  - L_M(x))^{-1} \Xi \quad\text{and}\quad \widetilde{\br}(x) = \Delta +  \widetilde{\Xi}^\ast (\widetilde{M}_0 - L_{\widetilde{M}}(x))^{-1} \widetilde{\Xi}$$
of $r$, which both satisfy the minimality constraint, we know from Lemma \ref{lem:descr} that there is an invertible $N\times N$ matrix $S$ satisfying $S^\ast \widetilde{M}_0 S = M_0$, such that
\begin{equation}
S M_0 M_j = \widetilde{M}_0 \widetilde{M}_j S,  \qquad      
S M_0 \Xi = \widetilde{M}_0 \widetilde{\Xi},    \qquad
\Xi^\ast  = \widetilde{\Xi}^\ast S.
\end{equation}
Hence, the matrices
$$\hpen(x) = \begin{pmatrix} \Delta & \Xi^\ast\\ \Xi & -(M_0 - L_M(x))\end{pmatrix} \qquad\text{and}\qquad \widetilde{\hpen}(x) = \begin{pmatrix} \Delta & \widetilde{\Xi}^\ast\\ \widetilde{\Xi} & -(\widetilde{M}_0 - L_{\widetilde{M}}(x))\end{pmatrix},$$
which we constructed according to Theorem \ref{thm:mainrep1}, are related by
$$\hpen(x) = \hat{S}^\ast \widetilde{\hpen}(x) \hat{S}, \qquad\text{where}\qquad \hat{S} := \begin{pmatrix} I_k & 0\\ 0 & S\end{pmatrix}.$$
Indeed, we can check that
$$\hat{S}^\ast \widetilde{\hpen}(x) \hat{S} = \begin{pmatrix} 0 & \widetilde{\Xi}^\ast S\\ S^\ast \widetilde{\Xi} & -S^\ast(\widetilde{M}_0-L_{\widetilde{M}}(x))S\end{pmatrix} = \begin{pmatrix} \Delta & \Xi^\ast\\ \Xi & -(M_0 - L_M(x))\end{pmatrix} = \hpen(x),$$
since it holds true that
\begin{itemize}
 \item $\widetilde{\Xi}^\ast S = \Xi^\ast$ and therefore also $S^\ast \widetilde{\Xi} = \Xi$;
 \item $S^\ast \widetilde{M}_0 S = M_0$ and $S^\ast \widetilde{M}_j S = M_j$.
\end{itemize}
The formula $S^\ast \widetilde{M}_j S = M_j$ stated above can be shown as follows: The assumption $S M_0 M_j = \widetilde{M}_0 \widetilde{M}_j S$ can be rewritten as $M_j = (M_0 S^{-1} \widetilde{M}_0) \widetilde{M}_j S$ and since by construction $S^\ast \widetilde{M}_0 S = M_0$ holds, which gives $S^\ast = M_0 S^{-1} \widetilde{M}_0$, we finally get the stated relation $S^\ast \widetilde{M}_j S = M_j$.

\subsection{How to calculate distributions and Brown measures of rational expressions}
\label{subsec:main_problems}

We have collected now all tools to settle the question formulated in Section \ref{subsec:polynomials}. In fact, since the linearization techniques that we have presented in Section \ref{sec:Realization} and Section \ref{sec:Algorithm} are by no means limited to noncommutative polynomials but apply equally well to noncommutative rational expressions, especially when they are regular at $0$, we can provide here a solution to the following, far more general problem; this continues \cite{BMS2013}, where only noncommutative polynomials were considered.

\begin{prob}\label{prob:dist}
Let $r$ be a selfadjoint rational expression in the variables $x=(x_1,\dots,x_g)$ and let $X_1,\dots,X_g$ be freely independent selfadjoint elements in some $C^\ast$-probability space $(\cA,\phi)$. We suppose that the evaluation $r(X_1,\dots,X_g)$ is defined, i.e, that $(X_1,\dots,X_g) \in \dom^\sa_\cA(r)$ holds. If the analytic distribution of each of the $X_j$'s is known, how can we compute the analytic distribution of $r(X_1,\dots,X_g)$?
\end{prob}

In \cite{BSS2015}, techniques developed in \cite{BMS2013} were extended to the case of non-selfadjoint polynomials. In such situations, the Brown measure that we have discussed in Section \ref{subsec:Brown} provides an appropriate substitute for the analytic distribution that is used in the selfadjoint setting. By combining their methods with our tools from Section \ref{sec:Realization} and Section \ref{sec:Algorithm}, we can extend the approach of \cite{BSS2015} to noncommutative rational expressions. More precisely, we address the following, non-selfadjoint counterpart of Problem \ref{prob:dist}. 

\begin{prob}\label{prob:Brown}
Let $r$ be an arbitrary rational expression $r$ in the variables $x=(x_1,\dots,x_g)$ and consider freely independent selfadjoint elements $X_1,\dots,X_g$ in some tracial $W^\ast$-probability space $(\cA,\phi)$. Suppose that the evaluation $r(X_1,\dots,X_g)$is defined, i.e., that $(X_1,\dots,X_g) \in \dom_\cA(r)$ holds. If the distribution of each of the $X_j$'s is known, how can we compute the Brown-measure of $r(X_1,\dots,X_g)$?
\end{prob}

At the heart of our approach are Theorem \ref{thm:mainrep1} and Theorem \ref{thm:mainrep2}, which will be combined with Theorem \ref{thm:subordination}. In fact, they allow us to provide an algorithmic solution to both of the aforementioned Problems \ref{prob:dist} and \ref{prob:Brown}, which is easily accessible for numerical computations; see Section \ref{subsec:examples} for examples.

\subsubsection{An algorithmic solution of Problem \ref{prob:dist}}

Let us first discuss our solution to Problem \ref{prob:dist}. This is the content of the following algorithm.

\begin{algorithm}\label{alg:dist}
Let $r$ be a selfadjoint noncommutative rational expression in the variables $x=(x_1,\dots,x_g)$ and let $X_1,\dots,X_g$ be freely independent selfadjoint elements in some $C^\ast$-probability space $(\cA,\phi)$, such that the condition $X=(X_1,\dots,X_g) \in \dom_\cA^\sa(r)$ is satisfied. If the scalar-valued Cauchy transforms $G_{X_1},\dots,G_{X_g}$ are given, then the distribution $\mu_{r(X)}$ of $r(X)$ can be obtained as follows:
\begin{enumerate}
 \item By means of Theorem \ref{thm:mainrep2} find a descriptor realization $$\br(x) = \Delta + \Xi^\ast \Lambda_M(x)^{-1} \Xi$$ of size $1\times 1$ with $\Lambda_M$ being of size $N\times N$ for some $N$, such that $r$ is realized at $X$ by $\br$ in the sense of Definition \ref{def:realized_by}.
 \item Consider the selfadjoint affine linear pencil $$\widehat{\Lambda}(x) = \widehat{\Lambda}_0 + \widehat{\Lambda}_1 x_1 + \ldots + \widehat{\Lambda}_g x_g$$ associated to $\br$, which was defined in \eqref{eq:shifted_linearization}, i.e.
 $$\widehat{\Lambda}(x) = \begin{pmatrix} \Delta & \Xi^\ast\\ \Xi & -\Lambda(x)\end{pmatrix}.$$
 \item Apply Theorem \ref{thm:mainrep1} in the case $k=1$ and deduce from \eqref{eq:repG} that the scalar-valued Cauchy transform of $r(X)$ is determined by the $M_{N+1}(\CC)$-valued Cauchy transform of $\widehat{\Lambda}(X)$; in fact, we have
$$G_{r(X)} (z) = \lim_{\varepsilon\searrow 0} \begin{pmatrix} 1 & 0\end{pmatrix} G_{\widehat{\Lambda}(X)} \left( \begin{pmatrix} z & 0\\ 0 & i \varepsilon 1_n\end{pmatrix} \right) \begin{pmatrix} 1 \\ 0\end{pmatrix}$$
for each $z\in\CC^+$.
 \item According to Remark \ref{freeness_matrices}, the operators $\hpen_1 \otimes X_1,\dots, \hpen_g \otimes X_g$ are freely independent with amalgamation over $M_{N+1}(\CC)$. Hence, the $M_{N+1}(\CC)$-valued Cauchy transform of $\hpen(X) - \hpen_0 \otimes 1_\cA$ can be computed by means of Theorem \ref{thm:subordination}; note that the matrix-valued Cauchy-transforms of $\hpen_1 \otimes X_1,\dots, \hpen_g \otimes X_g$ can be computed as explained in Remark \ref{rem:Cauchy_matval}. The desired $M_{N+1}(\CC)$-valued Cauchy transform of $\hpen(X)$ is then obtained by the following shift $$G_{\hpen(X)}(B) = G_{\hpen(X)-\hpen_0 \otimes 1_\cA}(B-\hpen_0)$$ for all $B\in \bH^+(M_{N+1}(\CC))$.
 \item Having computed $G_{r(X)}$, the desired distribution of $r(X)$ can then be obtained by Stieltjes inversion formula; see Remark \ref{rem:Stieltjes}.
\end{enumerate}
\end{algorithm}

\subsubsection{An algorithmic solution of Problem \ref{prob:Brown}}

Finally, let us discuss the algorithmic solution of Problem \ref{prob:Brown}. 

\begin{algorithm}\label{alg:Brown}
Let $r$ be any noncommutative rational expression in the variables $x=(x_1,\dots,x_g)$ and let $X_1,\dots,X_g$ be freely independent selfadjoint elements in some $C^\ast$-probability space $(\cA,\phi)$, for which $X=(X_1,\dots,X_g) \in \dom_\cA(r)$ holds. If the scalar-valued Cauchy transforms $G_{X_1},\dots,G_{X_g}$ are given, then the Brown measure $\nu_{r(X_1,\dots,X_g)}$ of $r(X_1,\dots,X_g)$ can be obtained then as follows:
\begin{enumerate}
 \item Consider the following matrix of noncommutative rational expressions
\begin{equation}\label{Brown-matrix}
\ur := \begin{pmatrix} 0 & r\\ r^\ast & 0 \end{pmatrix},
\end{equation}
where $r^\ast$ denotes the adjoint of $r$ in the sense of Remark \ref{rem:rational_expression_adjoint}. Then the matrix $\ur$ is selfadjoint in the sense of Section \ref{subsubsec:symmetric_mat-val_rational_expressions}.
 \item By means of Theorem \ref{thm:mainrep2}, find a descriptor realization $$\br(x) = \Delta + \Xi^\ast \Lambda_M(x)^{-1} \Xi$$ of size $2\times 2$ with $\Lambda_M$ being of size $N\times N$ for some $N$, such that $r$ is realized at $X$ by $\br$ in the sense of Definition \ref{def:realized_by}.
 \item Consider the selfadjoint affine linear pencil $$\widehat{\Lambda}(x) = \widehat{\Lambda}_0 + \widehat{\Lambda}_1 x_1 + \ldots + \widehat{\Lambda}_g x_g$$ associated to $\br$, which was defined in \eqref{eq:shifted_linearization}, i.e.
 $$\widehat{\Lambda}(x) = \begin{pmatrix} \Delta & \Xi^\ast\\ \Xi & -\Lambda(x)\end{pmatrix}.$$
 \item Apply Theorem \ref{thm:mainrep1} in the case $k=2$ and deduce from \eqref{eq:repG} that the scalar-valued Cauchy transform of $r(X)$ is determined by the $M_{N+2}(\CC)$-valued Cauchy transform of $\hpen(X)$; in fact, we have
$$G_{\ur(X)} (B) = \lim_{\varepsilon\searrow 0} \begin{pmatrix} I_2 & 0\end{pmatrix} G_{\hpen(X)} \left( \begin{pmatrix} B & 0\\ 0 & i \varepsilon 1_n\end{pmatrix} \right) \begin{pmatrix} I_2\\ 0\end{pmatrix}$$
for each $B\in\bH^+(M_2(\CC))$.
 \item According to Remark \ref{freeness_matrices}, the operators $\hpen_1 \otimes X_1,\dots, \hpen_g \otimes X_g$ are freely independent with amalgamation over $M_{N+2}(\CC)$. Hence, the $M_{N+2}(\CC)$-valued Cauchy transform of $\hpen(X)-\hpen_0 \otimes 1_\cA$ can be computed by means of Theorem \ref{thm:subordination}; note that the matrix-valued Cauchy-transforms of $\hpen_1 \otimes X_1,\dots, \hpen_g \otimes X_g$ can be computed as explained in Remark \ref{rem:Cauchy_matval}. The desired $M_{N+2}(\CC)$-valued Cauchy transform of $\hpen(X)$ is then obtained by the following shift $$G_{\hpen(X)}(B) = G_{\hpen(X)-\hpen_0 \otimes 1_\cA}(B-\hpen_0)$$ for all $B\in \bH^+(M_{N+2}(\CC))$.
 \item The regularized Cauchy transform $G_{r(X),\varepsilon}$ is determined by \eqref{eq:hermCauchy}, i.e., we have
$$G_{r(X),\varepsilon}(z) = \Big[G_{\ur(X)}\Big(\begin{pmatrix} i\varepsilon & z \\ \overline{z} & i\varepsilon\end{pmatrix}\Big)\Big]_{2,1}$$
for all $z\in\CC$.
 \item The regularized Brown measure $\mu_{r(X),\varepsilon}$ can be obtained, according to \eqref{eq:BrownCauchy}, from the regularized Cauchy transform by
$$d\mu_{r(X),\varepsilon}(z) = \frac{1}{\pi} \frac{\partial}{\partial\overline{z}} G_{r(X),\varepsilon}(z)\, d\lambda^2(z).$$
 \item As $\varepsilon\searrow0$, the regularized Brown measure $\mu_{r(X),\varepsilon}$ converges weakly to the Brown measure $\nu_{r(X)}$.
\end{enumerate}
\end{algorithm}

We add here the useful observation that a selfadjoint descriptor realization of the matrix $\hat{r}$, which we introduced above in \eqref{Brown-matrix}, can be constructed from any realization of the involved rational expression $r$. The precise statement, which in addition covers the case of rational expressions, which are not necessarily regular at $0$, reads as follows.

\begin{lem}\label{lem:block-realization}
Let $r$ be a scalar-valued rational expression in the variables $x=(x_1,\dots,x_g)$ and consider as in \eqref{Brown-matrix} the matrix-valued rational expression $\hat{r}$ given by
$$\hat{r} := \begin{pmatrix} 0 & r\\ r^\ast & 0 \end{pmatrix}.$$
\begin{enumerate}
 \item If $\rho=(u,Q,v)$ is any formal linear representation of $r$, then $$\hat{\rho} = (\hat{Q},\hat{v}) := \Big(\begin{pmatrix} 0 & Q\\ Q^\ast & 0\end{pmatrix}, \begin{pmatrix} 0 & v\\ u^\ast & 0\end{pmatrix}\Big)$$ gives a selfadjoint formal linear representation of $\hat{r}$.
 \item If $r$ is regular at $0$ and if $$\br(x) = \Delta + \Xi_1 (M_0 - L_M(x))^{-1} \Xi_2$$ with $L_M(x) = M_1 x_1 + \dots + M_g x_g$ is any descriptor realization of $r$, then we may obtain a selfadjoint realization of $\hat{r}$ by $$\hat{\br}(x) = \begin{pmatrix} 0 & \Delta\\ \Delta^\ast & 0\end{pmatrix} + \hXi^\ast (\hM_0 - L_{\hM}(x))^{-1} \hXi,$$ where we put $\hXi := \begin{pmatrix} 0 & \Xi_2\\ \Xi_1^\ast & 0\end{pmatrix}$ and $\hM_j := \begin{pmatrix} 0 & M_j\\ M_j^\ast & 0\end{pmatrix}$ for $j=0,1,\dots,g$.
\end{enumerate}
\end{lem}

\begin{proof}
(1) Let $\cA$ be any $\ast$-algebra. We clearly have that $\dom_\cA(Q^{-1}) = \dom_\cA(\hat{Q}^{-1})$, and since $\rho$ is a formal linear representation of $r$, we have by definition $\dom_\cA(r) \subseteq \dom_\cA(Q^{-1})$. In combination, this gives $\dom_\cA(r) \subseteq \dom_\cA(\hat{Q}^{-1})$ and in particular $\dom_\cA^\sa(r) \subseteq \dom_\cA(\hat{Q}^{-1})$. Furthermore, $\rho$ enjoys the property that $r(X) = - u Q(X)^{-1} v$ and hence $r^\ast(X^\ast) = - v^\ast Q^\ast(X^\ast)^{-1} u^\ast$ holds for any $X\in\dom_\cA(r)$. Thus, if we take $X\in\dom^\sa_\cA(r)$, we may deduce that $r(X) = - u Q(X)^{-1} v$ and $r^\ast(X) = - v^\ast Q^\ast(X)^{-1} u^\ast$ holds, so that
\begin{align*}
-\hat{v}^\ast \hat{Q}(X)^{-1} \hat{v}
&= -\begin{pmatrix} 0 & u\\ v^\ast & 0\end{pmatrix} \begin{pmatrix} 0 & Q(X)\\ Q^\ast(X) & 0\end{pmatrix}^{-1} \begin{pmatrix} 0 & v\\ u^\ast & 0\end{pmatrix}\\
&= -\begin{pmatrix} 0 & u\\ v^\ast & 0\end{pmatrix} \begin{pmatrix} 0 & Q^\ast(X)^{-1}\\ Q(X)^{-1} & 0\end{pmatrix} \begin{pmatrix} 0 & v\\ u^\ast & 0\end{pmatrix}\\
&= -\begin{pmatrix} 0 & u\\ v^\ast & 0\end{pmatrix} \begin{pmatrix} Q^\ast(X)^{-1} u^\ast & 0\\ 0 & Q(X)^{-1} v \end{pmatrix}\\
&= -\begin{pmatrix} 0 & u Q(X)^{-1} v \\ v^\ast Q^\ast(X)^{-1} u^\ast & 0\end{pmatrix}\\
&= \begin{pmatrix} 0 & r(X)\\ r^\ast(X) & 0 \end{pmatrix}\\
&= \hat{r}(X).
\end{align*}
This shows that $\hat{\rho}$ is indeed a selfadjoint formal linear representation.

(2) First of all, we note that $\hM_0^\ast = \hM_0$ and $\hM_0^2 = I_{2k}$, since by assumption $M_0^\ast = M_0$ and $M_0^2 = I_k$ holds. Thus, $\hat{\br}$ is indeed a selfadjoint descriptor realization of some matrix-valued rational expression. It remains to prove that it forms in fact a realization of $\hat{r}$. For this purpose, it is according to \cite[Proposition A.7]{HMV06} sufficient to check that $\hat{r}$ and $\hat{\br}$ are $\bbM(\CC)_\sa$-evaluation equivalent.

For doing so, we take any matrix $X\in\bbM(\CC)_\sa$, which belongs to the domain of $\hat{r}$ and to the domain of $\hat{\br}$. Since
$$\hM_0 -  L_{\hM}(X) =  \begin{pmatrix} 0 & M_0 - L_M(X)\\ M_0^\ast - L_{M^\ast}(X) & 0\end{pmatrix},$$
we have that $X$ also belongs to the domain of $\br$ and furthermore
\begin{align*}
\hXi^\ast (\hM_0 - & L_{\hM}(X))^{-1} \hXi\\
&= \begin{pmatrix} 0 & \Xi_2\\ \Xi_1^\ast & 0\end{pmatrix}^\ast \begin{pmatrix} 0 & M_0 - L_M(X)\\ M_0^\ast - L_{M^\ast}(X) & 0\end{pmatrix}^{-1} \begin{pmatrix} 0 & \Xi_2\\ \Xi_1^\ast & 0\end{pmatrix}\\
&= \begin{pmatrix} 0 & \Xi_1\\ \Xi_2^\ast & 0\end{pmatrix} \begin{pmatrix} 0 & (M_0^\ast - L_{M^\ast}(X))^{-1}\\ (M_0 - L_M(X))^{-1} & 0\end{pmatrix} \begin{pmatrix} 0 & \Xi_2\\ \Xi_1^\ast & 0\end{pmatrix}\\
&= \begin{pmatrix} 0 & \Xi_1\\ \Xi_2^\ast & 0\end{pmatrix} \begin{pmatrix} (M_0^\ast - L_{M^\ast}(X))^{-1}\Xi_1^\ast & 0\\ 0 & (M_0 - L_M(X))^{-1}\Xi_2\end{pmatrix}\\
&= \begin{pmatrix} 0 & \Xi_1(M_0 - L_M(X))^{-1}\Xi_2\\ \Xi_2^\ast(M_0^\ast - L_{M^\ast}(X))^{-1}\Xi_1^\ast & 0\end{pmatrix},
\end{align*}
so that
$$\hat{\br}(X) = \begin{pmatrix} 0 & \Delta\\ \Delta^\ast & 0\end{pmatrix} + \hXi^\ast (\hM_0 - L_{\hM}(X))^{-1} \hXi = \begin{pmatrix} 0 & \br(X)\\ \br(X)^\ast & 0 \end{pmatrix}.$$
Moreover, since $\br$ is a realization of $r$ and therefore $\br(X) = r(X)$ holds, we may continue
$$\hat{\br}(X) = \begin{pmatrix} 0 & \br(X)\\ \br(X)^\ast & 0 \end{pmatrix} = \begin{pmatrix} 0 & r(X)\\ r(X)^\ast & 0 \end{pmatrix} = \hat{r}(X).$$
This concludes the proof.
\end{proof}

A more complicated construction underlies the minimal symmetric realization asserted in Lemma \ref{lem:symR} \eqref{lem:symR2}.

\subsection{Examples}
\label{subsec:examples}

We conclude with several concrete examples by which we discuss the theory presented above. While the Examples \ref{ex:anticommutator} and \ref{ex:rational-distr} concern Problem \ref{prob:dist} for the anticommutator $p(x_1,x_2) = x_1x_2+x_2x_1$ and for some rational expression $r$ represented by the descriptor realization
$$\br(x_1,x_2) = \begin{pmatrix} \frac{1}{2} & 0 \end{pmatrix} \begin{pmatrix} 1-\frac{1}{4} x_1 & -\frac{1}{4}x_2\\ -\frac{1}{4}x_2 & 1-\frac{1}{4} x_1 \end{pmatrix}^{-1} \begin{pmatrix} \frac{1}{2}\\ 0\end{pmatrix},$$
Example \ref{ex:rational-Brown} concerns Problem \ref{prob:Brown} for some rational expression $r$ represented by the the descriptor realization
$$\br(x_1,x_2) = \begin{pmatrix} 0 & \frac{1}{2} \end{pmatrix} \begin{pmatrix} 1-\frac{1}{4} x_1 & -ix_2\\ -\frac{1}{4}x_2 & 1-\frac{1}{4} x_1 \end{pmatrix}^{-1} \begin{pmatrix} \frac{1}{2}\\ 0\end{pmatrix}.$$

Since we will perform the numerical computations in the case of freely independent elements $X_1,\dots,X_g$ whose distribution is given either by the semicircular distribution or by the free Poisson distribution (where the latter is also known as the Marchenko-Pastur distribution), their joint distribution arises also as the limit of the joint distribution of (classically) independent Gaussian and Wishart random matrices $(X_1^{(N)},\dots,X_g^{(N)})$ of dimension $N\times N$, respectively, in the limit $N\to \infty$.

Thus, we will compare below the computed distribution of $r(X_1,\dots,X_g)$ with the normalized histogram of all eigenvalues of the random matrix obtained by $r(X_1^{(N)},\dots,X_g^{(N)})$. This will show in all cases a nice conformity.

Note that, whereas in the special case of a noncommutative polynomial $r$ the convergence of the eigenvalue distribution of $r(X_1^{(N)},\dots,X_g^{(N)})$ to $r(X_1,\dots,X_g)$ is obvious, the situation for general $r$ is more intricate. The main difficulty here his that $(X_1^{(N)},\dots,X_g^{(N)})$ must belong almost surely to the domain of $r$ if their dimension $N$ is sufficiently large. Nevertheless, it is conceivable that this is can be shown in many cases. For recent progress in this direction, see \cite{Y17}.

In the Brown measure case, however, despite the amazing similarity between the output of our algorithm and of the random matrix simulation, there is up to now no general statement which would give a rigorous justification of this phenomenon.

\begin{figure}
\centering
\includegraphics[height=8cm]{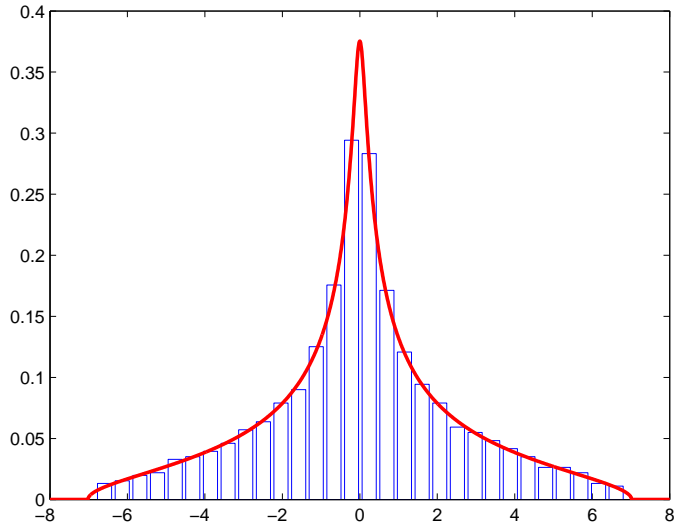}
\caption{Histogram of eigenvalues of $p(X_1^{(N)},X_2^{(N)})$, where $p(x_1,x_2)$ was defined in Example \ref{ex:anticommutator}, for one realization of independent random matrices $X_1^{(N)},X_2^{(N)}$, where $X_1^{(N)}$ is a Wishart random matrix and $X_2^{(N)}$ a Gaussian random matrix, both of size $N=1000$, compared with the distribution of $p(X_1,X_2)$ for freely independent elements $X_1,X_2$, where $X_1$ is a free Poisson element and $X_2$ a semicircular element.}
\label{fig:anticommutator-ex2}
\end{figure}

\begin{example}[Anti-commutator, see Figure \ref{fig:anticommutator-ex2}]
\label{ex:anticommutator}
We consider the so-called \textbf{anti-commutator}\index{commutator!anti-}
$$p(x_1,x_2) := x_1 x_2 + x_2 x_1.$$
In order to produce a (selfadjoint) representation and finally a (selfadjoint) realization of $p$, we could of course just apply the algorithm that we have presented in Section \ref{sec:Algorithm}. However, since we can write
$$p(x_1,x_2) = \begin{pmatrix} x_1 & x_2\end{pmatrix} \begin{pmatrix} 0 & 1\\ 1 & 0 \end{pmatrix}^{-1} \begin{pmatrix} x_1 \\ x_2\end{pmatrix},$$
this gives directly the selfadjoint representation
$$p(x_1,x_2) = - \begin{pmatrix} 0 & 0 & 0 & 1 \end{pmatrix} \begin{pmatrix} 0 & x_1 & x_2 & -1\\ x_1 & 0 & -1 & 0\\ x_2 & -1 & 0 & 0\\ -1 & 0 & 0 & 0\end{pmatrix}^{-1} \begin{pmatrix} 0 \\ 0 \\ 0 \\ 1 \end{pmatrix}.$$

According to Theorem \ref{thm:mainrep1}, we consider now the matrix
$$\hpen(x_1,x_2) =
\begin{pmatrix}
0 & 0 & 0 & 0 & 1\\
0 & 0 & x_1 & x_2 & -1\\
0 & x_1 & 0 & -1 & 0\\
0 & x_2 & -1 & 0 & 0\\
1 & -1 & 0 & 0 & 0
\end{pmatrix}$$
which decomposes as $\hpen(x_1,x_2) = \hpen_0 + \hpen_1 x_1 + \hpen_2 x_2$, where
$$\hpen_0 =
\begin{pmatrix}
0 & 0 & 0 & 0 & 1\\
0 & 0 & 0 & 0 & -1\\
0 & 0 & 0 & -1 & 0\\
0 & 0 & -1 & 0 & 0\\
1 & -1 & 0 & 0 & 0
\end{pmatrix}, \quad \hpen_1 =
\begin{pmatrix}
0 & 0 & 0 & 0 & 0\\
0 & 0 & 1 & 0 & 0\\
0 & 1 & 0 & 0 & 0\\
0 & 0 & 0 & 0 & 0\\
0 & 0 & 0 & 0 & 0
\end{pmatrix}, \quad\text{and}\quad
  \hpen_2 =
\begin{pmatrix}
0 & 0 & 0 & 0 & 0\\
0 & 0 & 0 & 1 & 0\\
0 & 0 & 0 & 0 & 0\\
0 & 1 & 0 & 0 & 0\\
0 & 0 & 0 & 0 & 0
\end{pmatrix}.$$
\end{example}

The shortcut that was used in both of the previous examples in order to produce realizations without using the algorithm of Section \ref{sec:Algorithm} relies on a more general observation, which we include here for the sake of completeness. This is the content of the following remark.

\begin{rem}
Assume that $Q_1,\dots,Q_k$ are rectangular matrices of rational expressions in the variables $x_1,\dots,x_g$, such that $Q_j$ is a $(n_j \times n_{j+1})$-matrix for $j=1,\dots,k$, where $n_1,\dots,n_{k+1}$ are some positive integers. Furthermore, let $\Xi_j\in M_{n_j}(\CC)$, for $j=1,\dots,k+1$, be invertible matrices. Consider now the $(n_1 \times n_{k+1})$-matrix
$$r := \Xi_1^{-1} Q_1 \Xi_2^{-1} Q_2 \cdots \Xi_k^{-1} Q_k \Xi_{k+1}^{-1}$$
of rational expression and introduce a matrix $Q$ of size $n:=n_1 + \dots + n_k$ by
$$Q:=
\begin{pmatrix}
            &         &         & Q_1     & -\Xi_1\\
            &         & \iddots & - \Xi_2 & \\
            & Q_{k-1} & \iddots &         & \\
Q_k         & - \Xi_k &         &         & \\
- \Xi_{k+1} &         &         &         &
\end{pmatrix}.$$
It is then not hard to check inductively that we have for any unital complex algebra $\cA$ and for any $(X_1,\dots,X_g)\in \dom_\cA(r) = \bigcap^k_{j=1} \dom_\cA(Q_j)$ that $(X_1,\dots,X_g) \in \dom_\cA(Q^{-1})$ and
$$r(X_1,\dots,X_g) = - \begin{pmatrix} 0 & \hdots & 0 & 0 & I_{n_{k+1}}\end{pmatrix} Q(X_1,\dots,X_g)^{-1} \begin{pmatrix} 0\\ \vdots\\ 0 \\ 0 \\ I_{n_1}\end{pmatrix}.$$
In fact, the induction step is based on the observation that if we put
$$\widetilde{r} := \Xi_1^{-1} Q_1 \Xi_2^{-1} Q_2 \cdots Q_{k-1} \Xi_k^{-1}$$
and correspondingly
$$\widetilde{Q} :=
\begin{pmatrix}
        &         & Q_1     & -\Xi_1\\
        & \iddots & - \Xi_2 & \\
Q_{k-1} & \iddots &         & \\
- \Xi_k &         &         & 
\end{pmatrix},$$
then we obtain a block decomposition
$$Q = \left(\begin{array}{c|c} \begin{matrix} 0\\ \vdots\\ 0\\ Q_k\end{matrix} & \widetilde{Q}\\ \hline -\Xi_k & \begin{matrix} 0 & \hdots & 0 & 0 \end{matrix}\end{array}\right).$$
\end{rem}

\begin{figure}
\centering
\includegraphics[height=8cm]{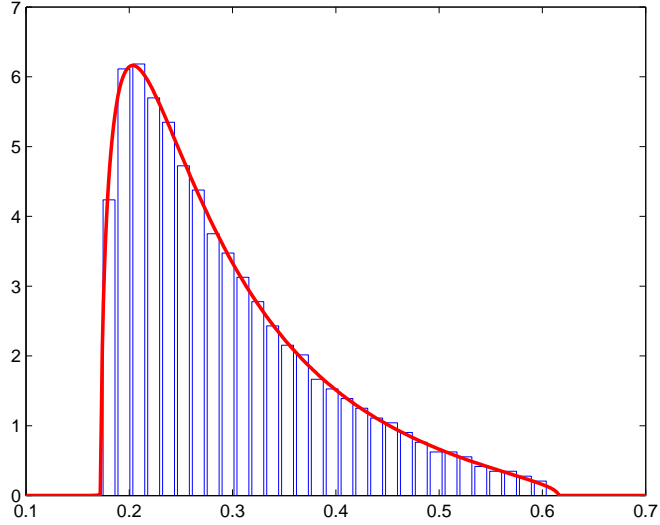}
\caption{Histogram of eigenvalues of $r(X_1^{(N)},X_2^{(N)})$ for one realization of independent Gaussian random matrices $X_1^{(N)},X_2^{(N)}$ of size $N=1000$, compared with the distribution of $r(X_1,X_2)$ for freely independent semicircular elements $X_1,X_2$. See Example \ref{ex:rational-distr}.}
\label{fig:rational-distr}
\end{figure}

\begin{example}
\label{ex:rational-distr}
Consider the following slight modification of the rational expression $r(x_1,x_2)$ that already appeared in Example \ref{ex:descriptor}, namely
$$r(x_1,x_2) = (4-x_1)^{-1} + (4-x_1)^{-1}x_2 \left((4-x_1)-x_2(4-x_1)^{-1}x_2\right)^{-1}x_2(4-x_1)^{-1},$$
which admits the selfadjoint realization
$$\br(x_1,x_2) := \begin{pmatrix} \frac{1}{2} & 0 \end{pmatrix} \begin{pmatrix} 1-\frac{1}{4} x_1 & -\frac{1}{4}x_2\\ -\frac{1}{4}x_2 & 1-\frac{1}{4} x_1 \end{pmatrix}^{-1} \begin{pmatrix} \frac{1}{2}\\ 0\end{pmatrix}.$$
According to Theorem \ref{thm:mainrep1}, we introduce
$$\hpen(x_1,x_2) = \begin{pmatrix} 0 & \frac{1}{2} & 0\\  \frac{1}{2} & -1+\frac{1}{4} x_1 &  \frac{1}{4}x_2\\ 0 &  \frac{1}{4}x_2 & -1+\frac{1}{4} x_1\end{pmatrix},$$
which decomposes as $\hpen(x_1,x_2) = \hpen_0 + \hpen_1 x_1 + \hpen_2 x_2$, where
$$\hpen_0 = \begin{pmatrix} 0 &  \frac{1}{2} & 0\\  \frac{1}{2} & -1 & 0\\ 0 & 0 & -1\end{pmatrix}, \quad \hpen_1 = \begin{pmatrix} 0 & 0 & 0\\ 0 & \frac{1}{4} & 0\\ 0 & 0 & \frac{1}{4}\end{pmatrix},\quad\text{and}\quad \hpen_2 = \begin{pmatrix} 0 & 0 & 0\\ 0 & 0 & \frac{1}{4}\\ 0 & \frac{1}{4} & 0\end{pmatrix}.$$
In Figure \ref{fig:rational-distr}, we compare the histogram of eigenvalues of $r(X_1^{(N)},X_2^{(N)})$ for one realization of independent Gaussian random matrices $X_1^{(N)},X_2^{(N)}$ of size $N = 1000$ with the distribution of $r(X_1,X_2)$ for freely independent semicircular elements $X_1,X_2$, calculated according to our algorithm.
\end{example}

\begin{figure}
\centering
\begin{subfigure}{0.47\textwidth}
\begin{center}
\includegraphics[height=6.2cm]{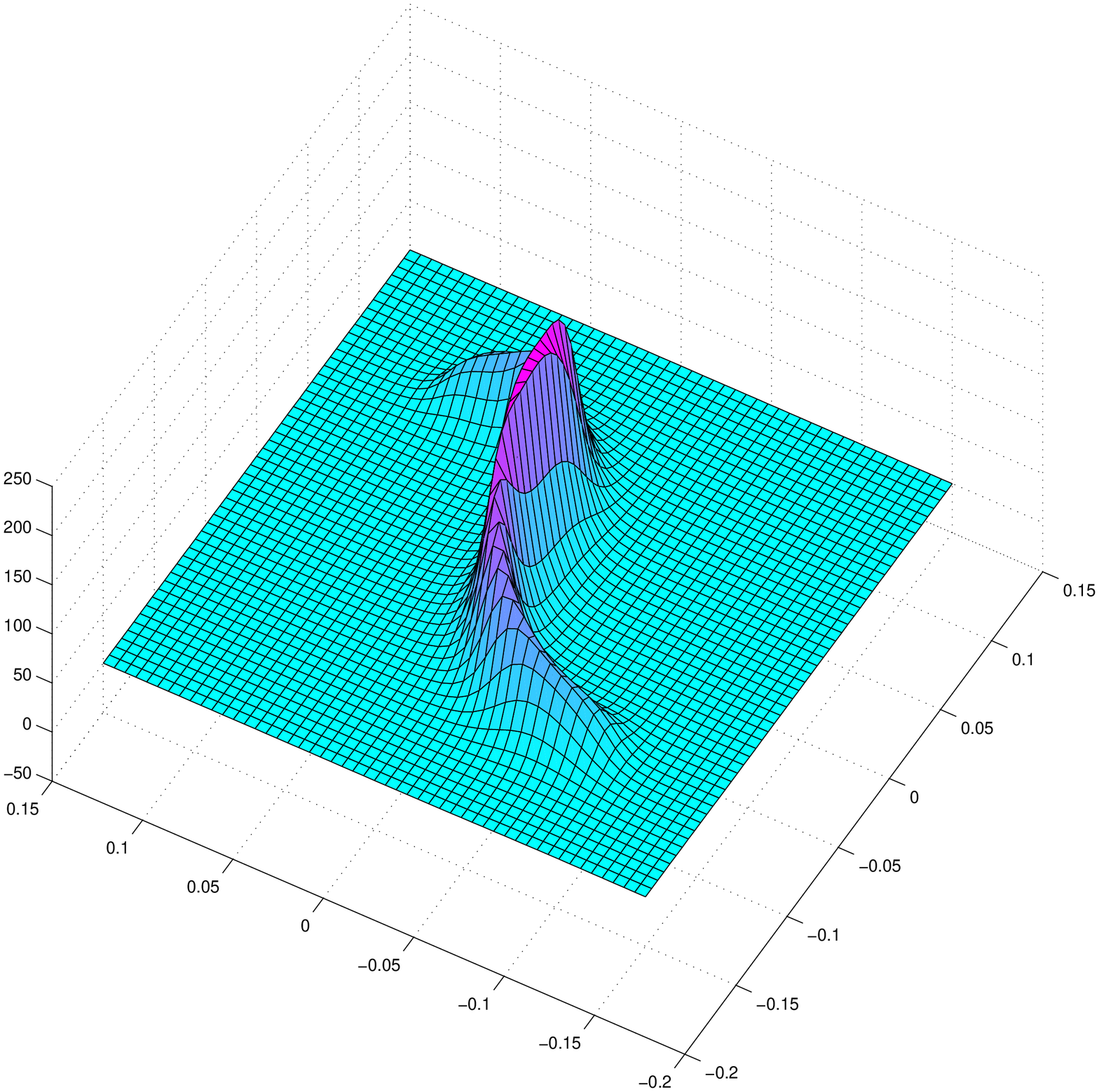}
\end{center}
\subcaption{Brown measure of $r(X_1,X_2)$ computed with Algorithm \ref{alg:Brown} for freely independent semicircular elements $X_1,X_2$.}
\label{fig:rational-Brown1}
\end{subfigure}
\qquad
\begin{subfigure}{0.47\textwidth}
\begin{center}
\includegraphics[height=6.2cm]{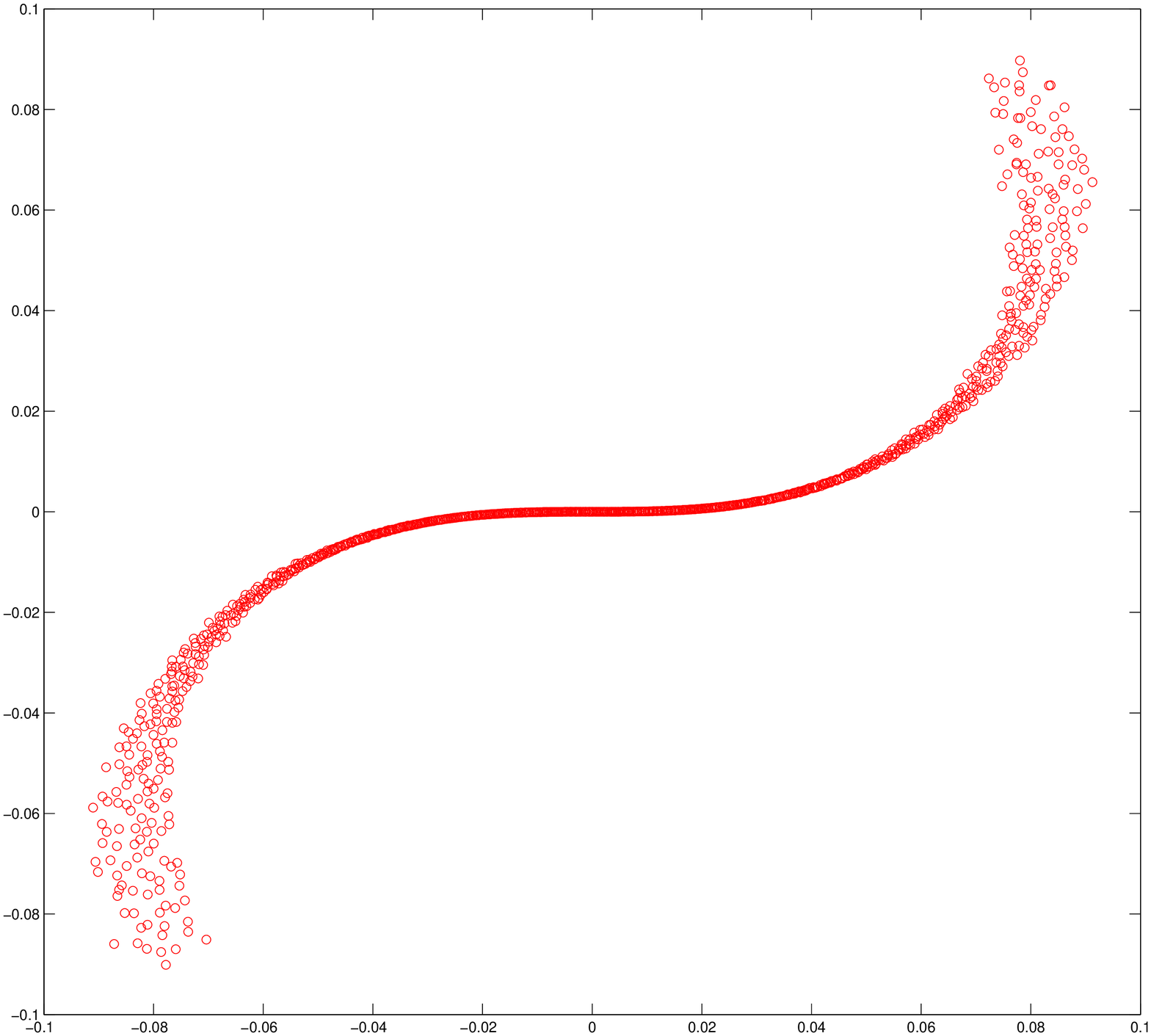}
\end{center}
\caption{Eigenvalues of $r(X_1^{(N)},X_2^{(N)})$ for independent Gaussian random matrices $X_1^{(N)},X_2^{(N)}$ of size $N=1000$.}
\label{fig:rational-Brown2}
\end{subfigure}
\caption{The result of Algorithm \ref{alg:Brown} compared with a random matrix simulation for the rational expression $r(x_1,x_2)$ defined in Example \ref{ex:rational-Brown}.}
\end{figure}

\begin{example}
\label{ex:rational-Brown}
We consider now the rational expression
$$r(x_1,x_2) := (4 - x_1)^{-1} x_2 \big(4 - x_1 - 4i x_2 (4 - x_1)^{-1} x_2\big)^{-1},$$
which is represented by the descriptor realization
$$\br(x_1,x_2) := \begin{pmatrix} 0 & \frac{1}{2} \end{pmatrix} \begin{pmatrix} 1-\frac{1}{4} x_1 & -ix_2\\ -\frac{1}{4}x_2 & 1-\frac{1}{4} x_1 \end{pmatrix}^{-1} \begin{pmatrix} \frac{1}{2}\\ 0\end{pmatrix}.$$
Using the construction that we presented in Lemma \ref{lem:block-realization}, we obtain a realization of
$$\br(x) = \begin{pmatrix} 0 & r(x)\\ r^\ast(x) & 0 \end{pmatrix}$$
by
$$\br(x_1,x_2) = \begin{pmatrix} 0 & 0 & 0 & \frac{1}{2}\\ \frac{1}{2} & 0 & 0 & 0\end{pmatrix} \begin{pmatrix} 0 & 0 & 1-\frac{1}{4}x_1 & -ix_2\\ 0 & 0 & -\frac{1}{4}x_2 & 1-\frac{1}{4}x_1\\ 1-\frac{1}{4}x_1 & -\frac{1}{4}x_2 & 0 & 0\\ ix_2 & 1-\frac{1}{4}x_1 & 0 & 0\end{pmatrix}^{-1} \begin{pmatrix} 0 & \frac{1}{2}\\ 0 & 0\\ 0 & 0\\ \frac{1}{2} & 0\end{pmatrix}.$$
According to Theorem \ref{thm:mainrep1}, we introduce now
$$\hpen(x_1,x_2) = \begin{pmatrix} 0 & 0 & 0 & 0 & 0 & \frac{1}{2}\\ 0 & 0 & \frac{1}{2} & 0 & 0 & 0\\ 0 & \frac{1}{2} & 0 & 0 & -1+\frac{1}{4}x_1 & ix_2\\ 0 & 0 & 0 & 0 & \frac{1}{4}x_2 & -1+\frac{1}{4}x_1\\ 0 & 0 & -1+\frac{1}{4}x_1 & \frac{1}{4}x_2 & 0 & 0\\ \frac{1}{2} & 0 & -ix_2 & -1+\frac{1}{4}x_1 & 0 & 0\end{pmatrix}.$$
Again, $\hpen(x_1,x_2)$ decomposes as $\hpen(x_1,x_2) = \hpen_0 + \hpen_1 x_1 + \hpen_2 x_2$, which provides the initial data for our algorithm: if $X_1,X_2$ are freely independent semicircular elements, then the obtained density of the Brown measure of $r(X_1,X_2)$ is shown in Figure \ref{fig:rational-Brown1}, whereas Figure \ref{fig:rational-Brown2} shows the eigenvalues of $r(X_1^{(N)},X_2^{(N)})$ for one realization of independent Gaussian random matrices $X_1^{(N)},X^{(N)}_2$ of size $N=1000$.
\end{example}


\newpage

\centerline{NOT FOR PUBLICATION}

\tableofcontents

\printindex

\end{document}